\documentclass[noams]{compositio}

\renewcommand{\contentsname}{Contents\\{\footnotesize\normalfont(A table
of contents should normally not be included)}}
\usepackage{hyperref,color,amsmath,mathabx,graphicx,float,stackrel}

\newtheorem{thm}{Theorem}
\newtheorem{prop}[thm]{Proposition}
\newtheorem{cor}[thm]{Corollary}
\newtheorem{lem}[thm]{Lemma}
\newtheorem{rem}[thm]{Remark}
\newtheorem{defn}[thm]{Definition}
\newtheorem{fact}[thm]{Fact}
\newtheorem*{fact*}{Fact}

\newcommand{\opcit}{{\it op.cit.\/}\ }
\newcommand{\ie}{{\it i.e.\/}\ }
\def\no{\noindent}

\def\C{{\mathbb C}}
\def\Q{{\mathbb Q}}
\def\R{{\mathbb R}}

\def\Z{{\mathbb Z}}
\def\Tr{{\rm Tr}}
\def\tr{{\rm Tr}}
\def\rep{\vartheta}
\def\vrep{\vartheta_{\rm m}}
\def\qd{^-\!\!\!\!\!d}
\def\Mod{{\rm Mod}}
\def\cH{{\mathcal H}}
\def\F{{\mathbb F}}
\def\fourier{\F}
\def\dirac{\underline \delta}
\def\id{{\mbox{Id}}}
\def\cP{{\mathcal P}}
\def\cS{{\mathcal S}}
\def\cJ{{\mathcal J}}
\def\cL{{\mathcal L}}
\def\cU{{\mathcal U}}
\def\cV{{\mathcal V}}
\def\kf{{\bf K}}
\def\cD{{\mathcal D}}
\def\nf{{\bf N}}
\def\epss{{\omega}}
\def\cQ{{\mathcal Q}}
\def\cT{{\mathcal T}}
\def\bp{{\bf e}}
\def\Spec{{\rm Spec\,}}

\begin{document}

\title{Weil positivity and Trace formula\\ the archimedean place}
\author{Alain Connes}
\email{alain@connes.org}
\address{Coll\`ege de France, IHES and Ohio State University}
\author{Caterina Consani}
\email{kc@math.jhu.edu}
\address{Dept. of Mathematics, Johns Hopkins University}

\dedication{\today}

\begin{abstract}
We provide a potential conceptual reason for the positivity of the Weil functional using the Hilbert space framework of the semi-local trace formula of \cite{Co-zeta}. We explore in great details the simplest case  of the single archimedean place. The root of the positivity is the trace of the scaling action compressed onto the orthogonal complement of the range of the cutoff projections associated to the cutoff in phase space, for $\Lambda=1$. 
  We express the difference between the Weil distribution and the Sonin trace (coming from the above compression of the scaling action) in terms of prolate spheroidal wave functions, and use as a key device the theory of hermitian Toeplitz matrices to control the difference.
   All the  ingredients and tools used above make sense in the general semi-local case, where Weil positivity implies RH.
\end{abstract}

\maketitle


\section*{Introduction}

It was shown  by A. Weil  \cite{Weil}  that the Riemann Hypothesis (RH) is equivalent to the negativity of the right-hand  side of the Riemann-Weil explicit formula \cite{EB} (Appendix~\ref{appendix2}: \eqref{bombieriexplicit}) 
\begin{equation}\label{explicit form}
\tilde f(0)-\sum_{\rho\in Z}\tilde f(\rho)+\tilde f(1)=\sum_v {\mathcal W}_v(f), \qquad  \tilde f(s):=\int_0^\infty f(x)x^{s-1}dx
\end{equation}
 where $Z$ is the multi-set of non-trivial zeros of the Riemann zeta function and for a precise class of complex-valued test functions on the positive half-line, of the form 
 $$
f(x)=\int_0^\infty g(xy)\overline{g(y)}dy,\qquad \tilde g(0)=0,~ \tilde g(1)=0.
$$
In fact, following \cite{yoshida}, it is enough to prove the negativity of the right-hand  side of \eqref{explicit form} for  functions $g$ with compact support in the (locally compact) multiplicative group $\R_+^*=(0,\infty)$.  Furthermore, given any finite set of complex numbers $F\supset \{0,1\}$, $F\cap Z=\emptyset$, using the notation  $g^\sharp(x):=x^{-1}g(x^{-1}), \ \bar g(x)=\overline{g(x)}$, one has (Appendix \ref{apppositivity}, Proposition~\ref{mainprop})  
\begin{equation}\label{weilnegav}
RH \iff \sum_v {\mathcal W}_v(g*\bar g^\sharp)\leq 0, \quad \forall g\in C_c^\infty(\R_+^*)\mid \tilde g(z)=0,\ \forall z\in F.
\end{equation}
The key point being that the  right-hand side of the explicit formula, when evaluated on a test function $f$ with compact support, involves only finitely many primes. Thus, even though the Riemann Hypothesis is about the asymptotic distribution of the primes, this equivalent formulation only involves finitely many  primes at a time.  For finitely many places, inclusive of the archimedean one, the semi-local trace formula of \cite{Co-zeta}  provides  a canonical Hilbert space theoretic set-up where the local distributions ${\mathcal W}_v$ and their sum appear in their precise form, when computing the trace of the scaling action on the semi-local version of the  adele class space.  In order to reflect the unitarity of the scaling action it is convenient to use the automorphism of $C_c^\infty(\R_+^*)$,  $f\mapsto \Delta^{1/2}f$, $\Delta^{1/2}f(x):=x^{1/2}f(x)$ which replaces the involution $f\mapsto \bar f^\sharp$ by the  involution $f\mapsto f^*$ of the convolution $C^*$-algebra and the restriction of the Mellin transform to the critical line by the Fourier transform (see Appendix \ref{appenmellinapp}). We let, for any place $v$,  
$W_v(f):={\mathcal W}_v(\Delta^{-1/2}f)$, $\forall f\in C_c^\infty(\R_+^*)$.\newline 
 In \cite{scalingH} we explained  the relation between the simplest instance of the semi-local operator theoretic framework-- the case of the single archimedean place-- and the so-called Berry-Keating Hamiltonian. We showed that if one removes from the ``white light" the absorption spectrum (as in \cite{Co-zeta}) one obtains  the emission spectrum (as suggested  in \cite{BKe0,BKe}), thus proving that there is no mismatch in the semiclassical approximations.  We also noticed (see \cite{scalingH} Figures 5 and 6) that there is a single ``quantum cell" arising from the absorption spectrum that was not accounted for by the cutoff suggested in  \cite{BKe0,BKe}. The  basic problem posed by the choice of the cutoff of \cite{BKe0,BKe} is its lack of invariance under the scaling action  $\rep$ introduced in \cite{Co-zeta},  as the one-parameter group generated by the scaling Hamiltonian giving the spectral realization of the zeros of the Riemann zeta function. This issue  makes problematic restricting the Hamiltonian. On the other hand, there is straightforward interpretation of this cutoff in terms of   the orthogonal projection $\bf S$ of the Hilbert space $L^2(\R)_{\rm ev}$ of square integrable even functions on the subspace of functions, which, together with their Fourier transform, vanish identically in the interval $[-1,1]$. This subspace is the well-known infinite dimensional Sonin's space  whose discovery  goes back to the work of N. Y. Sonin in the XIX-th century \cite{Sonin}. Even though the scaling action $\rep$ does not restrict to this subspace, one can associate to a test function $f\in C_c^\infty(\R_+^*)$ the trace $\Tr(\rep(f)\, {\bf S})$, and one sees that this functional is positive definite by construction, since when evaluated on $f=g*g^*$ it is the trace $\Tr(\rep(g)\, {\bf S}\,\rep(g)^*)$  of a positive operator. 
 Moreover we prove that  this positive functional differs from the opposite of the archimedean distribution $W_\R$  by an infinitesimal in the sense of  quantized calculus (\ie the associated operator  is compact).  It is thus natural to implement this fundamental positivity in the  semi-local framework  applied to the finite set of places $\{\infty, 2,3,\ldots ,p\}$, and provide a conceptual reason for Weil's negativity for  functions $f$ fulfilling the support condition Support$(f)\subset (p^{-1},p)$. In this paper we  consider  the simplest instance of this strategy, namely when the support of the test function is contained in the  interval $(1/2, 2) \subset \R^*_+$ so that  rational primes are not involved (see \eqref{bombieriexplicit1} in Appendix \ref{appendix2}). In this case, the geometric side of the explicit formula  (in contrast to the spectral side involving the zeros of  the Riemann zeta function) is given by an elusive distribution $W_\R$  which coincides, outside  $1\in \R_+^*$, with a locally rational, positive function that  tends to $+\infty$ as the variable tends to $1$ (see formula \eqref{bombieriexplicit2} of Appendix \ref{appendix2}). The distribution $W_\R$ is then defined as a principal value. Weil's  inequality in this context is the following statement (note the change of sign) 
  
 \textit{For any smooth,  positive definite function $f$ with support 	in the  interval $(1/2,2)$ and whose Fourier transform vanishes at $\pm \frac i2$ one has: $W_\infty(f)\geq 0$ where $W_\infty:=-W_\R$.}\vspace{.02in}
 
  This result was proved in \cite{yoshida} by reducing it to an explicit computation. It is equivalent to the positivity of the quadratic form  \begin{equation}\label{mainquadform}
 	 QW(g):=W_\infty(g*g^*)
 \end{equation}
 defined on the vector space $V$ of smooth functions with support in the interval $[2^{-1/2},2^{1/2}]$ and whose Fourier transform  vanishes at $ \frac i2$.  The main result of the present paper is the following strengthening of the above positivity which provides an operator theoretic conceptual reason for Weil's positivity, rooted in the compression of the scaling action $\rep$ (of $\R_+^*$ in the Hilbert space $L^2(\R)_{\rm ev}$ of square integrable even functions) on the Sonin's space.

 \begin{thm}\label{mainthmintro} Let $g\in C_c^\infty(\R_+^*)$  have  support in the interval $[2^{-1/2},2^{1/2}]$ and Fourier transform  vanishing at $\frac i2$ and $0$. Then  one has 
\begin{equation}\label{maininequintronew}
	W_\infty(g*g^*)\geq \Tr(\rep(g)\, {\bf S}\,\rep(g)^*).
	\end{equation}	
\end{thm}
Equivalence \eqref{weilnegav} shows that, since $\zeta(\frac 12+i\,s)\neq 0$ for $s=0, \frac i2$, the imposed vanishing conditions do not compromise our strategy towards RH.
In Theorem \ref{mainthmfine} we show, more precisely, that there exists a finite constant $c$ (with $13<c<17$)   such that, for any $g\in C_c^\infty([2^{-1/2},2^{1/2}])$ whose Fourier transform vanishes at $\frac i2$ ($\widehat g(\frac i 2)=0$) one has
\begin{equation}\label{maininequintronew1}
	W_\infty(g*g^*)\geq \Tr(\rep(g)\, {\bf S}\,\rep(g)^*) - c \,  \vert \widehat g(0)\vert^2 .
	\end{equation}
	Since the evaluation of the Fourier transform at $0$ defines a character of the convolution algebra, 
it follows that $W_\infty(f)\geq 0$ for any smooth positive definite function $f$ with support 	in the  interval $(1/2,2)\subset \R_+^*$, whose Fourier transform vanishes at $\pm \frac i2$ and at $0$.
 The meaning of the vanishing condition  at $\pm \frac i2$ is well understood.  In the function field case, it  corresponds to  focus on the key contribution  of $H^1$ in Lefschetz's trace formula. Inequality \eqref{maininequintronew} together with \eqref{explicit form} imply
\begin{cor}
Let $g\in C_c^\infty([2^{-1/2},2^{1/2}])$ with $\widehat g(\frac i 2)=0$ and let $Z=\frac 12 +i S$ be the multi-set of  non-trivial zeros of the Riemann zeta function. Then  the following inequality holds
\begin{equation}\label{maininequintro1}
c \,  \vert \widehat g(0)\vert^2+	\sum_{s \in S} \widehat g(s) \overline{ \widehat g(\bar s)}
\geq \Tr(\rep(g)\, {\bf S}\,\rep(g)^*). 
	\end{equation}	
\end{cor}
 This shows that Sonin's trace requires, besides the zeros of the Riemann zeta function, an additional  contribution (irrelevant for RH) on the critical line, in between the two first zeros $\sim \frac 12 \pm 14.1347\, i$, as if one would multiply $\zeta(z) $ by $(z-\frac 12)^{17}$.

 This paper is motivated by the desire to understand the link between the analytic  Hilbert space operator theoretic   strategy first proposed in \cite{Co-zeta}, and the geometric approach pursued in the joint work of the two authors \cite{Crh,CCsurvey}. 
The latter unveiled a novel geometric landscape   still in development for an intersection theory of divisors (on the square of the Scaling Site), thus  not yet in shape to handle  the delicate principal values involved in the Riemann-Weil explicit formula. 
The first contribution of this paper is to make explicit  the relation between the two  approaches,  thus overcoming the above problem. 
The connection between the operator theoretic and the geometric viewpoints  is effected by the Schwartz kernels associated to operators. By implementing the additive structure of the adeles,  one sees that the Schwartz kernel of the scaling operator corresponds geometrically to the divisor of the Frobenius correspondence. \vspace{.03in}


\no\textbf{Fact} \emph{The  additive structure of the adeles  allows one to write the Schwartz kernel $k(x,y)$ of the scaling action $f(x)\mapsto f(\lambda x)$, $\lambda\in\R^*_+$, in the form 
$$
k(x,y)=\dirac(\lambda x-y).
$$
}

\no Indeed, one has
 $$
 \int k(x,y)f(y)dy=\int \dirac(\lambda x-y) f(y)dy=f(\lambda x).
 $$
The  geometric interpretation of the  term   
$$
\frac{1}{\vert 1-\lambda\vert}=\int \dirac(\lambda x-x)dx=\int k(x,x)dx
$$
as the trace of the scaling operator remains formal  until one handles  the singularity at $\lambda =1$. 
This is what we achieve in Section \ref{sectdirectcomp}. The main   idea      developed in this section (through the use  of Schwartz kernels) is to translate geometrically the quantized calculus. This means that, rather than viewing its basic constituent: the operator $H$ of square $1$ (see \eqref{HilbTrsf} in Appendix \ref{appquantized}) as the Hilbert transform, we   pass in Fourier and  interpret $H$ as a multiplication operator. Then the quantized differential, namely  the commutator with $H$, acquires an explicit geometric meaning. In this way  we  obtain  a direct geometric proof of a key equality in the local trace formula of \cite{Co-zeta} as revisited in \cite{CMbook}, namely 
 \begin{equation}\label{intersection}
 	 W_\infty (f)=-\frac 12 \tr(\widehat f  u^*\,\qd u)
 \end{equation}
 where  $\widehat f$ is the Fourier transform  of a test function, $u= u_\infty$ is the unitary classically  associated to  the Fourier transform  composed with the inversion (\!\cite{tate} and Appendix \ref{appendixsigns}) and $\,\qd u$ is the quantized differential of $u$ (Appendix \ref{appquantized}). The geometric counterpart of the multiplication by $\widehat f$ is the  convolution operator $\vrep(f)$. We compute the Schwartz kernel of the geometric counterpart  $(u^*\,\qd u)^g$ of the operator $u^*\,\qd u$.   It is well known that the trace of an operator with Schwartz kernel $k(x,y)$ is given by the integral of the diagonal values $k(x,x)$ and it corresponds geometrically to the intersection with the diagonal. Moreover, the trace of a product of two operators is given by an integral in the full square. Then the geometric description of the quantized differential  splits   the right hand side of \eqref{intersection} as the sum of two contributions corresponding to the two ``squares" $\Delta$ and $\Sigma$ inside the first quadrant: see  Figure \ref{littlesqu}. 

 In Section \ref{sectlittlesq}
we prove the important fact that the obstruction to get   Weil's positivity in the local archimedean case is due to the specific contribution of the small square $\Delta$. To isolate  this discrepancy we introduce the function ``trace-remainder"  
\begin{equation}\label{chiremintro}
\delta(\rho):=\tr\left(\left(\vrep(\rho^{-1})-P \vrep(\rho^{-1})P\right)\frac 12\, (u^*_\infty\,\qd u_\infty)^g\right)
\end{equation}	
where $\vrep$ is the scaling action, and $P$ is the orthogonal  projection given by the multiplication operator by the characteristic function of $[1,\infty)$. The compression $ P T P$ of an operator $T$ associated to a Schwartz kernel $k(x,y)$    reduces  it   to the big square $\Sigma$. The cyclicity of the  trace then shows that \eqref{chiremintro} measures the difference between the Schwartz kernel of $(u^*_\infty\,\qd u_\infty)^g$ and its compression.  We prove, following an idea of \cite{scalingH}, the positivity  of the following functional without any restriction on the  support of the test functions $f\in C_c^\infty(\R_+^*)$ 
\begin{equation}\label{sch19intro}
L(f)=  D(f)+ W_\infty(f), \qquad  \ D(f):= \int f(\rho^{-1})\delta(\rho)d^*\rho.
\end{equation}
The proof  (Proposition \ref{proplittlesq} and Corollary \ref{corlittlesq}) is abstract and conceptual  and rests on Hilbert space operators. 
	This result is also checked numerically using  the equivalence between positivity in the  convolution algebra $C_c^\infty(\R_+^*)$ and pointwise positivity after Fourier transform. The Fourier transform of the distribution $L$ in \eqref{sch19intro} is  the function $\hat \delta(t)+2\theta'(t)$, where $\hat \delta(t):=\int_{\R_+^*}\delta(\rho)\rho^{-it}d^*\rho $ and $\theta'(t)$ is the derivative of the Riemann-Siegel angular function. As shown in Figures \ref{poscheck0} and \ref{poscheck}  the function $\hat \delta(t)+2\theta'(t)$  is non-negative.  The two graphs  displayed in Figures \ref{poscheck2}  and \ref{poscheck3} are very striking: they  show that the  term $\hat \delta(t)$ compensates almost exactly the part of the graph of $\theta'(t)$  where this function is negative. This derivative  is an even function which tends to $+\infty$ as the variable tends to $\pm \infty$. Such behavior accounts for the singularity of the distribution $W_\infty$ at $\rho=1$ and the  use of a principal value in its definition. The Fourier transform $\hat \delta(t)$  is also even but tends to $0$ as $t\to\pm \infty$. In fact the function $\delta(\rho)$ can be written  explicitly for $\rho\geq 1$ as follows 
	\begin{equation}\label{sch18intro}
\delta(\rho)=2\rho^{\frac 12}\left( \frac{\text{Si}(2  \pi (1+\rho))}{2 \pi (1+\rho)}+\frac{\text{Si}(2  \pi (\rho-1))}{2 \pi (\rho-1)}\right)
\end{equation}
where $\text{Si}$ is the Sine Integral function ( \ie the primitive of $\frac{\sin(x)}{x}$ vanishing at $0$). Moreover one also has $\delta(\rho^{-1})=\delta(\rho)$. The behavior of  $\delta(\rho)$  at $\rho=1$  plays a decisive  role in what follows. The graph (see Figure \ref{deltarho})  displays the jump in  the derivative $\delta'(\rho)$ from the value $-1$ as  $\rho\to 1^-$ to the value $1$ as  $\rho\to 1^+$. The reason why this behavior plays a decisive role is discussed in Section \ref{sectsupport}  where we prove  that imposing on the test functions the  vanishing conditions
\begin{equation}\label{vanishingintro}
\widehat f (\pm \frac i2)=\int f(\rho)\rho^{\pm \frac 12}d^*\rho=0
\end{equation}
amounts, without changing the support condition for test  functions, to replace  $\delta(\rho)$ by $Q\delta(\rho)$, where\footnote{The role of the operator $Q$ is to multiply the Fourier transforms of the test functions by $z^2+\frac 14$ so that they then vanish at $z=\pm \frac i2$, thus imposing the boundary conditions while keeping positivity and support restrictions} $Q=-(\rho\partial_\rho)^2+\frac 14$. Thus the jump in the derivative $\delta'(\rho)$ at $\rho=1$ generates  $-2\dirac_1$  where $\dirac_1$ is the Dirac distribution at $\rho=1$. This fact, together with the smoothness of  $\delta(\rho)$ outside $1$ suffices to show, on abstract grounds, that the 
functional $D\circ Q(f):=\int f(\rho^{-1})Q\delta(\rho) d^*\rho$ is essentially negative (in the sense of convolution) on the subspace $C_c^\infty(I)\subset C_c^\infty(\R_+^*)$ of functions with   support  contained in a fixed compact interval $I$. The  essential negativity of $D\circ Q$ follows from the decomposition  $-2 \,\id+K$  (Theorem \ref{thmqkey1}) of the  operator in the Hilbert space $L^2(\sqrt I,d^*\rho)$ associated to $D\circ Q(g*g^*)$  for $g$ with support in $\sqrt I$, where $K$ is a compact operator of Hilbert-Schmidt class.  Hence, by imposing finitely many linear conditions (\ie by restricting to a subspace of $L^2(\sqrt I,d^*\rho)$ of finite codimension), one obtains a negative quadratic form. Thus, by applying the same conditions one obtains the positivity of the Weil distribution $W_\infty=L-D$, using the already proven positivity of $L$  and \eqref{sch19intro}. 
As a preliminary test we prove, using a simple estimate (see Corollary~\ref{qeasy} and  Remark \ref{improve})  that for  small enough intervals $I$ the positivity holds.\vspace{.2pt}

Section \ref{sectlsqmove} describes the next main step. We perform a  construction whose effect is that to ``move" the small square $\Delta$ inside  $\Sigma$, in order to  use the (infinite) reservoir of positivity  attached to $\Sigma$. For this part we apply the technique of pairs of projections and of prolate functions as introduced in \cite{Co-zeta}. The key ingredients  are the orthogonal projections $\cP_\Lambda$ and $\widehat \cP_\Lambda$, associated to the cutoff parameter $\Lambda$  (\!\cite{CMbook} Chapter 2, \S 3.3) for  $\Lambda=1$. 
In Lemma \ref{onequantum} we show that for $\rho\geq 1$ one has
\begin{equation}\label{chirem1intro}
\delta(\rho)=\tr\left(\vrep(\rho^{-1})\widehat \cP_1\,\cP_1\right)
\end{equation}
where the right hand side is expressed as a sum of coefficients\footnote{in the sense of representation theory} of the scaling action on  prolate spheroidal functions.  The equality \eqref{chirem1intro} provides in particular the link between the geometric  square $\Delta$ and the single ``quantum cell" that was protruding  in Figure 5 of \cite{scalingH} in relation to the absorption-emission discussion. The process of moving $\Delta$ inside $\Sigma$  is performed by implementing  the decomposition into irreducible components of the unitary representation of the infinite dihedral group $\Z \ltimes \Z/2\Z$ associated to the pair of projections $\widehat \cP_1,\,\cP_1$.  In the non-trivial part given by the range of $\cP_1\vee \widehat \cP_1$, these irreducible components  are two dimensional and labeled by the eigenvalues of the prolate differential operator (see \cite{Co-zeta} and \cite{CMbook}).  The orthogonal  space to the range of $\cP_1\vee \widehat \cP_1$  is the infinite dimensional Sonin's space  where  the representation is trivial. The outcome of  moving  $\Delta$  inside $\Sigma$   is to reduce the negative contribution of $D$ (see \eqref{sch19intro}) using the positivity of $\tr(\rep(h){\bf S})$, where 
$\bf S$ is the  projection onto the  orthogonal complement of  the range of $\cP_1\vee \widehat \cP_1$. The main outcome is the following  (Theorem \ref{devil}) 

\begin{thm}\label{devilintro}  The functional $\tr(\rep(f){\bf S})$ is positive and one has
\begin{equation}\label{sonine1.0}
\tr(\rep(f){\bf S})=W_\infty(f)+\int f(\rho)\epsilon(\rho) d^*\rho,\qquad \forall f \in C_c^\infty(\R_+^*),
\end{equation}	
where $\epsilon(\rho^{-1})=\epsilon(\rho)$,
 $\rho\in \R_+^*$,  and with $\epsilon(\rho)$ given, for $\rho \geq 1$,  by 
\begin{equation}\label{sonine0intro}
\epsilon(\rho)=\sum \frac{\lambda(n)}{\sqrt{1-\lambda(n)^2}}\langle \xi_n\mid  \rep(\rho^{-1})\zeta_n\rangle.
\end{equation}
\end{thm}

We refer to Section \ref{sectlsqmove} for the notations and the precise definition of the vectors $\zeta_n, \xi_n$ in terms of prolate functions. The scaling action $\vrep$ is now represented in the Hilbert space $L^2(\R)_{\rm ev}$ and denoted by $\rep$.
In Section \ref{sectcompactK} we analyze the functional $E\circ Q(h)=\int h(\rho^{-1})Q\epsilon(\rho) d^*\rho$. The function $\epsilon(\rho)$  behaves, at $\rho=1$,  similarly to the function $\delta(\rho)$. One has $\epsilon(\rho^{-1})=\epsilon(\rho)$ and the value of the derivative $\epsilon'(1^+)$ is  approximately $22.9965$.  
It follows,   likewise for $\delta(\rho)$, that the operator, in the Hilbert space $L^2(\sqrt I,d^*\rho)$, associated to the functional $E\circ Q$,   decomposes as $-2\epsilon'(1^+) \,(\id-K)$, where $K$ is a compact operator associated to a Hilbert-Schmidt kernel. 
The  functional $E\circ Q$ is thus  essentially negative on  the space of functions with   support  contained in a fixed compact interval $I$.
The next task  is  to compute  the spectrum of the compact operator $K$ when $I=[\frac 12,2]$. Section \ref{sectspectrum} is entirely devoted to this goal. 
The strategy developed in this part is summarized by the  following steps
\begin{enumerate}
\item[1.] Discretize the group $\R_+^*$ by approximating it with $q^\Z$, where $q\to 1^+$. \item[2.] Identify the approximating operator $K_q$ used in 1. as a Toeplitz matrix and compute numerically its eigenvalues. 
\item[3.] Apply the general theory of Toeplitz matrices to rewrite $K_q$ in canonical form.	
\item[4.] Guess a formula for the operator $K$ independently of $q$, by comparing its approximate behavior  for different values of $q\to 1^+$. 
\item[5.] Construct a finite rank operator $T$ that provides a good approximation of $K$ on $L^2(\sqrt I,d^*\rho)$ for $I=[\frac 12,2]$.
\item[6.] Compute the spectrum of $T$ and identify the single eigenvector which, after conditioning, makes the  functional $E\circ Q$ negative.
\end{enumerate}
The first step is in line with the results of \cite{CC1, CC2}  providing  a Hasse-Weil formula, in   the limit of $\F_q$ as $q\to 1^+$,  for the complete Riemann zeta function. 
The key observation in  2. is that, for $q\sim 1$, the operator $K_q$ has  only one eigenvalue $>1$.
The method used in 2. and 3. is closely related to the technique applied in \cite{toulouse} in the context of  RH for curves over finite fields.
Step 4.  provides  an explicit approximation of the function $Q\epsilon(\rho)$ on the relevant interval $I=[\frac 12,2]$ which in turns gives, in 5., an approximation of $K$ on $L^2(\sqrt I,d^*\rho)$ by a finite rank operator. 
Step 6. keeps track of the eigenvector associated to the only eigenvalue of $K$ which is $>1$. Finally, this process  allows one, by conditioning on the orthogonality to this vector (or a sufficiently nearby one), to lower the value of the maximal eigenvalue of $K$ to $<1$,  thus getting the positivity of $\id -K$  that in turns yields  Weil's positivity and the stronger inequality \eqref{maininequintronew1}.

\numberwithin{thm}{section}
 
 \section{Geometric interpretation of $\tr(\widehat f \, u^*\,\qd u) $}\label{sectdirectcomp}

In this section we give a geometric proof of the local trace formula at the archimedean place  \cite{Co-zeta,CMbook}, namely the equality for the local term ($ W_\R=-W_\infty$)
 \begin{equation}\label{intersectiongeom}
 	 W_\R (f)=\frac 12 \tr(\widehat f \, u^*\,\qd u)
 \end{equation}
 where $\widehat f$ is the Fourier transform of a  test function $f\in C_c^\infty(\R_+^*)$,  $u= u_\infty$ is the unitary classically associated with the Fourier transform composed with the inversion (see \cite{tate} and Appendix \ref{appendixsigns}) and $\,\qd u$ is the quantized differential of   $u$.   
 
 The main  effect of this geometric proof is  to decompose the right hand side of  \eqref{intersectiongeom} in two terms corresponding   to the contributions of the two squares $\Delta$ and $\Sigma$ of Figure \ref{littlesqu}.  
 We normalize the inner product in $L^2(\R)_{\rm ev}$ as follows
 \begin{equation}\label{innerltwoeven}
 	\langle\eta\mid \xi\rangle:=\frac 12\int_\R \overline {\eta(x)}\xi(x)dx=\int_0^\infty \overline {\eta(x)}\xi(x)dx.
 \end{equation}
  We first  relate the trace computations performed in $L^2(\R)_{\rm ev}$ with those done in the isomorphic Hilbert space $L^2(\R_+^*,d^*\lambda)$ using the unitary isomorphism 
\begin{equation}\label{isow}
 	 w:L^2(\R)_{\rm ev}\to L^2(\R_+^*,d^*\lambda), \qquad (w\xi)(\lambda):=\lambda^{\frac 12}\xi(\lambda).
 \end{equation}
 Given an operator $T$ in $L^2(\R)_{\rm ev}$, we shall denote the corresponding operator in $L^2(\R_+^*)$ by 
 \begin{equation}\label{nota}
     T^w:=wTw^{-1}.
 \end{equation}

 The following lemma provides the description (of the multiplicative version) of the Schwartz kernel of an operator in $L^2(\R)_{\rm ev}$  after implementing the isomorphism $w$
 

\begin{lem}\label{add2mult} Let $T$ be an operator in $L^2(\R)_{\rm ev}$ of the form 
\begin{equation}\label{sch0}
T\xi(x)=\int_{y\geq 0}k(x,y)\xi(y)dy,\qquad \forall x\geq 0, \quad\forall\xi \in L^2(\R)_{\rm ev}.
\end{equation}
Then the Schwartz kernel of $T^w=wTw^{-1}$ in $L^2(\R_+^*)$ takes the following form 
\begin{equation}\label{sch1}
k^w(\lambda,\mu)=\lambda^{\frac 12}\mu^{\frac 12}k(\lambda,\mu).
\end{equation}
Moreover, the trace of $T$, namely $\int_{x\geq 0}k(x,x)dx$, is equal to $\int k^w(\lambda,\lambda)d^*\lambda$. 	
\end{lem}
\begin{proof} Since the two Hilbert spaces are isomorphic,  one can start with the expression of $wTw^{-1}$ in $L^2(\R_+^*)$ as 
$$
wTw^{-1}(\eta)(\lambda):=\int_{\R_+^*} k^w(\lambda,\mu)\eta(\mu)d^*\mu
$$
Then, with $\eta=w\xi$ one has $\eta(\mu)d^*\mu=\xi(\mu)\mu^{-1/2}d\mu$, for $x=\lambda > 0$ and $y=\mu$ one  gets
$$
T\xi(x)=x^{-1/2}\int_{y> 0}k^w(x,y)y^{-1/2}\xi(y)dy,\qquad \forall x> 0, \ \xi \in L^2(\R)_{\rm ev}.
$$
This shows that $k(x,y)=x^{-1/2}k^w(x,y)y^{-1/2}$  and \eqref{sch1} holds.\end{proof}

\begin{rem}\label{sch2} In general an operator $T$ in $L^2(\R)$ restricts to $L^2(\R)_{\rm ev}$ if it commutes with the symmetry $s\xi(x)=\xi(-x)$. For $T$ with Schwartz kernel $t(x,y)$, \ie $T\xi(x)=\int_{\R}t(x,y)\xi(y)dy$, this means that $t(-x,-y)=t(x,y)$. The restriction $T_{\rm ev}$ of $T$ to $L^2(\R)_{\rm ev}$ is of the form \eqref{sch0} with $k(x,y)=t(x,y)+t(x,-y)$. Its trace is 
$$
\Tr(T_{\rm ev})= \int_{x\geq 0}k(x,x)dx=\int_{x\geq 0}(t(x,x)+t(x,-x))dx=\frac 12 \int_{\R}(t(x,x)+t(x,-x))dx
$$
which is the trace of the composition of $T$ with the projection $\frac 12 (1+s)$. 
\end{rem}
We define  the duality $\langle \R^*_+, \R \rangle$ by the
bi-character of $\R^*_+\times \R$
\begin{equation}\label{FwIPhi0}
\mu(v,s)=v^{-is} \,,\ \ \forall v\in \R^*_+, s\in \R.
\end{equation}
In the following part we  make systematic use of the Fourier transform 
\begin{equation}\label{PhiFourier}
\fourier_\mu   :
L^2(\R^*_+)\to L^2(\R),\qquad \fourier_\mu   (f)(s):=\int_0^{\infty} f(v) v^{-is}d^*v\,.
\end{equation}
Given an operator $T$ in $L^2(\R)$, we denote the corresponding\footnote{the upper index $g$ stands for ``geometric"} operator in $L^2(\R_+^*)$ by 
\begin{equation}\label{notamu}
T^g:=\fourier_\mu   ^{-1}  \circ T
\circ   \fourier_\mu.
\end{equation}
Next, we use the quantized calculus technique as in \cite{CMbook} (Chapter 2, \S 5.1). 

We first recall   Lemma 2.20 of \opcit which uses the unitary inversion operator $I$  that performs the change of variable $ \lambda\to \lambda^{-1}$, in $L^2(\R^*_+)$. We let $\fourier_{e_\R}$ denote the Fourier transform  with respect
to the basic character $e_\R(x)=e^{-2\pi ix}$: it defines the unitary in $L^2(\R)_{\rm ev}$
\begin{equation}\label{cFalpha}
\fourier_{e_\R}(\xi)(y):=\int_{-\infty}^\infty\,\xi(x)\,e^{-2\pi ixy}\,dy.
\end{equation}
On implementing the isomorphism $w$ we obtain
\begin{lem}\label{technfourier}
One has: $\fourier_{e_\R}^w=I\circ u_\infty^g$. Equivalently, the following equality holds in  $L^2(\R)_{\rm ev}$
\begin{equation}\label{FwIPhi}
\fourier_{e_\R}= w^{-1}\circ I \circ  \fourier_\mu   ^{-1}  \circ u_\infty
\circ   \fourier_\mu   \circ  w,
\end{equation}
where $u_\infty$ is the multiplication operator by the function
\begin{equation}\label{RieSiegel}
u_\infty(s):= e^{2\, i\, \theta(s)},
\end{equation}
and $\theta(s)$ is the Riemann-Siegel\index{Riemann-Siegel} angular function recalled in equation \eqref{riesie} of Appendix~\ref{appendix2}. 
\end{lem}
From the proof of the lemma given in \opcit, we extract and include the following part  since it plays  a key role in this paper. The Fourier transform $\fourier_{e_\R}$ preserves globally $L^2(\R)_{\rm ev}$.
For $\xi \in L^2(\R^*_+)$, one has
\begin{align*}
(w\, \circ \fourier_{e_\R}\circ w^{-1} )(\xi)(v)&= v^{1/2}\,
\int_\R|x|^{-1/2}\,\xi(|x|)\,e^{-2\pi ixv}\,dx\\
&= v^{1/2}\, \int_{\R^*_+} u^{1/2}\,\xi(u)\,(e^{2\pi iuv}+e^{-2\pi
iuv})\,d^*u \,.
\end{align*}
Using the inversion $I$ in $L^2(\R^*_+)$, this gives
\begin{align*}
(I\circ w\, \circ \fourier_{e_\R}\circ w^{-1} )(\xi)(\lambda)&= (w\,
\circ \fourier_{e_\R}\circ w^{-1} )(\xi)(\lambda^{-1})\\
&=\lambda^{-1/2}  \int_{\R^*_+} (e^{2i\pi \mu \lambda^{-1}}  +
e^{-2i\pi \mu \lambda^{-1}}) \,\mu^{1/2} \xi(\mu)d^*\mu.
\end{align*}
The following lemma gives the description of the Schwartz kernels in $L^2(\R^*_+)$  of several relevant operators
\begin{lem}\label{sch3} $(i)$~The Schwartz kernel in $L^2(\R^*_+)$ of the convolution  operator $u_\infty^g$ is given by
\begin{equation}\label{sch4}
k^u(\lambda,\mu)=2\lambda^{-\frac 12}\mu^{\frac 12}\cos(2\pi \mu/\lambda)	
\end{equation}
$(ii)$~The Schwartz kernel in $L^2(\R^*_+)$ of the   operator  $(\frac 12\, \qd u_\infty)^g=[P,u_\infty^g]$ is 
\begin{equation}\label{sch5}
(P(\lambda)-P(\mu))k^u(\lambda,\mu), \quad \text{with}\quad P(v)=\begin{cases}1 & \text{if $v\geq 1$,} \\ 0 & \text{if $v<1$.}\end{cases}
\end{equation}
$(iii)$~The Schwartz kernel in $L^2(\R^*_+)$ of the   operator  $(\frac 12\, u^*_\infty\,\qd u_\infty)^g$ is 
\begin{equation}\label{sch6.0}
\ell(\nu,\mu)=\int k^u(\lambda,\nu) (P(\lambda)-P(\mu))k^u(\lambda,\mu)d^*\lambda.\end{equation}
\end{lem}
\begin{proof} $(i)$~Equality \eqref{sch4} follows  from the fact that the convolution  operator $\fourier_\mu   ^{-1}  \circ u_\infty
\circ   \fourier_\mu$ associated to $u_\infty$ can be written as  $I\circ w\, \circ \fourier_{e_\R}\circ w^{-1}$, by implementing \eqref{FwIPhi}.\newline
$(ii)$~ follows from the equality $H^g=2P-1$, where $H$ is the operator defining the quantized differential \ie the Hilbert transform. Moreover,  recall that for $T$ with Schwartz kernel $k$, the kernel of the commutator $[P,T]$ is  $(P(\lambda)-P(\mu))k(\lambda,\mu)$.\newline 
$(iii)$~One derives \eqref{sch6.0}  by recalling that the Schwartz kernel of the adjoint of an operator with kernel $k(\lambda,\mu)$ is $\overline{k(\mu,\lambda)}$, and that the kernel of the composite $T_1T_2$ is $\int k_1(\nu,\lambda)k_2(\lambda,\mu)d^*\lambda $. \end{proof} 
Next, we  change variables in \eqref{sch6.0} by letting $y=1/\lambda$. The geometric meaning of this change of variables is reflected in the appearance of the two squares $\Delta$ and $\Sigma$ as in Figure~\ref{littlesqu}.  We obtain
\begin{align*}
\ell(\nu,\mu)&=\int k^u(\lambda,\nu) (P(\lambda)-P(\mu))k^u(\lambda,\mu)d^*\lambda=\\ 
&=4\mu^{\frac 12}\nu^{\frac 12}\int \lambda^{-\frac 12}\cos(2\pi \nu/\lambda)\lambda^{-\frac 12}\cos(2\pi \mu/\lambda)(P(\lambda)-P(\mu))d^*\lambda=\\
&=4\mu^{\frac 12}\nu^{\frac 12}\int \cos(2\pi \nu y)\cos(2\pi \mu y)(P(1/y)-P(\mu))dy.
\end{align*}

Let $\vrep$ be the regular representation of $\R^*_+$ on
$L^2(\R^*_+)$
\begin{equation}\label{vrepdefnfirst}
(\vrep(\lambda)\,\xi)(v):=\,\xi(\lambda^{-1}\,v),\qquad \forall\xi \in
L^2(\R^*_+)\,.
\end{equation}

The following result  provides the description of the distribution $W_\R$ as stated at the beginning of this section
\begin{prop}\label{lemsch7} $(i)$~The Schwartz kernel in $L^2(\R^*_+)$ of the operator  $\vrep(\rho^{-1})(\frac 12\, u^*_\infty\,\qd u_\infty)^g$ is 
\begin{equation}\label{sch6}
\ell_\rho(\nu,\mu)=4\rho^{\frac 12}\mu^{\frac 12}\nu^{\frac 12}\int \cos(2\pi\rho \nu y)\cos(2\pi \mu y)(P(1/y)-P(\mu))dy\end{equation}
$(ii)$~The trace of the operator  $\vrep(\rho^{-1})(\frac 12\, u^*_\infty\,\qd u_\infty)^g$ is formally given by
\begin{equation}\label{sch6.1}
\tau(\rho)=4\rho^{\frac 12}\int_{x>0\atop y>0} \cos(2\pi\rho x y)\cos(2\pi x y)(P(1/y)-P(x))dydx\end{equation}	
$(iii)$~The equality \eqref{sch6.1} defines a distribution $\tau$  equal to $W_\R=-W_\infty$.\newline
$(iv)$~For $f\in C_c^\infty(\R^*_+)$ the operator  $\vrep(f)\frac 12 (u^*_\infty\,\qd u_\infty)^g$ is of trace class and its trace is given by 
\begin{equation}\label{traceequa}
\Tr\left(\vrep(f)\frac 12 (u^*_\infty\,\qd u_\infty)^g\right)=\int f(\rho^{-1})\tau(\rho)d^*\rho.
\end{equation}
\end{prop}
\begin{proof} $(i)$~ follows from the equality $\ell_\rho(\nu,\mu)=\ell(\rho\nu,\mu)$. More generally,
for $T$ with Schwartz kernel $k$ in $L^2(\R^*_+)$, the Schwartz kernel of $(\vrep(\rho^{-1})\circ T)$ in $L^2(\R^*_+)$ is $  k(\rho a,b)$ since
$$
((\vrep(\rho^{-1})\circ T)\xi)(a)=(T\xi)(\rho a)=\int k(\rho a,b)\xi(b)d^*b.
$$ 
$(ii)$~follows from $(i)$, using the definition of $\tau(\rho)=\int\ell_\rho(\mu,\mu)d^*\mu$ and $\mu d^*\mu=d\mu$. 

$(iii)$~For the computation of the integral \eqref{sch6.1} we note that, in its domain of integration \ie  the positive quadrant, the term $P(1/y)-P(x)$ vanishes except in the  subset $\Delta\cup \Sigma$ where
$$
\Delta:=\{(x,y)\mid 0\leq x\leq 1, \ 0\leq y \leq 1\}, \qquad \Sigma:=\{(x,y)\mid x >1, \  y > 1\}.
$$
Moreover one has:
\[
P(1/y)-P(x)=\begin{cases}1 &\text{for $(x,y)\in \Delta$},\\ -1 & \text{for $(x,y)\in \Sigma$.}\end{cases}
\]
We thus derive
\begin{equation}\label{sch7}
\tau(\rho)=4\rho^{\frac 12}\left(\int_\Delta \cos(2\pi\rho x y)\cos(2\pi x y)dydx
-\int_\Sigma \cos(2\pi\rho x y)\cos(2\pi x y)dydx\right).
\end{equation}	
Let $\alpha(t)$ be the area of the subset $\{(x,y)\in \Delta\mid xy\leq t\}\subset \Delta$. For $t<1$ one has:  $\alpha(t)=t+\int_t^1\frac tx dx=t-t\log t$. Thus  with $I(\Delta):=\int_\Delta \cos(2\pi\rho x y)\cos(2\pi x y)dydx$ one obtains
\begin{equation}\label{sch8}
I(\Delta)
=\int_0^1 \cos(2\pi\rho t)\cos(2\pi t)d\alpha(t)
=\int_0^1 \cos(2\pi\rho t)\cos(2\pi t)(-\log t)dt.
\end{equation}
 \begin{figure}[H]	\begin{center}
\includegraphics[scale=0.55]{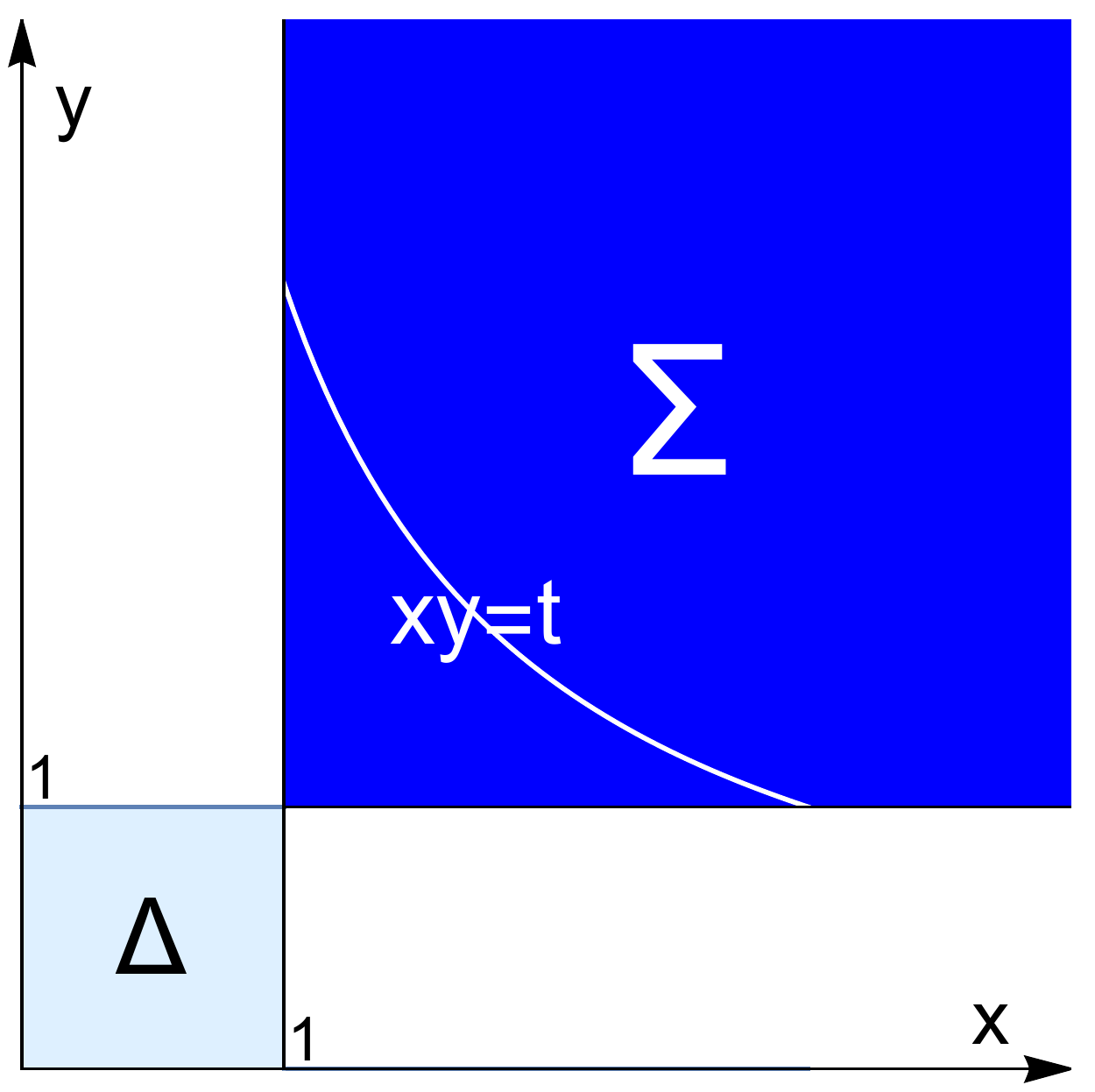}
\end{center}
\caption{The small square $\Delta$ and the big (infinite) square $\Sigma$ \label{littlesqu} }
\end{figure}
Let $\beta(t)$ be the area of the subset $\{ (x,y)\in \Sigma\mid xy\leq t\}\subset \Sigma$. For $t>1$ one has:  $\beta(t)=\int_1^t\frac tx dx-(t-1)=t\log t-(t-1)$. Thus,  with  $I(\Sigma):=\int_\Sigma \cos(2\pi\rho x y)\cos(2\pi x y)dydx$, one obtains
\begin{equation}\label{sch9}
I(\Sigma)=\int_1^\infty \cos(2\pi\rho t)\cos(2\pi t)d\beta(t)
=\int_1^\infty \cos(2\pi\rho t)\cos(2\pi t)\log tdt.
\end{equation}
This gives in turn
\begin{equation}\label{sch10}
\tau(\rho)=4\rho^{\frac 12}\left(I(\Delta)-
I(\Sigma)\right)=4\rho^{\frac 12}\int_0^\infty \cos(2\pi\rho t)\cos(2\pi t)(-\log t)dt
\end{equation}
 that can be rewritten as 
\begin{equation}\label{sch11}
\tau(\rho)=\rho^{\frac 12}\int_\R \left(\exp(2\pi i(1+\rho) t)+\exp(2\pi i(1-\rho) t)\right)(-\log \vert t\vert)dt.
\end{equation}
The Fourier transform of the distribution $-\log \vert t\vert$ is the distribution $W=\fourier_{e_\R}(-\log \vert t\vert)=\frac 12 \frac{1}{\vert x\vert}$ (with Weil's normalization). To  understand the presence of the factor $\frac 12 $, note that one has the equalities 
$$
(\fourier_{e_\R}(\partial_x f))(y)=2\pi i y (\fourier_{e_\R}( f))(y), \qquad
(\fourier_{e_\R}(2\pi i x f))(y)=-\partial_y (\fourier_{e_\R}( f))(y) 
$$
and hence one obtains
$$
(\fourier_{e_\R}(x\partial_x f))(y)=-\partial_y(y\fourier_{e_\R}( f))(y).
$$
Since $-t\partial_t(-\log \vert t\vert)=1$,  the distribution $W$ fulfills the equation $\partial_y(y W)=\dirac_0$,  so that  one obtains: $yW=\frac 12 {\rm sign}(y)$  (since $y W$ is odd) and $W(y)=\frac 12 \frac{1}{\vert y\vert}$. With this definition of the principal value,  \eqref{sch11} gives the Weil distribution  (\!\cite{CMbook} Section 4.2)
\begin{equation}\label{sch12}
\tau(\rho)=\frac {\rho^{\frac 12}}{2}\left(\frac{1}{1+\rho}+\frac{1}{\vert 1-\rho\vert} \right).
\end{equation}
$(iv)$~Let us  show that $\vrep(f)\frac 12 (u^*_\infty\,\qd u_\infty)^g$ is of trace class. It is enough, by using the commutativity of $\vrep(f)$ with the bounded operator $(u^*_\infty)^g$, to show that $\vrep(f)\,(\qd u_\infty)^g$ is of trace class.  Using the Fourier transform $\hat f$ and its associated multiplication operator one has
 $$
 \vrep(f)\,(\qd u_\infty)^g=(\hat f \,\qd u_\infty)^g=\,\left(\qd (\hat f u_\infty)-\,\qd (\hat f) u_\infty\right)^g.
 $$
 Thus it suffices to show that for $f\in \cS(\R_+^*)$, the function $h=\hat f u_\infty$ is in the Schwartz space  $\cS(\R)$ (see Lemma \ref{quantsmooth}). To prove this one needs to control the behavior  at $\pm \infty$ of the function $h$.  We know that it is of rapid decay since $u_\infty$ is of modulus one. Moreover, one has
$$
\partial_s(h)=\partial_s(\hat f)  u_\infty+\hat f \partial_s(  u_\infty)
=\partial_s(\hat f)  u_\infty+2i \hat f u_\infty \partial\theta
$$
where $\theta$ is the Riemann-Siegel angular function (see \eqref{riesie}) and its derivative $\partial\theta(s)$ is $O(\log \vert s\vert)$ when $\vert s\vert \to \infty$.  Thus $\partial_s(h)$ is of rapid decay.  One controls the growth of the higher derivatives of $\theta$ by using  Binet's first formula  
 $$\log(\Gamma(z))=(z-\frac 12)\log z-z+\frac 12 \log(2\pi) + \int_0^\infty  \left(\frac{1}{\exp (t)-1}-\frac{1}{t}+\frac{1}{2}\right)e^{-tz}\frac{dt}{t}
 $$
which, applied at $z=\frac{1}{4} + i \frac{s}{2}$,  gives for the first derivative
$$
\partial_s\log(\Gamma(\frac{1}{4} + i \frac{s}{2}))=\frac{i}{2}  \log \left(\frac{1}{4} + i \frac{s}{2}\right)-\frac{1}{2 s-i}-\frac i2 \int_0^\infty  \left(\frac{1}{\exp (t)-1}-\frac{1}{t}+\frac{1}{2}\right)e^{-t(\frac{1}{4} + i \frac{s}{2})}dt
$$
and shows that all the higher derivatives are bounded.  The higher derivatives of $h$ involve sums of products of $\partial_s^k(\hat f) u_\infty$ by products of derivatives of $\partial\theta$. Since the latter are of tempered growth, one concludes that all the higher derivatives $\partial_s^k(h)$ are of rapid decay so that the function $h$ is in  $\cS(\R)$.  This shows that $\vrep(f)\frac 12 (u^*_\infty\,\qd u_\infty)^g$ is of trace class, its trace is then given by the integral of the diagonal values of its Schwartz kernel which proves \eqref{traceequa}. \end{proof} 
 One checks that \eqref{sch12} is coherent with the equality $W_\R(f)={\mathcal W}_\R(\Delta^{-1/2}f)$, we compare with     \eqref{bombieriexplicit2} of Appendix \ref{appendix2}. One has (assuming for simplicity $f(1)=0$)
$$
W_\R(f)=\int_0^\infty f(\rho^{-1})\tau(\rho)d^*\rho=\int_0^\infty f(\rho)\tau(\rho)d^*\rho
$$ 
For $\rho>1$ one has $\tau(\rho)=\frac {\rho^{\frac 32}}{\rho^2-1}$ by \eqref{sch12} and with $k=\Delta^{-1/2}f$, $k(u)=u^{-1/2}f(u)$,
$$
\int_1^\infty\tau(\rho)f(\rho)d^*\rho=\int_1^\infty\frac {\rho^{\frac 32}}{\rho^2-1}f(\rho)d^*\rho =\int_1^\infty\frac {\rho^{\frac 12}}{\rho^2-1}f(\rho)d\rho =\int_1^\infty\frac {1}{x-x^{-1}}k(x)dx
$$
For $\rho\leq 1$ one has, by \eqref{sch12}, $$\tau(\rho)=\frac {\rho^{\frac 12}}{2}\left(\frac{1}{1+\rho}+\frac{1}{ 1-\rho} \right)=\frac {\rho^{\frac 12}}{1-\rho^2},$$
and thus 
$$
\int_0^1\tau(\rho)f(\rho)d^*\rho=\int_0^1\frac {x^{\frac 12}}{1-x^2}f(x)dx/x=\int_0^1\frac {x}{1-x^2}k(x)dx/x
$$
Moreover,  since $k^\sharp(x):=\frac 1x k(x^{-1})$, one obtains
$$
\int_0^1\frac {x}{1-x^2}k(x)dx/x=\int_1^\infty\frac {1}{x-x^{-1}}k(x^{-1})dx/x=\int_1^\infty\frac {1}{x-x^{-1}}k^\sharp(x)dx
$$
which fits with $W_\R(f)={\mathcal W}_\R(k)$ using  \eqref{bombieriexplicit2} of Appendix \ref{appendix2}. The relation between  the Mellin transform $\tilde k(z):=\int_0^\infty k(u)u^zd^*u $  and the (multiplicative) Fourier transform $\hat f$: 
\begin{equation}\label{fouriermellin}
\tilde k(\frac 12+is)=\int_0^\infty k(u)u^{\frac 12+is}d^*u =\int_0^\infty f(u)u^{is}d^*u=\hat f(-s).
\end{equation}
The negative sign in $-s$  is due to the convention for the multiplicative Fourier transform \eqref{PhiFourier}. We specifically note that this direct computation of the trace $\tr(\hat f \frac 12 u^*\,\qd u) $ does not depend on the above normalization (for the sign) since it is done without the use of the multiplicative Fourier transform. What matters is instead the  minus sign in front of the distribution $-\log \vert t\vert$ in \eqref{sch11}. The convention for the multiplicative Fourier transform intervenes twice both in $u_\infty$ and in the operator $H$ defining the quantized calculus and the signs cancel in the expression for  $u^*\,\qd u$. 

\section{The square $\Delta$ and the  trace-remainder}\label{sectlittlesq}

In this section we show that the obstruction to get  Weil's positivity at  the  archimedean place  is due to the specific contribution of the small square $\Delta$  displayed in Figure \ref{littlesqu}. 
We first single out the contribution of  $\Delta$  by introducing a precise  definition of  ``trace-remainder". We use the same notations of Proposition~\ref{lemsch7} and the definitions of the operators given in  \eqref{nota} and \eqref{notamu}. With  $\vrep$  we denote (see \eqref{vrepdefnfirst}) the regular representation of $\R^*_+$ on
$L^2(\R^*_+)$. As in \eqref{sch5} of  Lemma~\ref{sch3}, the projection $P$ is associated  to the multiplication operator by the characteristic function of $[1,\infty)\subset \R^*_+$.  

\begin{defn}\label{chinese} The trace-remainder is  the function of $\rho\in \R_+^*$ given by
\begin{equation}\label{chirem}
\delta(\rho):=\tr\left(\left(\vrep(\rho^{-1})-P \vrep(\rho^{-1})P\right)\frac 12\, (u^*_\infty\,\qd u_\infty)^g\right).
\end{equation}	
\end{defn}
Next proposition  provides an explicit formula for $\delta(\rho)$. In particular it  shows that, unlike the distribution $\tau(\rho)$ in \eqref{sch6.1} (corresponding to $W_\R$)  that is  is not a function because of the  divergency at $\rho=1$,  $\delta(\rho)$ is a function. Moreover it  fulfills the symmetry $\delta(\rho)=\delta(\rho^{-1})$.  This fact allows one to extend the explicit formula (valid for $\rho\geq 1$) to all of $\R^*_+$. Next equality    \eqref{sch13} will play a crucial role in the proof of  Corollary \ref{corlittlesq}.
\begin{prop} \label{proplittlesq}
$(i)$~For $\rho\geq 1$ one has
\begin{equation}\label{rhoform}
\delta(\rho)=4\rho^{\frac 12}\int_\Delta \cos(2\pi\rho x y)\cos(2\pi x y)dydx.
\end{equation}
$(ii)$~For all $\rho \in \R_+^*$, one has: $\delta(\rho)=\delta(\rho^{-1})$.\newline
$(iii)$~For $f\in C_c(\R_+^*)$, and with $\hat{ P}=(\fourier_{e_\R}^w)^{-1}P\fourier_{e_\R}^w $, one has 
	\begin{equation}\label{sch13}
\tr\left(\vrep(f)P\widehat PP\right)= \int_0^\infty f(\rho^{-1})\left(\delta(\rho)-\tau(\rho)\right) d^*\rho.
\end{equation}
\end{prop}
\begin{proof} $(i)$~For $\rho \in \R_+^*$ and $\xi \in L^2(\R^*_+)$, one has
$$
(\vrep(\rho^{-1})P\vrep(\rho)\xi)(\lambda)=(P\vrep(\rho)\xi)(\rho\lambda)=P(\rho\lambda)
(\vrep(\rho)\xi)(\rho\lambda)=P(\rho\lambda)\xi(\lambda),\qquad \forall\lambda \in \R_+^*.
$$
When $\rho\geq 1$ one  obtains:    $P(\lambda)P(\rho\lambda)=P(\lambda)$ for all $\lambda \in \R_+^*$, since both sides are $0$ for $\lambda<1$ and $1$ for $\lambda\geq 1$. Thus, for $\rho\geq 1$ one has 
$$
P\vrep(\rho^{-1})P\vrep(\rho)=P
$$
and  so
\begin{equation}\label{sch13.5}
\vrep(\rho^{-1})-P \vrep(\rho^{-1})P=\left(1-P \vrep(\rho^{-1})P\vrep(\rho)\right)\vrep(\rho^{-1})=(1-P)\vrep(\rho^{-1}).
\end{equation}
By Proposition~\ref{lemsch7}, the Schwartz kernel  of the operator  $(1-P)(\vrep(\rho^{-1})\frac 12\,( u^*_\infty\,\qd u_\infty)^g$ is 
\begin{equation}\label{sch14}
(1-P)(\nu)\ell_\rho(\nu,\mu)=(1-P)(\nu)\ell(\rho\nu,\mu).
\end{equation}
 This equality  together with \eqref{sch6}  gives the formula  
\begin{equation}\label{sch15.0}
(1-P)(\nu)\ell_\rho(\nu,\mu)=4\rho^{\frac 12}\mu^{\frac 12}\nu^{\frac 12}(1-P)(\nu)\int \cos(2\pi\rho \nu y)\cos(2\pi \mu y)(P(1/y)-P(\mu))dy.\end{equation}
Its trace is given by the integral of the diagonal values: $\int (1-P)(\mu)\ell(\rho\mu,\mu)d^*\mu$. This gives using 
$\mu d^*\mu=d\mu$, the following expression for $\mu=\nu=x$ 
\begin{equation}\label{sch15}
4\rho^{\frac 12}\int_{x>0\atop y>0} \cos(2\pi\rho x y)\cos(2\pi x y)(1-P)(x)(P(1/y)-P(x))dydx.\end{equation}
Moreover one has: $(1-P)(x)(P(1/y)-P(x))=(1-P)(x)P(1/y)$ and this product is $0$ unless $x<1$ and $y\leq 1$, in which case it is equal to $1$. Thus  \eqref{sch15}  reduces to \eqref{rhoform}.\newline
$(ii)$~The operator $\frac 12\,( u^*_\infty\,\qd u_\infty)^g=(u_\infty^g)^* P u_\infty^g-P$ is self-adjoint. The adjoint of the operator $R(\rho):=\left(\vrep(\rho^{-1})-P \vrep(\rho^{-1})P\right)$ is $
\left(\vrep(\rho)-P \vrep(\rho)P\right)=R(\rho^{-1})$, thus  for any $\rho \in \R_+^*$ one has, using Proposition \ref{lemsch7} $(iv)$ to justify the permutation under the trace
\begin{align*}
\overline{\delta(\rho)}&=\tr\left(\left(R(\rho)\frac 12\, ( u^*_\infty\,\qd u_\infty)^g\right)^*\right)=\tr\left(\frac 12\, ( u^*_\infty\,\qd u_\infty)^g R(\rho^{-1})\right)=\\ &=\tr\left(R(\rho^{-1})\frac 12\, ( u^*_\infty\,\qd u_\infty)^g \right)=\delta(\rho^{-1}).
\end{align*}
By $(i)$,  $\delta(\rho)$ is real for $\rho \geq 1$, and the above equality shows that it is real for all $\rho \in \R_+^*$, and that $(ii)$ holds. \newline 
$(iii)$~One has $\hat{ P}=(\fourier_{e_\R}^w)^{-1}P\fourier_{e_\R}^w $ and it follows from Lemma~\ref{technfourier} that $\fourier_{e_\R}^w=I\circ u_\infty^g$. Thus since $u^*_\infty= u_\infty^{-1}$  one obtains
$$
\hat{ P}=(u_\infty^{-1})^g\circ I\circ P\circ I\circ u_\infty^g=(u^*_\infty)^g(1-P)u_\infty^g
$$ 
and that, using $[P,u_\infty^g ]=\frac 12(\qd u_\infty)^g $, (see \eqref{sch5})  the  positive operator $P\hat{ P}P $ is equal to 
$$
P\hat{ P}P=P(u^*_\infty)^g(1-P)u_\infty^g P= P(u^*_\infty)^g[(1-P),u_\infty^g ]P=-P\frac 12 (u^*_\infty\,\qd u_\infty)^g P.
$$
Thus, working first at the formal level, one  obtains for $f\in C_c^\infty(\R_+^*)$
\begin{align*}
\tr\left(\vrep(f)P\widehat PP\right)&=-\tr\left(\left(\int f(\rho^{-1})\vrep(\rho^{-1})d^*\rho\right)P\frac 12 (u^*_\infty\,\qd u_\infty)^g P    \right)=
\\
&=-\tr\left(\left(\int f(\rho^{-1})\vrep(\rho^{-1})d^*\rho\right)\frac 12 (u^*_\infty\,\qd u_\infty)^g  \right)+\\ &+\int f(\rho^{-1})\tr\left(\left(\vrep(\rho^{-1})-P \vrep(\rho^{-1})P\right)\frac 12\, (u^*_\infty\,\qd u_\infty)^g\right)d^*\rho=\\
&=\int f(\rho^{-1})\left(\delta(\rho)-\tau(\rho)\right) d^*\rho.
\end{align*}
One needs to exert care to justify the formal manipulation of replacing 
 $$\tr\left(\vrep(f)P\frac 12 (u^*_\infty\,\qd u_\infty)^g P\right)\to \tr\left(P\vrep(f)P\frac 12 (u^*_\infty\,\qd u_\infty)^g \right). $$
 In order to justify  this step one has to show that $\vrep(f)P\frac 12 (u^*_\infty\,\qd u_\infty)^g$ is of trace class. For $f$ in the Schwartz space $\cS(\R_+^*)$, its (multiplicative) Fourier transform is in $\cS(\R)$ and thus the commutator $[P,\vrep(f)]$ is of trace class. Since $(u^*_\infty\,\qd u_\infty)^g$ is bounded, it is thus enough to show that $\vrep(f) (u^*_\infty\,\qd u_\infty)^g$ is of trace class which follows from Proposition \ref{lemsch7} $(iv)$. This justifies the above formal manipulation. \end{proof} 

It follows from  Proposition \ref{proplittlesq} $(i)$ and \eqref{sch8}, that for $\rho\geq 1$, 
\begin{equation}\label{sch17}
\delta(\rho)=4\rho^{\frac 12} I(S)
=4\rho^{\frac 12}\int_0^1 \cos(2\pi\rho t)\cos(2\pi t)(-\log t)dt.
\end{equation}
One can check that
$$
\int_0^1  \cos (2 \pi  a t) \,(-\log t) dt=\frac{\text{Si}(2  \pi a )}{2 \pi  a}
$$
where $\text{Si}(z)$ is the Sine Integral function $\text{Si} (z):=\int _0^z\frac{ \sin (t)}{t}dt$:  an entire function of $z\in \C$. Using the  formula $2 \cos (x) \cos (y)=\cos (x+y)+\cos (x-y)$, one thus obtains, for $\rho\geq 1$ 
\begin{equation}\label{sch18}
\delta(\rho)=2\rho^{\frac 12}\left( \frac{\text{Si}(2  \pi (1+\rho))}{2 \pi (1+\rho)}+\frac{\text{Si}(2  \pi (\rho-1))}{2 \pi (\rho-1)}\right).\end{equation}
Near $\rho=1$ one obtains the following expansion for $\rho\geq 1$ 
\begin{equation}\label{sch18.5}
\delta(\rho)=2 \left(\frac{\text{Si}(4 \pi )}{4 \pi }+1\right)+(\rho-1)-\frac{\left(9 \text{Si}(4 \pi )+64 \pi ^3\right) }{144 \pi }(\rho-1)^2+O\left((\rho-1)^3\right)\end{equation}
Since $\delta(\rho)=\delta(\rho^{-1})$  (Proposition \ref{proplittlesq} $(ii)$), this  shows that the function $\delta(\rho)$ has a jump in its first derivative at $\rho=1$. This fact will play a key role in Chapter~\ref{sectlsqmove} (Theorem~\ref{thmqkey1}). 

Since the Sine Integral function $\text{Si} (z)$ is positive for $z\geq 0$,  $\delta(\rho)$ is positive. Its graph is  shown in Figure \ref{deltarho}.
 \begin{figure}[H]	\begin{center}
\includegraphics[scale=0.65]{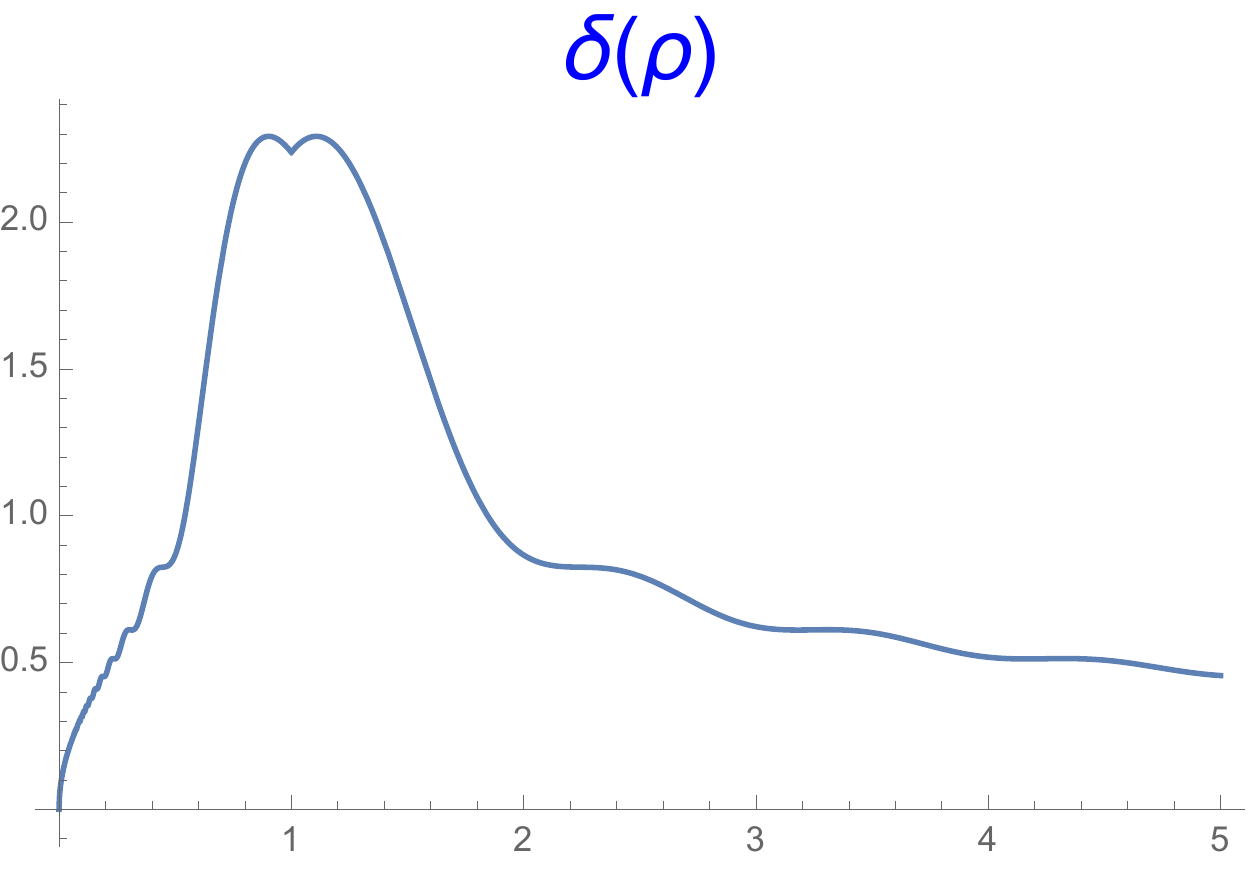}
\end{center}
\caption{Graph of $\delta(\rho)$ \label{deltarho} }
\end{figure} 
Proposition \ref{proplittlesq} also provides the following corollary which  represents  a crucial step towards the positivity of $W_\infty=-W_\R$, by expressing $W_\infty$ (\ie the functional defined by $-\tau$) as the difference between  a positive functional $L$ and the functional defined by $\delta$. In the second statement the positivity of $L$ is translated in terms of Fourier transforms.
\begin{cor}\label{corlittlesq} 
$(i)$~The following functional is positive on the convolution algebra $C_c^\infty(\R_+^*)$
\begin{equation}\label{sch19}
L(f)= \int f(\rho^{-1})\left(\delta(\rho)-\tau(\rho)\right) d^*\rho.
\end{equation}
	$(ii)$~The function $2\theta'(t)+\hat \delta(t)$ is non-negative, where $\hat \delta(t):=\int_{\R_+^*}\delta(\rho)\rho^{-it}d^*\rho $.
\end{cor}
\begin{proof} $(i)$~By Proposition \ref{proplittlesq} $(iii)$  $L(f):=\tr\left(\vrep(f)P\widehat PP\right)$  is positive for $f=g*g^*$ as the trace of a product of two positive operators.\newline
$(ii)$~By applying \eqref{qdofu1}, one has
$$-\int f(\rho^{-1})\tau(\rho) d^*\rho=-\Tr \left (\hat f\left(\frac{1}{2} u^{-1}\,  \qd \,u\right)\right)=\int \hat f(t)\frac{2\partial_t\theta(t)}{2 \pi}dt.
$$ 
Moreover, by Parseval's formula,
$$
\int f(\rho^{-1})\delta(\rho) d^*\rho=\int f(\rho)\delta(\rho) d^*\rho=\frac{1}{2\pi}\int \hat f(t)\hat \delta(t)dt.
$$
Thus one obtains 
\begin{equation}\label{sch20}
L(f)= \int \hat f(t)\left(2\partial_t\theta(t)+\hat \delta(t)\right)\frac{dt}{2 \pi}
\end{equation}
and the positivity of  $L$ implies (in fact it is equivalent to) the positivity of the function $2\theta'(t)+\hat \delta(t)$.\end{proof} 

\begin{figure}[H]
\begin{minipage}[b]{0.43\linewidth}
\centering
\includegraphics[width=\textwidth]{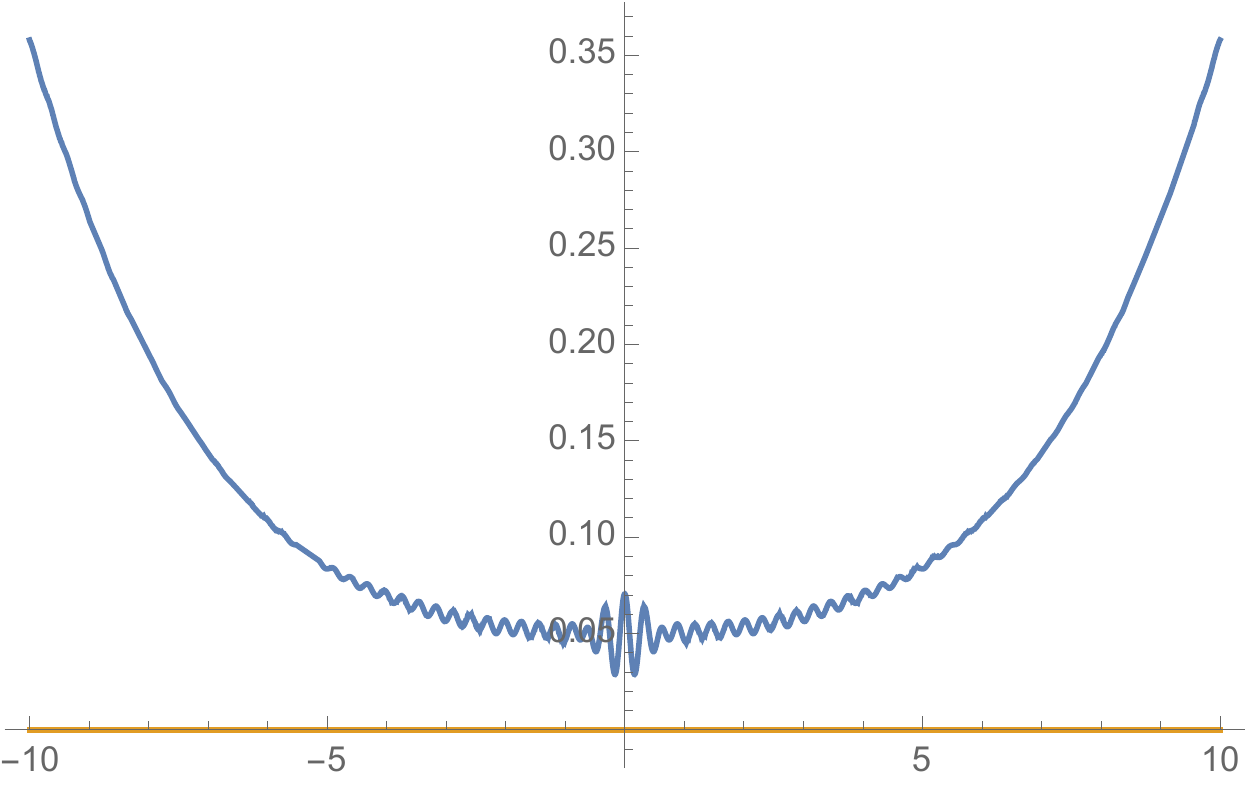}
\caption{\small{Graph of $2\theta'(t)+\hat \delta(t)$ in $[-10,10]$}}
\label{poscheck0}
\end{minipage}
\hspace{0.5cm}
\begin{minipage}[b]{0.45\linewidth}
\centering
\includegraphics[width=\textwidth]{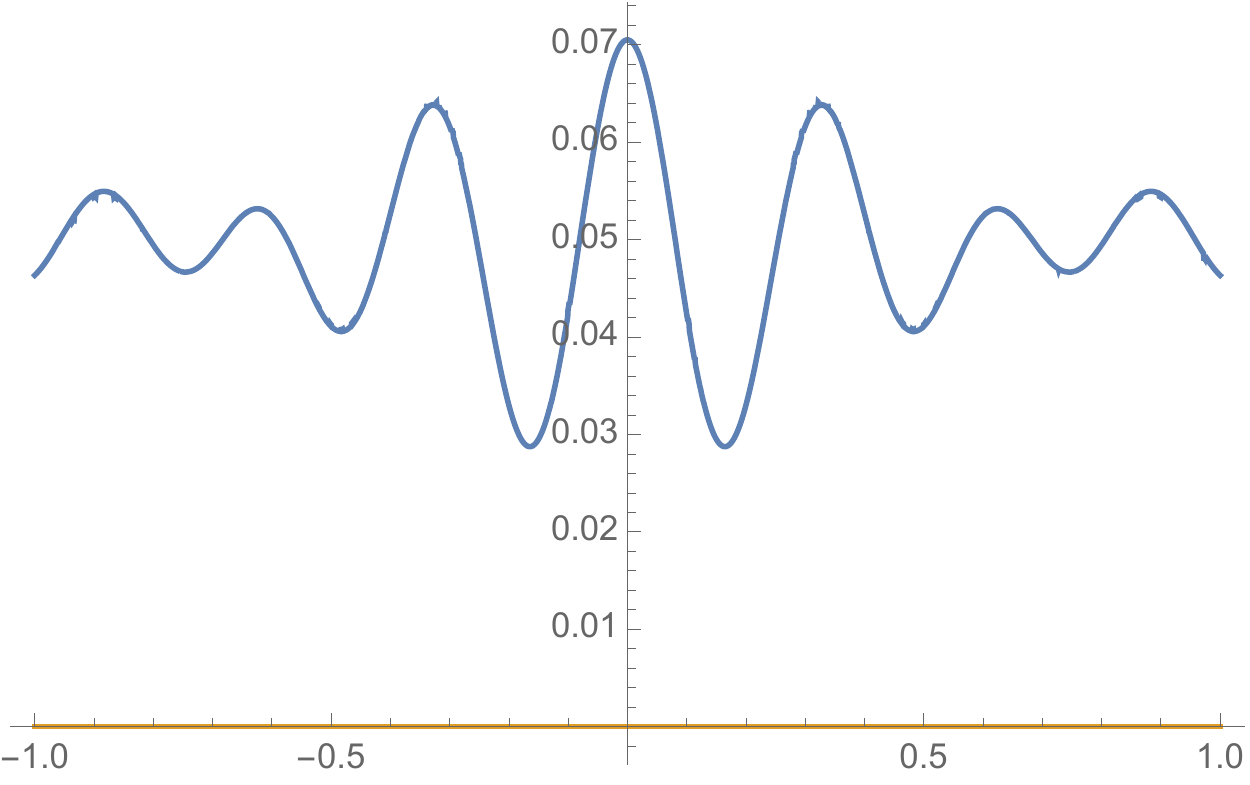}
\caption{\small{Graph of $2\theta'(t)+\hat \delta(t)$ in $[-1,1]$}}
\label{poscheck}
\end{minipage}
\end{figure}



\vspace{.1in}

Figures \ref{poscheck0} and \ref{poscheck} show the graph of $2\theta'(t)+\hat \delta(t)$  in two versions (the second is zoomed near the origin) displaying that this function is positive. We use the formula (obtained from the symmetry  $\delta(\rho)=\delta(\rho^{-1})$)
$$
\hat \delta(t)=\int_1^\infty \delta(\rho)2\cos(t \log \rho)d^*\rho.
$$
One knows that  $\frac{\text{Si}(x)}x\leq 1$ for $x>0$, while $\text{Si}(x)\leq 2$ for $x>0$ and $\text{Si}(x)\to \pi/2$ as $x\to \infty$. This provides    the estimate, as $\rho\to \infty$,
$$
\delta(\rho)=\rho^{-1/2}+O\left(\rho^{-3/2}\right)
$$
showing that the function $\hat \delta(t)$ is smooth since the derivatives in $t$ only involve powers of $\log \rho$ which do not alter the absolute convergence of the integral. 

\begin{figure}[H]
\begin{minipage}[b]{0.43\linewidth}
\centering
\includegraphics[width=\textwidth]{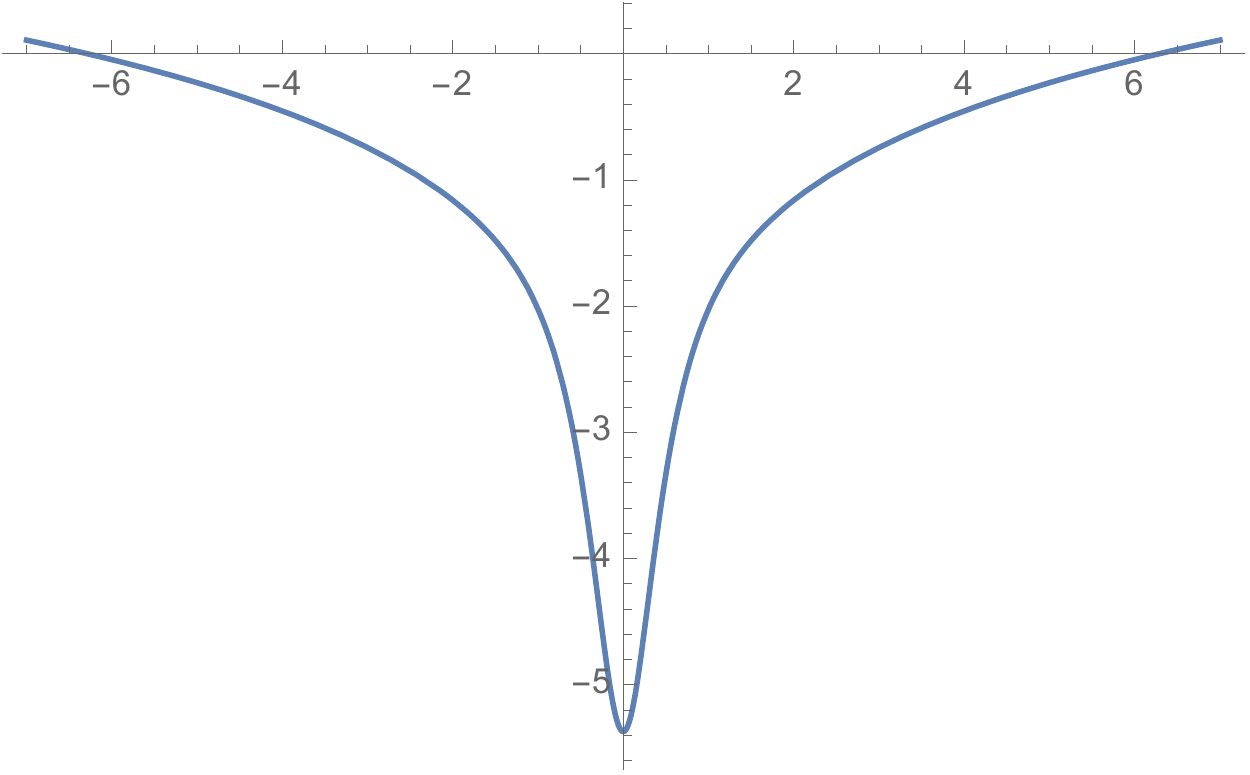}
\caption{\small{Graph of $2\theta'(t)$ in $[-7,7]$}}
\label{poscheck2}
\end{minipage}
\hspace{0.5cm}
\begin{minipage}[b]{0.45\linewidth}
\centering
\includegraphics[width=\textwidth]{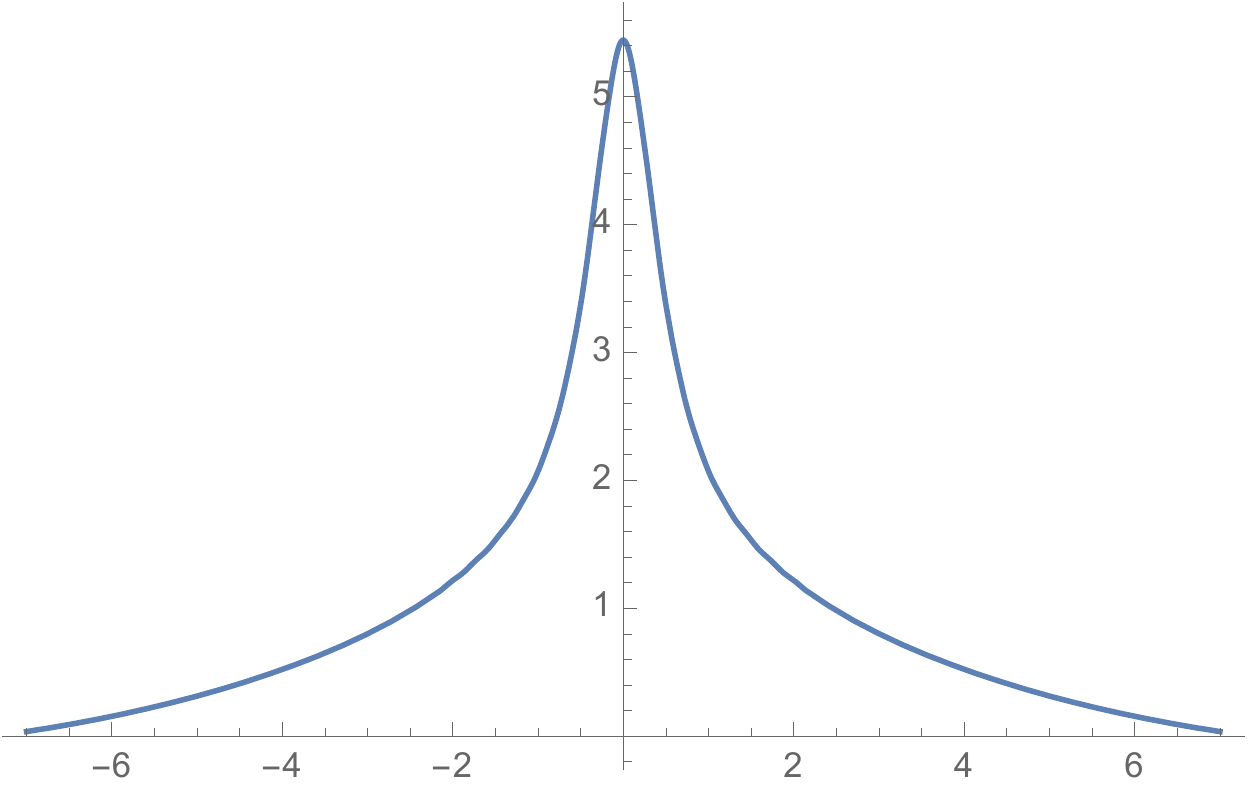}
\caption{\small{Graph of $\hat \delta(t)$ in $[-7,7]$}}
\label{poscheck3}
\end{minipage}
\end{figure}


The graphs in Figures~\ref{poscheck2} and \ref{poscheck3} show the striking precision with which  $\hat \delta(t)$ imitates the function $-2\theta'(t)$ in the interval $[-7,7]$. However, unlike $2\theta'(t)$ which tends to infinity when $\vert t\vert \to \infty$, the function $\hat \delta(t)$ tends to $0$ when $\vert t\vert \to \infty$, being the Fourier transform of an integrable function.

\section{Support and boundary conditions}\label{sectsupport} 

Throughout this section we use the following terminology 
\begin{defn} \label{defnposdef}
Let $G$ be a locally compact abelian group and $f\in L^1(G,dg)$. We say that $f$ is positive definite  when its Fourier transform is pointwise positive, \ie $\widehat f(t)\geq 0 $,  $\forall t\in \widehat G$. 	
\end{defn}

We  recall the following result of  Boas and Kac (\!\cite{BK} Lemma 5.1 and \cite{EGR})

\begin{prop}\label{boaskac}
Let $f\in C_c^\infty(\R)$ have support in the interval $[-A,A]$ ($A>0$). The following conditions are equivalent
\begin{enumerate}
\item The Fourier transform	$\widehat f$ is pointwise positive.
\item There exists $g\in C_c^\infty(\R)$ with support in $[-A/2,A/2]$ such that $f=g*g^*$.
\end{enumerate}
\end{prop}
\begin{proof} 
Assume 1. Then by Lemma 5.1 of \cite{BK},  since the Fourier transform of $f$ is in $L^1(\R)$  and is pointwise positive, $f$ can be written as $g*g^*$, where $g$ is square integrable and has support in $[-A/2,A/2]$. One has $\widehat f(t)=\vert\widehat g(t)\vert^2$, and since $f\in C_c^\infty(\R)$, one gets $\widehat f(t)=O(\vert t \vert ^{-N})$ for any $N$. The same holds for $\widehat g(t)$  showing that $g$ is smooth. \newline
Conversely, the equality $\widehat f(t)=\vert\widehat g(t)\vert^2$ shows that $f$ is  positive definite, moreover its support is contained in $[-A,A]$.  \end{proof} 

This result applied to the  multiplicative group $\R_+^*$ (isomorphic to $\R$) gives, for any positive definite function $f\in C_c^\infty(\R_+^*)$ with support in a symmetric interval $I=[A^{-1},A]\subset \R_+^*$, a convolution square root with support in $\sqrt I=[A^{-1/2},A^{1/2}]$. 

Next, we investigate 
 the functional  on the convolution algebra $C_c^\infty(\R_+^*)$
\begin{equation}\label{sch22}
W_\infty(f)= -\int f(\rho^{-1})\tau(\rho) d^*\rho
\end{equation}
for test functions $f$ whose support is in the interval $[\frac 12,2]$. Moreover, we assume  the vanishing conditions
\begin{equation}\label{vanishing}
\int f(\rho)\rho^{\pm \frac 12}d^*\rho=0
\end{equation}
 to isolate on the left hand side of the explicit formula the contribution of the zeros of the Riemann zeta function. To understand why the support condition is needed,  we give an example of a test function $f$ on $\R_+^*$  of the form $f=g*g^*$ that fulfills  \eqref{vanishing}, but for which $W_\infty(f)<0$. We first determine the Fourier transform $\hat f(t):=\int f(\rho)\rho^{-it}d^*\rho$. With $\hat f(t)=\vert \hat g(t)\vert^2$ we let 
$$
\hat g(t)=(1+4t^2)e^{-t^2/4}, \qquad \hat f(t)=(1+4t^2)^2e^{-t^2/2}
$$
\begin{figure}[H]	\begin{center}
\includegraphics[scale=0.5]{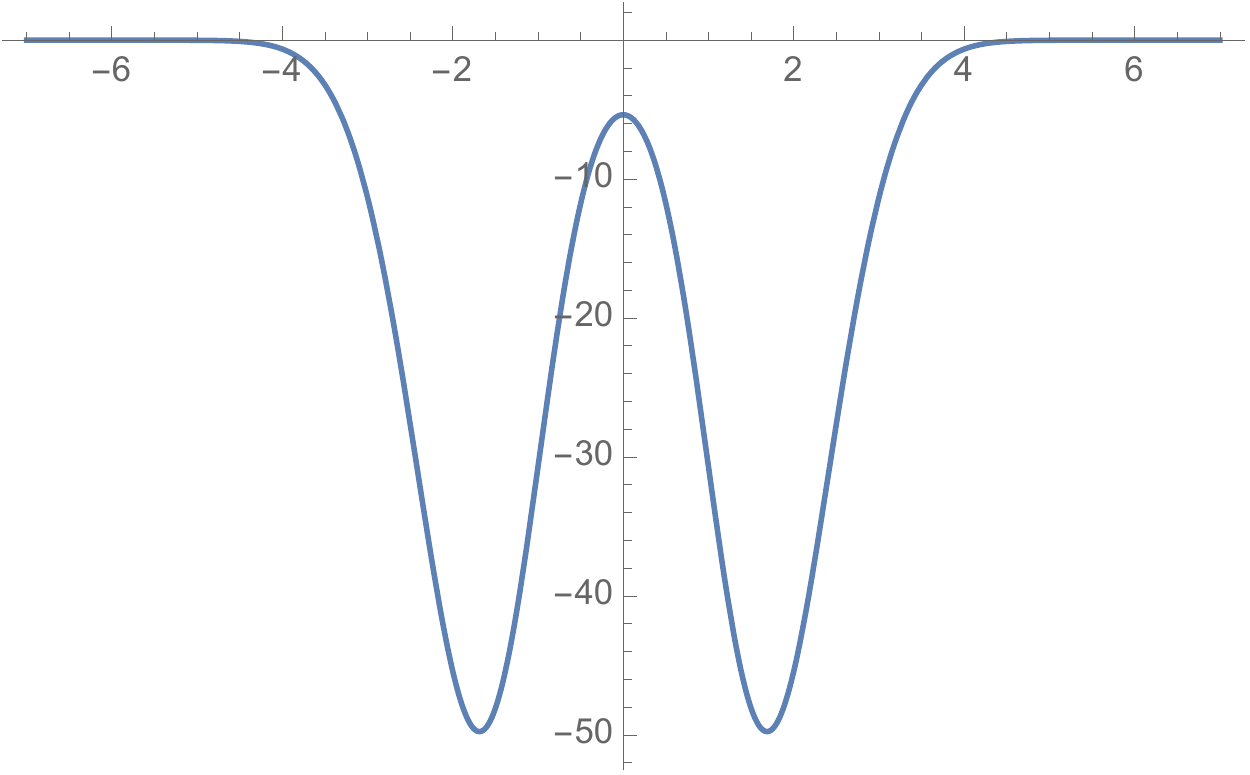}
\end{center}
\caption{Graph of $2\theta'(t)\hat f(t)$ in $[-7,7]$ \label{need1} }
\end{figure}
By construction one has $\hat f(\pm i/2)=0$, so that \eqref{vanishing} holds, moreover by \eqref{qdofu1}
$$
W_\infty(f)= \int \hat f(t)\frac{2\partial_t\theta(t)}{2 \pi}dt.
$$
The graph of $2\theta'(t)\hat f(t)$ in $[-7,7]$ (see Figure \ref{need1}) shows that $\hat f(t)$ is concentrated in the region where $\theta'(t)<0$ and the approximate value of the functional is $W_\infty(f)\sim -28.8971<0$.\newline
Next lemma shows that,  in the convolution algebra $C_c^\infty(\R_+^*)$, the ideal $\cJ$  defined by the vanishing condition \eqref{vanishing} is the range of a second order differential operator 

\begin{lem}\label{vanishing1} $(i)$~The vanishing conditions \eqref{vanishing} define an ideal $\cJ$ in the convolution algebra $C_c^\infty(\R_+^*)$.\newline
$(ii)$~Let $g\in C_c^\infty(\R_+^*)$, then $(-(\rho\partial_\rho)^2+\frac 14)g\in \cJ$ and its support is contained in the support of $g$.\newline
$(iii)$~Let $f\in C_c^\infty(\R_+^*)$ with support in an interval $I$, fulfill the vanishing conditions \eqref{vanishing}. Then there exists $g\in C_c^\infty(\R_+^*)$ with support in  
$I$ and such that \begin{equation}\label{Qop}Q(g):=(-(\rho\partial_\rho)^2+\frac 14)g=f.\end{equation}
One has $g= Y^**Y*f, $
where $Y(\rho)=0$ for $\rho<1$, $Y(\rho)=\rho^{\frac 12}$ for $\rho\geq 1$ and $ Y^*(\rho):=Y(\rho^{-1})$.\newline
$(iv)$~Let $f$ and $g$ as in $(iii)$, then $f$ is positive definite if and only if $g$ is positive definite.
\end{lem}
\begin{proof} $(i)$~For any complex power $z$ the functional $\int f(\rho)\rho^{z}d^*\rho$ defines a character of the convolution algebra $C_c^\infty(\R_+^*)$. The vanishing conditions \eqref{vanishing} thus define the intersection of the kernels of two characters, \ie an ideal. \newline
$(ii)$~Notice that the (linear) kernel of the operator $-(\rho\partial_\rho)^2+\frac 14$ acting on  distributions\footnote{\ie here the dual of $C_c^\infty(\R_+^*)$ which is strictly larger than the space $\cS'(\R_+^*)$ of tempered distributions, and $\rho^{\pm \frac 12}\notin \cS'(\R_+^*)$} on $\R_+^*$  contains the functions $\rho^{\pm \frac 12}$. Using integration by parts, since $g$ has compact support, one gets
$$
\int \left((-(\rho\partial_\rho)^2+\frac 14)g(\rho)\right)\rho^{\pm \frac 12}d^*\rho
=\int g(\rho)\left((-(\rho\partial_\rho)^2+\frac 14)\rho^{\pm \frac 12}\right)d^*\rho=0.
$$
Moreover $(-(\rho\partial_\rho)^2+\frac 14)g(\rho)$ vanishes identically outside the support of $g$. \newline
$(iii)$~The vanishing conditions \eqref{vanishing} give $\int_I v^{- \frac 12} f(v) \, d^*v=0$ and this shows that the function $k(u):=u^{\frac 12}\int_{0 }^u v^{- \frac 12} f(v) \, d^*v$  vanishes when $u\notin I$.   Thus the support of $k$ is contained in $I$.  Using again integration by parts and \eqref{vanishing} one obtains
$$
\int_I u^{\frac 12}k(u)d^*u=\int_I \left(\int_{0 }^u v^{- \frac 32} f(v) \, dv\right ) du=-\int_I u^{- \frac 32}f(u) udu=-\int_I u^{ \frac 12}f(u) d^*u=0.
$$
  Thus  one again derives that   $g(\rho):=\rho^{- \frac 12}\int_{\rho }^\infty u^{\frac 12}k(u)d^*u$ also has support in $I$. 
Moreover, one has 
\begin{align*}
\rho\partial_\rho g(\rho)&=- \frac 12 g(\rho)-\rho^{- \frac 12}\rho \int_{0 }^\rho v^{- \frac 12} f(v) \, d^*v,\\
(\rho\partial_\rho)^2g(\rho)&=- \frac 12 \rho\partial_\rho g(\rho)-\rho\partial_\rho\left(\rho^{\frac 12} \int_{0 }^\rho v^{- \frac 12} f(v) \, d^*v\right)=\\
&=- \frac 12 \rho\partial_\rho g(\rho)-\frac 12 \left(\rho^{\frac 12} \int_{0 }^\rho v^{- \frac 12} f(v) \, d^*v\right)-f(\rho)= \frac 14 g(\rho)-f(\rho).
\end{align*}
This gives $Q(g)=(-(\rho\partial_\rho)^2+\frac 14)g=f$.  Note that by construction one has: $k=Y*f$, where $Y(\rho)=0$ for $\rho<1$ and $Y(\rho)=\rho^{\frac 12}$ for $\rho\geq 1$  and that $g= Y^**Y*f$.\newline
$(iv)$~follows since the Fourier transforms are related by the equality $(\frac 14+t^2)\widehat g(t)=\widehat f(t)$.
\end{proof} 
For a symmetric  interval $I\subset \R_+^*$ we let $C_c^\infty(I)\subset C_c^\infty(\R_+^*)$ be the subspace of functions whose support is contained in $I$. 
 \begin{defn} \label{defnposfunctional}
Let $E \subset C_c^\infty(\R_+^*)$ be a subspace and $\cL$ a linear form on $E$. Then $\cL$ is said to be positive if  $\cL(f)\geq 0$ for any positive definite $f\in E$. \end{defn}
 
 Next proposition  plays a central role. It gives a criterion to test the positivity of a functional $\phi$ after imposing the vanishing conditions \eqref{vanishing}, \ie  on the intersection $C_c^\infty(I)\cap \cJ$, by testing the  positivity  of $\phi\circ Q$ under the same support conditions \ie on $C_c^\infty(I)$.
 
\begin{prop}\label{vanishing2} Let $\phi$ be a functional on $C_c^\infty(\R_+^*)$ and $I$ a symmetric interval. Then the restriction of $\phi$ to $C_c^\infty(I)\cap \cJ$ is positive  if and only if the functional $\phi\circ Q$ is positive  on $C_c^\infty(I)\subset C_c^\infty(\R_+^*)$.	
\end{prop}
\begin{proof} Assume that the restriction of $\phi$ to $C_c^\infty(I)\cap \cJ$ is positive and let $g\in C_c^\infty(I)$ be positive definite. Then $Qg$ is also positive definite by Lemma \ref{vanishing1} $(iv)$. Moreover by  $(i)$ of the same lemma one has: $Qg\in C_c^\infty(I)\cap \cJ$. Thus $\phi(Qg)\geq 0$.\newline
Conversely, assume  that the functional $\phi\circ Q$ is positive  on $C_c^\infty(I)\subset C_c^\infty(\R_+^*)$ and	 let $f\in C_c^\infty(I)\cap \cJ$ be positive definite.  Let  $g\in C_c^\infty(I)\subset C_c^\infty(\R_+^*)$, with $Qg=f$ (Lemma \ref{vanishing1} $(iii)$), then by $(iv)$ of the same lemma, $g$ is positive definite so that $\phi(f)=\phi(Qg)\geq 0$.\end{proof} 

As just proved in the proposition, for a given functional $\phi\in C_c^\infty(\R_+^*)$ the positivity of $\phi$ on $C_c^\infty(I)$ implies the positivity of $\phi\circ Q$ on $C_c^\infty(I)$.
Next, we explain why it is easier  to obtain the positivity of $\phi\circ Q$  on $C_c^\infty(I)\subset C_c^\infty(\R_+^*)$ than the positivity of $\phi$ (on $C_c^\infty(I)$).  

Consider the functional
\begin{equation}\label{sch23}
D(f)= \int f(\rho^{-1})\delta(\rho) d^*\rho.
\end{equation}
It follows from Corollary \ref{corlittlesq} that  the functional $L=D+W_\infty$, with $W_\infty$ as in \eqref{sch22}, is positive  on $C_c^\infty(\R_+^*)$. One derives from Proposition \ref{vanishing2} the following implication
\begin{equation}\label{control}
D\circ Q\leq 0 \ \ \text{on}\ \ C_c^\infty(I)~\Longrightarrow~ W_\infty\geq 0\ \ \text{on}\ \ C_c^\infty(I)\cap \cJ.
\end{equation}
Next theorem shows  that the quadratic form $D\circ Q(\xi*\xi^*)$ associated to $D\circ Q$ is essentially negative. Therefore one  needs to impose only finitely many linear conditions on test functions to obtain $D\circ Q\leq 0$ and hence  $W_\infty\geq 0$ on $C_c^\infty(I)\cap \cJ$.

\begin{thm}\label{thmqkey1} Let $I\subset \R_+^*$ be a symmetric and bounded interval. There exists a compact operator $K_I$ in the Hilbert space $L^2(\sqrt I, d^*\rho)$ such that for any vector $\xi\in L^2(\sqrt I, d^*\rho)$ one has\footnote{We use the convention that the inner product $\langle \xi\mid \eta\rangle$ is antilnear in $\xi$ (and liner in $\eta$)} 
 \begin{equation}\label{thmQprime}
	D\circ Q(\xi*\xi^*)=\langle \xi \mid (-2\, \id +K_I)\xi\rangle =-2\Vert \xi \Vert^2+\langle \xi\mid K_I\,\xi\rangle.
\end{equation}
\end{thm}
\begin{proof} We use the group isomorphism $\exp:\R\to \R_+^*$ to transfer  the statement to the convolution algebra $C_c^\infty(\R)$ and the interval $I':= \log I$. The operator $Q$ becomes \[Q_+:=-\partial_x^2+\frac 14\]
and the intent is to analyse the functional $D_+\circ Q_+$ on $C_c^\infty(I')\subset C_c^\infty(\R)$  where 
$$
D_+(f):=\int f(x)\delta(\exp(\vert x\vert))dx,\qquad \forall f\in C_c^\infty(\R).
$$
 We use the symmetry $\delta(\rho)=\delta(\rho^{-1})$ to reformulate the integrand. By integration by parts one obtains
 $$
 D_+(Q_+f)=\int Q_+f(x)\delta(\exp(\vert x\vert))dx=\int f(x)Q_+\delta(\exp(\vert x\vert))dx
 $$
 where $Q_+\delta(\exp(\vert x\vert))$ is a distribution. 
 One  has 
 $$
\delta(\exp(\vert x\vert))= 2 \left(\frac{\text{Si}(4 \pi )}{4 \pi }+1\right)+\vert x\vert+\frac{\left(-9 \text{Si}(4 \pi )-64 \pi ^3+72 \pi \right) x^2}{144 \pi }+O\left(\vert x\vert^3\right).
 $$
 Thus $Q_+\delta(\exp(\vert x\vert))=(-\partial_x^2+\frac 14)\delta(\exp(\vert x\vert))$ gives the sum of $-2\dirac_0$ (where $\dirac_0$ is the Dirac distribution at $0$) and the even function  which coincides with $(-\partial_x^2+\frac 14)\delta(\exp( x))$ for $x\geq 0$. Equivalently, one has 
 \begin{equation}\label{Qprime}
	 D_+(Q_+f)=-2f(0)+\int_0^\infty (f(x)+f(-x))Q_+\delta(\exp(x))dx.
\end{equation}
 Now we let $f=\xi*\xi^*$, where  $\xi\in L^2(\sqrt I, d^*\rho)$. Then: $f(0)=(\xi*\xi^*)(0)=\Vert \xi \Vert^2$. Moreover the following formula defines a compact operator $K_I$  in the Hilbert space $L^2(\sqrt I, d^*\rho)$ 
$$
\langle \xi\mid K_I(\eta)\rangle :=\int_1^\infty\left( (\xi^**\eta)(\rho)+(\xi^**\eta)(\rho^{-1})\right) Q\delta(\rho)d^*\rho 
$$
which gives \eqref{thmQprime}, using \eqref{Qprime}. To prove that $K_I$ is a compact operator we first show that it
 is given by a Schwartz kernel $\tilde K_I(v,u)$. One has  
$$
 (\xi^**\eta)(\rho)=\int \overline{\xi(u^{-1})}\eta(\rho u^{-1})d^*u, \ \ 
 \int_1^\infty (\xi^**\eta)(\rho) Q\delta(\rho)d^*\rho=\int_J \overline{\xi(v)}\eta(u)Q\delta(u/v)d^*ud^*v
$$
where $J=\{(v,u)\in \sqrt I\times \sqrt I\mid u/v\geq 1\} $. Similarly one gets 
\begin{align*}
 (\xi^**\eta)(\rho^{-1})&=\int \overline{\xi(u^{-1})}\eta(\rho^{-1} u^{-1})d^*u, \\  
 \int_1^\infty (\xi^**\eta)(\rho^{-1}) Q\delta(\rho)d^*\rho&=\int_{J'} \overline{\xi(v)}\eta(u)Q\delta(v/u)d^*ud^*v
\end{align*}
where $J'=\{(v,u)\in \sqrt I\times \sqrt I\mid v/u\geq 1\} $. This shows that the Schwartz kernel $\tilde K_I(v,u)$ is  defined as follows
$$
\tilde K_I(v,u)=\begin{cases}Q\delta(u/v) & \text{if $u\geq v$}, \\Q\delta(v/u) &\text{if $v\geq u$.}
\end{cases}
$$
Since the interval $\sqrt I$ is bounded, the function $\tilde K_I(v,u)$ is square integrable and hence the operator $K_I$ is of Hilbert-Schmidt class and hence compact. 
\end{proof}
\begin{rem}\label{remQprime} Since the equality \eqref{Qprime} plays a key role in the sequel we give  an elementary proof  without appealing to distribution theory. Let $k(x)$ be an even function on $\R$ whose restriction to $[0,\infty)$ is a smooth function, then for any $f\in C_c^\infty(\R)$ one has
\begin{align*}
\int_\R f''(x)k(x)dx&=\int_{-\infty}^0 f''(x)k(x)dx+\int_0^\infty f''(x)k(x)dx\\
\int_0^\infty f''(x)k(x)dx&=\bigl[f'(x)k(x)\bigr]_0^\infty -\int_0^\infty f'(x)k'(x)dx=-f'(0)k(0)-\int_0^\infty f'(x)k'(x)dx=\\
&=-f'(0)k(0)-\bigl[f(x)k'(x)\bigr]_0^\infty+\int_0^\infty f(x)k''(x)dx=\\
&=-f'(0)k(0)+f(0)k'(0^+)+\int_0^\infty f(x)k''(x)dx
\end{align*}
where  $k'(0^+):=\displaystyle{\lim_{\epsilon\to 0\atop \epsilon >0}}\frac{k(\epsilon)-k(0)}{\epsilon} $.  Similarly one has
\begin{align*}
\int_{-\infty}^0 f''(x)k(x)dx&=\bigl[f'(x)k(x)\bigr]_{-\infty}^0 -\int_{-\infty}^0 f'(x)k'(x)dx=f'(0)k(0)-\int_{-\infty}^0 f'(x)k'(x)dx=\\
&=f'(0)k(0)-\bigl[f(x)k'(x)\bigr]_{-\infty}^0+\int_{-\infty}^0 f(x)k''(x)dx=\\
&=f'(0)k(0)-f(0)k'(0^-)+\int_0^\infty f(x)k''(x)dx.
\end{align*}
Thus in the global sum the boundary terms $-f'(0)k(0)$ and $f'(0)k(0)$ cancel out but the boundary  terms $f(0)k'(0^+)$ and $-f(0)k'(0^-)$ do not, due to the discontinuity of the first derivative of $k$. Indeed, since  $k(x)$ is an even function one has: $k'(0^-)=-k'(0^+)$, thus is $k'(0^+)\neq 0$ the first derivative is discontinuous at zero. Letting then  $k(x):=\delta(\exp(\vert x\vert))$, one obtains \eqref{Qprime}.	
\end{rem}
 \begin{figure}[H]	\begin{center}
\includegraphics[scale=0.8]{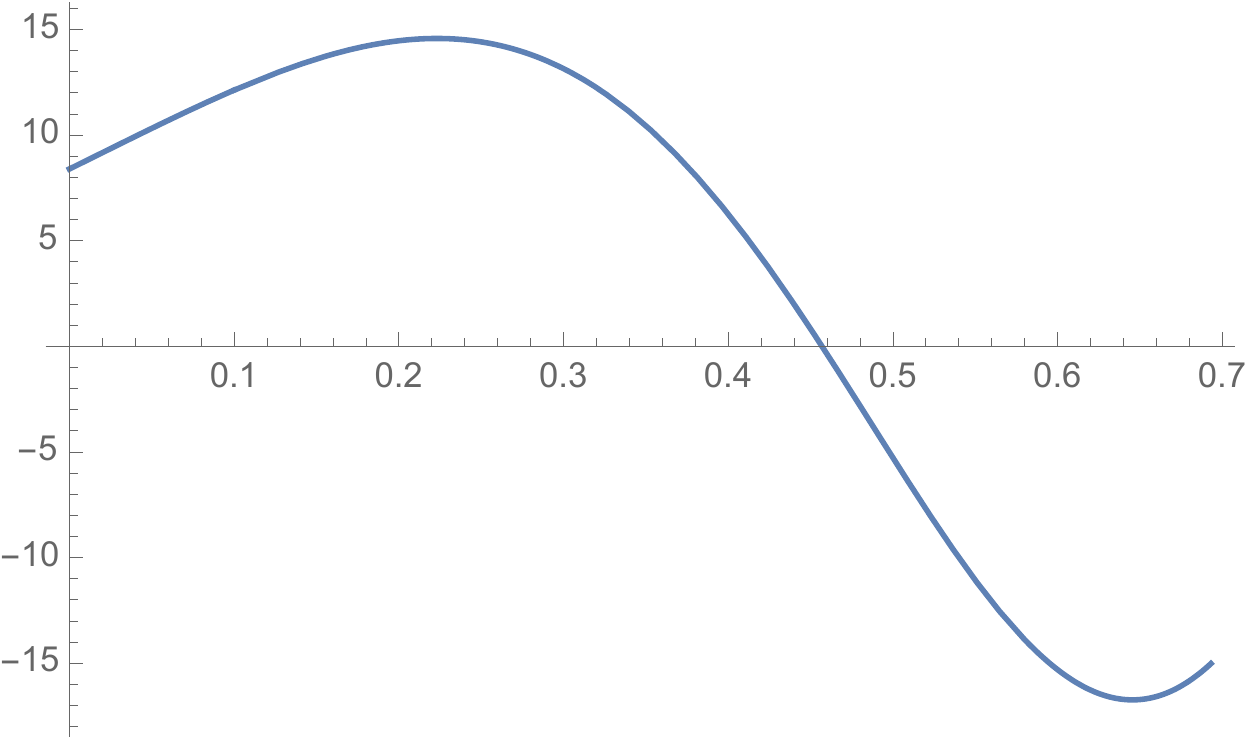}
\end{center}
\caption{Graph of $Q_+\delta(\exp(x))$ in $[0,\log 2]$ \label{qprime} }
\end{figure}

\begin{cor}\label{qeasy} One has  $D\circ Q\leq 0 \ \ \text{on}\ \ C_c^\infty(I)$, where $I=[u^{-1},u]$, $u=1.10246$.
\end{cor}
\begin{proof} We adopt the same strategy as at the beginning of the proof of the above theorem. Thus we test the negativity  $D_+\circ Q_+$ on $C_c^\infty([-\log u,\log u])\subset C_c^\infty(\R)$.
When the function $f\in C_c^\infty(\R)$ is positive definite  \ie the matrix $f(x_i-x_j)$ is a positive matrix for any $x_i\in \R$, and this implies in particular that $\vert f(x)\vert \leq f(0)$ for all $x\in \R$. It follows that the functional $D_+\circ Q_+$ is positive on $C_c^\infty(I)$ where 
$I$ is a symmetric interval $[-s,s]$, as long as 
$$
\int_0^s \vert Q_+\delta(\exp(x))\vert dx\leq 1.
$$
One verifies numerically that this holds for $s=0.097542$, which gives $e^s\sim u$.\end{proof} 

\begin{rem}\label{improve} $(i)$~One can improve the constant $u$ of Corollary~\ref{qeasy} using the  bound given by Theorem~2 of \cite{BK}\footnote{A typo corrected in \cite{BKe} indicates that one should use the ceiling function $\left\lceil \frac{s}{x}\right\rceil$ in the formula.}
$$
\vert f(x)\vert \leq f(0)\,\cos \left(\frac{\pi }{\left\lceil \frac{s}{x}\right\rceil +1}\right)
$$
for any positive definite function $f$  with support in the interval $[-s,s]$. The condition on $s$ now becomes
$$
\int_0^s \vert Q_+\delta(\exp(x))\vert \,\cos \left(\frac{\pi }{\left\lceil \frac{s}{x}\right\rceil +1}\right)
 dx\leq 1.
$$
This holds for $s=0.14043$ and improves the value of $u$ of Corollary~\ref{qeasy} to $1.15077$.	

$(ii)$~The functional  $D\circ Q$ is not negative on $[2^{-1},2]$. 
To see this we construct a positive definite function $f\in C_c^\infty(\R)$, which has support in $[-\log 2,\log 2]$, with $D_+(Q_+f)>0$. Let $\epsilon >0$, $\phi_\epsilon(x)$ be a positive smooth function of integral equal to $1$ and with support in $[-\epsilon,\epsilon]$. We let 
$$
g_\epsilon= \phi_\epsilon* 1_{[\epsilon,Log 2-\epsilon]}, \qquad g_\epsilon(x)=\int_\epsilon^{Log 2-\epsilon} \phi_\epsilon(x-t)dt.
$$
By construction one has $0\leq g_\epsilon\leq 1$, $g_\epsilon\in C_c^\infty([0,\log 2])$ and   
$g_\epsilon(x)=1$ for $x\in [2\epsilon,\log 2-2\epsilon]$. 

Let $f_\epsilon:=g_\epsilon*g_\epsilon^*$. One has 
$$
f_\epsilon(-x)=f_\epsilon(x), \qquad f_\epsilon(x)=\int g_\epsilon(y)g_\epsilon(x+y)dy.
$$
Then $f_\epsilon$ is positive in the convolution algebra $C_c^\infty(\R)$, and has support in $[-\log 2,\log 2]$. One obtains the inequality 
$$
\vert f_\epsilon(x)-( \log 2-\vert x\vert)\vert\leq 4 \epsilon,\qquad \forall x\in [-\log 2,\log 2].
$$
The functional $D_+(Q_+f_\epsilon)$ is given by \eqref{Qprime} and thus depends continuously on $\epsilon >0$. The limit when $\epsilon \to 0$ is $\sim 2.98699$.
\end{rem}

\section{Moving $\Delta$ inside $\Sigma$}\label{sectlsqmove}

In the previous section we proved that the functional $D(f)= \int f(\rho^{-1})\delta(\rho) d^*\rho$  is  the difference $L-W_\infty$  of the positive functional $L$ in \eqref{sch19}  and the Weil functional $W_\infty$  \eqref{sch22}.  The implication \eqref{control} shows that in order to prove Weil's positivity one needs to control the sign of the functional $D\circ Q$, where the differential operator $Q=-(\rho\partial_\rho)^2+\frac 14$ implements the vanishing conditions \eqref{vanishing} (Proposition \ref{vanishing2}). Moreover, in Theorem~\ref{thmqkey1} we proved that the functional $D\circ Q$ is essentially negative since it is represented by the  operator $-2\, \id +K_I$, where $K_I$ is a compact operator depending on the interval $I$ on which the test functions are supported.

In this section we refine the decomposition $W_\infty=L-D$  (Theorem~\ref{devil}) in the form: $W_\infty=S-E$, where $S\leq L$ is still a positive trace-functional, and the negativity of the   ``remainder" $E$, will be easier to handle later since $E\leq D$. 
This refinement plays a crucial role in the explicit computations of Section \ref{sectspectrum}. 

The  main idea is to implement   the geometry of pairs of projections in Hilbert space to ``move" the contribution of the small square $\Delta$ inside the big square $\Sigma$. 

The starting point is the following Lemma~\ref{onequantum} which allows one to relate the trace $\delta(\rho)$ of \eqref{chirem} with the cutoff projections introduced in \cite{Co-zeta}. Here, as above, $P$ denotes the projection operator  given in  $L^2(\R_+^*,d^*\lambda)$ as    multiplication by the characteristic function $1_{[1,\infty)}$ and  $\widehat P=(\fourier_{e_\R}^w)^{-1}P\fourier_{e_\R}^w$. 

\begin{lem}\label{onequantum}
	For $\rho\geq 1$ one has \begin{equation}\label{disponequa}\delta(\rho)=\tr\left(\vrep(\rho^{-1})(1-\widehat P)\,(1-P)\right).\end{equation}
\end{lem}
\begin{proof} Since $\rho\geq 1$ it follows from \eqref{sch13.5} that
$$
\vrep(\rho^{-1})-P \vrep(\rho^{-1})P=(1-P)\vrep(\rho^{-1})
$$
so that one obtains
\begin{align*}
\delta(\rho)&=\tr\left(\left(\vrep(\rho^{-1})-P \vrep(\rho^{-1})P\right)\frac 12\, (u^*_\infty\,\qd u_\infty)^g\right)=\tr\left((1-P)\vrep(\rho^{-1})\frac 12\,( u^*_\infty\,\qd u_\infty)^g\right)=\\
&=\tr\left((1-P)\vrep(\rho^{-1})
\, ((u^g_\infty)^*\,P\, u_\infty^g-P)\right).
\end{align*}
Now we use the equalities $\fourier_{e_\R}^w=I\circ u_\infty^g$  (see \eqref{FwIPhi}) and  $I\circ P\circ I=1-P$ and get   $$(u^g_\infty)^*\,P\, u_\infty^g=(\fourier_{e_\R}^w)^{-1}(1-P)\fourier_{e_\R}^w=(1-\widehat P).$$
Using the cyclic property of the trace together with $P(1-P)=0$ we finally have
$$
\delta(\rho)=\tr\left((1-P)\vrep(\rho^{-1})
\, ((u^g_\infty)^*\,P\, u_\infty^g-P)\right)
=\tr\left(\vrep(\rho^{-1})(1-\widehat P)\,(1-P)\right).
$$
We thus obtain \eqref{disponequa}.
\end{proof} 
Next, we link  the two projections $1-P$, $1-\widehat P$  with  the pair $\cP_\Lambda$ and $\widehat \cP_\Lambda$, associated to the cutoff parameter $\Lambda$ of \cite{Co-zeta} (see also \cite{CMbook} Chapter 2, \S 3.3), for $\Lambda=1$.

 We switch to the Hilbert space $L^2(\R)_{\rm ev}$ and use the isomorphism $w$ of \eqref{isow} to rewrite  \eqref{disponequa} in $L^2(\R)_{\rm ev}$. Let $\rep:=w^{-1}\vrep w$ be the unitary representation $\vrep$ conjugated  by the isomorphism $w$, its action  is given by\begin{equation}\label{vrepdefnadd}
(\rep(\lambda)\,\xi)(v):=\,\lambda^{-1/2}\,\xi(\lambda^{-1}\,v),\qquad \forall \xi \in
L^2(\R)_{\rm ev}\,.
\end{equation}
The projection $P$ (still denoted by the same letter), becomes the multiplication by the characteristic function of the interval complement $\{x\in \R \mid  \vert x\vert \geq1\}$  and $\widehat P$ becomes  $\fourier_{e_\R}^{-1}P\fourier_{e_\R}$.
 These projections are related to the projections $\cP_\Lambda$ and $\widehat \cP_\Lambda$, 
 by the equalities 
\begin{equation}\label{complementproj}
P=1-\cP_1, \ \ \widehat P=1-\widehat \cP_1.
\end{equation}
By implementing these notations  Lemma \ref{onequantum} states, for $\rho \geq 1$ 
\begin{equation}\label{1quantum}
	\delta(\rho)=\tr\left(\rep(\rho^{-1})\widehat \cP_1\,\cP_1\right).
\end{equation}
We recall the results explained in \opcit to understand the pair of projections $\widehat \cP_1,\,\cP_1$: we refer to \opcit for the proof
(the generalities on pairs of projections in Hilbert space follow from  Lemma 2.3 of \cite{CMbook})

\begin{lem}\label{pairofproj}
$(i)$~Giving a pair of orthogonal projections $P_i$, $i=1,2$ on a
Hilbert space $\cH$ is equivalent to giving a unitary
representation of the dihedral group $\Gamma = \Z / 2\Z \, *\, \Z /
2\Z$ (the free product of two copies of the group $\Z / 2\Z$). These
irreducible unitary representations are parameterized by an angle
$\alpha \in [0, \pi / 2]$.\newline
$(ii)$~There exists a unique operator $\alpha$, $0\leq \alpha\leq \pi / 2$, commuting with $P_i$, $i=1,2$ such that
\begin{equation}
\sin(\alpha) = \, \vert P_1 - \, P_2 \vert. \label{sinalpha}
\end{equation}
Moreover one has
\begin{equation}
P_1\,P_2\,P_1=\cos^2(\alpha) P_1   \label{cosalpha.0}\,.
\end{equation}
\end{lem}

\begin{defn}
	The operator $\alpha$ uniquely defined by \eqref{sinalpha} is called the angle operator between $P_1$ and $P_2$ and denoted:  $\angle(P_1,P_2)$. 
\end{defn}

As discovered by D. Slepian and H. Pollack \cite{Slepian,Slepian0}  there is a second-order differential operator
${\bf W}$ on $\R$, which commutes with both $\cP_{1}$ and
$\widehat{\cP}_{1}$: it is defined by
\begin{equation}
({\bf W}\xi)(x) = \,- \partial (( 1- x^2) \partial )\,
\xi(x) + (2 \pi   x)^2 \, \xi(x) .  \label{WLambdaq}
\end{equation}
The operator ${\bf W}$ commutes with the Fourier transform $\fourier_{e_\R}$ since 
$$
(\fourier_{e_\R}^{-1}\partial\fourier_{e_\R}\xi)(x)=-2 \pi i x\xi(x), \ \ \left(-\partial^2+ 
\fourier_{e_\R}^{-1}(-\partial^2)\fourier_{e_\R}\right)\xi(x)=-\partial^2\xi(x)+(2 \pi   x)^2 \, \xi(x).
$$
The angle operator $\alpha=\angle(\cP_{1},\widehat{\cP}_{1})$ fulfills
\begin{equation}
\cP_1\,\widehat \cP_1\,\cP_1=\cos^2(\alpha) \cP_1   \label{cosalpha}\,.
\end{equation}
The eigenfunctions of $\alpha
$ are determined using the prolate spheroidal wave functions $\text{\textit{PS}}_{2n,0}(2 \pi ,x )$ with bandwidth parameter $c=2 \pi$: these are  even functions  (\ie the integer index $2n$ is even in the traditional notation).   They give eigenvectors for the truncated Fourier transform according to the equality  (see \cite{Slepian0, Sl, Slepian, Rokhlin, Wang, Wang1})
\begin{equation}\label{prolateeq}
	\int_{-1}^1\text{\textit{PS}}_{2n,0}(2 \pi ,x ) \exp (i 2 \pi  x  \omega )dx=\lambda(n) \text{\textit{PS}}_{2n,0}(2 \pi ,\omega ).
\end{equation}
The eigenvalues $\lambda(n)$ are the $\lambda_{2n}^c$ ($c=2\pi$) in the notations of \cite{Wang}. They are given numerically by the list $$
\lambda(0)=0.999971,\ \lambda(1)=-0.979485,\ \lambda(2)=0.524086,\ \lambda(3)=-0.0589766, $$ $$\lambda(4)=0.00273233,\ \lambda(5)=-0.0000762914, \ldots 
$$
and all the further ones decay very fast to $0$ (see \eqref{rapid-decay}). The equality \eqref{prolateeq} means that, using the restriction $\phi_n$ of  $\text{\textit{PS}}_{2n,0}(2 \pi ,x )$ to $[-1,1]$ as an element of $\cP_1L^2(\R)_{\rm ev}$, one has 
\begin{equation}\label{cosalphan}
\cP_1\fourier_{e_\R}\cP_1\phi_n=\lambda(n)\cP_1\phi_n.
\end{equation}
In $L^2(\R)_{\rm ev}$ the Fourier transform $\fourier_{e_\R}$ is its own inverse so that  $\widehat \cP_1=\fourier_{e_\R}^{-1}\cP_1\fourier_{e_\R}=\fourier_{e_\R}\cP_1\fourier_{e_\R}$. Thus by \eqref{cosalphan} one gets:   $\cP_1\fourier_{e_\R}\cP_1\fourier_{e_\R}\cP_1\phi_n=\lambda(n)^2\cP_1\phi_n$ and 
\begin{equation}\label{cosalphan1}
\cP_1\widehat{ \cP_1}\cP_1\phi_n=\lambda(n)^2\cP_1\phi_n.
\end{equation}
By construction of the angle operator $\alpha=\angle(\cP_{1},\widehat{\cP}_{1})$ one  has  \eqref{cosalpha}, thus
it follows from \eqref{cosalphan1} that the non-zero eigenvalues $\alpha(n)$ of $\alpha$ are given by: $\cos\, \alpha(n)=\vert \lambda(n)\vert$. The sequence $\vert \lambda(n)\vert$ is of rapid decay, and one has more precisely the inequality (see \cite{Rokhlin},  Theorem 14, and Appendix \ref{appendix-cv})
\begin{equation}\label{rapid-decay}
	\vert \lambda(n)\vert\leq \frac{2^{2 n} \pi ^{2 n+\frac{1}{2}} ((2 n)!)^2}{(4 n)! \Gamma \left(2 n+\frac{3}{2}\right)}\sim (4 n+1)^{-2 n-\frac{1}{2}} (e \pi )^{2 n+\frac{1}{2}}.
\end{equation}

 Note that \eqref{cosalpha} does not determine $\alpha$ on the orthogonal complement of the linear span of the ranges of  $\cP_1$ and $\widehat \cP_1$. On this subspace, which is  Sonin's space $S(1,1)$ of Definition \ref{defnsonine}, both   $\cP_1$ and $\widehat \cP_1$ are $=0$ and thus by \eqref{sinalpha} one has  $\angle(\cP_{1},\widehat{\cP}_{1})\vert_{S(1,1)}=0$. 
 
\begin{defn}\label{defnsonine}For $\alpha,\beta>0$, let  Sonin's space $S(\alpha,\beta)\subset L^2(\R)_{\rm ev}$ be defined by\begin{equation}\label{sonine}S(\alpha,\beta):=\{\xi\in L^2(\R)_{\rm ev}\mid \xi(q)=0, \ \forall q, \vert q \vert \leq \alpha, \  (\fourier_{e_\R}\xi)(p)=0, \ \forall p,  \vert p \vert \leq \beta\}.\end{equation}	\end{defn}

Given two vectors $\xi,\eta\in \cH$ in a Hilbert space $\cH$ we shall use the Dirac notation
$\vert \xi \rangle \langle \eta \vert$ for the rank one projection 
\begin{equation}\label{rankone}
\vert \xi \rangle \langle \eta \vert(\eta'):= \xi \, \langle \eta\mid \eta'\rangle,\qquad \forall \eta'\in \cH,
\end{equation}
where by convention the inner product $\langle \eta\mid \eta'\rangle$ is anti-linear in the first vector and linear in $\eta'$.

\begin{prop}\label{devil0} $(i)$~Let $\xi_n=\cP_1\phi_n/\Vert\cP_1\phi_n\Vert$,  $\eta_n=\fourier_{e_\R}\xi_n$ and $\psi_n=P\eta_n$.  One has \begin{equation}\label{chirem0.5}\cP_1\eta_n=\cP_1\fourier_{e_\R}\xi_n=\lambda(n)\xi_n.
\end{equation}
$(ii)$~For $\rho\geq 1$
\begin{equation}\label{chirem1}
\delta(\rho)=\tr\left(\rep(\rho^{-1})\widehat \cP_1\,\cP_1\right)=\sum_n\left( \lambda(n)^2\langle \xi_n\mid \rep(\rho^{-1})\xi_n\rangle +\lambda(n)\langle \xi_n\mid \rep(\rho^{-1})\psi_n\rangle\right)
\end{equation}
$(iii)$~The functions $\psi_n=P\fourier_{e_\R}\xi_n$ are real valued, pairwise orthogonal and: $\Vert \psi_n\Vert=\sqrt{1-\lambda(n)^2}$, 
	\begin{equation}\label{smaller}
\sum \lambda(n)^2 \vert \zeta_n \rangle \langle \zeta_n \vert\leq P\widehat P P, \qquad \zeta_n=\frac{1}{\sqrt{1-\lambda(n)^2}}\psi_n
\end{equation}
 $(iv)$~Let $\tau(n):=\frac{\lambda(n)}{\sqrt{1-\lambda(n)^2}}$, one has 
\begin{equation}\label{hattrick}
\langle \xi_n \mid \rep(\rho^{-1})\xi_n\rangle=
\langle \zeta_n \mid  \rep(\rho^{-1})\zeta_n\rangle+\tau(n)(\langle \xi_n \mid \rep(\rho^{-1})\zeta_n)\rangle+\langle \zeta_n\mid  \rep(\rho^{-1})\xi_n\rangle).
\end{equation}	
\end{prop}
\begin{proof} $(i)$~This follows from \eqref{cosalphan}.\newline
 $(ii)$~The vectors $\xi_n$ form an orthonormal basis of the range of $\cP_1$ so that 
$$
\cP_1=\sum  \vert \xi_n \rangle \langle \xi_n \vert, \ \ \widehat \cP_1\,\cP_1=\sum  \vert \widehat \cP_1\xi_n \rangle \langle \xi_n \vert
$$
and, using the equality \eqref{cosalphan}, $\cP_1\fourier_{e_\R}\cP_1\phi_n=\lambda(n)\cP_1\phi_n$, one has 
\begin{equation}\label{chirem1.5}
\widehat \cP_1\xi_n=\fourier_{e_\R}\cP_1\fourier_{e_\R}\xi_n=\lambda(n)\fourier_{e_\R}\xi_n=\lambda(n)\eta_n
\end{equation}
which thus gives 
\begin{equation}\label{chirem2}
\widehat \cP_1\,\cP_1=\sum \lambda(n) \vert \eta_n \rangle \langle \xi_n \vert.
\end{equation}
 Hence we obtain using \eqref{1quantum} and \eqref{chirem2}
$$
\delta(\rho)=\tr\left(\rep(\rho^{-1})\widehat \cP_1\,\cP_1\right)=\sum \lambda(n)\langle \xi_n\mid \rep(\rho^{-1})\eta_n\rangle.
$$
One has the orthogonal sum:  $\eta_n=\fourier_{e_\R}\xi_n=P\fourier_{e_\R} \xi_n+\cP_1\fourier_{e_\R} \xi_n$ and $\cP_1\fourier_{e_\R}\xi_n=\lambda(n)\xi_n$. Thus $
\eta_n=\psi_n+\lambda(n)\xi_n
$
which gives \eqref{chirem1} and
\begin{equation}\label{chirem3}
\widehat \cP_1\,\cP_1=\sum \lambda(n)^2 \vert \xi_n \rangle \langle \xi_n \vert+\sum \lambda(n) \vert \psi_n \rangle \langle \xi_n \vert.
\end{equation}
 $(iii)$~One has 
$$
\Vert P\fourier_{e_\R} \xi_n\Vert^2+\Vert \cP_1\fourier_{e_\R} \xi_n\Vert^2=\Vert \fourier_{e_\R} \xi_n\Vert^2=\Vert  \xi_n\Vert^2=1
$$
$$
\Vert  \psi_n\Vert^2=\Vert P\fourier_{e_\R} \xi_n\Vert^2=1-\lambda(n)^2, \ \ \Vert  \psi_n\Vert=\sqrt{1-\lambda(n)^2}.
$$
For $n\neq m$ one has, by unitarity of the Fourier transform $\fourier_{e_\R}$,
$$
0=\langle \fourier_{e_\R} \xi_n,\fourier_{e_\R} \xi_m\rangle= \langle \eta_n,\eta_m\rangle=\langle \psi_n+\lambda(n)\xi_n,\psi_m+\lambda(m)\xi_m\rangle
=\langle \psi_n,\psi_m\rangle
$$
so that the $\zeta_n=\frac{1}{\sqrt{1-\lambda(n)^2}}\psi_n$ form an orthonormal family. 
By \eqref{complementproj} the two dimensional irreducible representations $\Pi_n$ of the infinite dihedral group with generators 
$$
\cU=1-2\cP_1, \qquad \cV=1-2 \widehat \cP_1
$$
are the same if one replaces the pair of projections $(\cP_1,\widehat \cP_1)$ by the pair $(P,\widehat P)$ and the angle operator is the same: $\angle(\cP_1,\widehat \cP_1)=\angle(P,\widehat P)$. Let $E_n$ be the two dimensional eigenspace associated to  $\Pi_n$
$$
E_n:=\{\xi \in L^2(\R)_{\rm ev}\mid \vert  \cP_1-\widehat \cP_1\vert(\xi)=\sqrt{1-\lambda(n)^2}\ \ \xi\}.
$$
Let us show that $\xi_n\in E_n$. The operator $\cP_1\widehat \cP_1\cP_1$ is positive and when restricted to the range of $\cP_1$ it has simple spectrum with eigenvalues 
$\lambda(n)^2$. By \eqref{cosalphan} the eigenvectors are the $\xi_n$ but the one dimensional space $\cP_1 E_n\subset E_n$ is also an eigenspace of the operator $\cP_1\widehat \cP_1\cP_1$ for the same eigenvalue $\lambda(n)^2$. Thus one gets 
$$
\xi_n\in \cP_1 E_n\subset E_n.
$$
It follows from \eqref{chirem1.5} that 
$$
\eta_n=\lambda(n)^{-1}\widehat \cP_1\xi_n\in \widehat\cP_1 E_n\subset E_n
$$
and since $\eta_n=\psi_n+\lambda(n)\xi_n$ we get that $\psi_n\in E_n$. It follows that 
$\zeta_n$ is a normalized eigenvector for the angle operator which gives 
$$
P\widehat P P \zeta_n=\lambda(n)^2 P \zeta_n=\lambda(n)^2 \zeta_n.
$$
The spectral decomposition of the positive operator $P\widehat P P$ is of the form 
\begin{equation}\label{spectral}
	P\widehat P P=\sum \lambda(n)^2 \vert \zeta_n \rangle \langle \zeta_n \vert+R
\end{equation}
where $R$ is the restriction of $P\widehat P P$ to the orthogonal complement of the subspace $\oplus_n E_n\subset L^2(\R)_{\rm ev}$. The operator $R$ is the orthogonal projection on  Sonin's space $S(1,1)$ of Definition \ref{defnsonine}. Note that by construction the vectors $\xi_n$  are all orthogonal to $S(1,1)$ and so are  $\eta_n=\fourier_{e_\R} \xi_n$ and $\zeta_n= \frac{1}{\sqrt{1-\lambda(n)^2}}(\eta_n-\lambda(n)\xi_n)$.\newline
$(iv)$~We use $\rep(\rho)=\fourier_{e_\R}^{-1} \rep(\rho^{-1})\fourier_{e_\R}$ and the fact that $\xi_n$ is a real valued function, to get 
\begin{align*}
\langle \xi_n\mid \rep(\rho^{-1})\xi_n\rangle &=\langle \rep(\rho) \xi_n\mid \xi_n\rangle=\langle \xi_n \mid 
\rep(\rho)\xi_n\rangle =\langle \xi_n\mid 
\fourier_{e_\R}^{-1} \rep(\rho^{-1})\fourier_{e_\R}\xi_n\rangle=\langle\eta_n \mid \rep(\rho^{-1})\eta_n\rangle=\\
&=\langle (\psi_n+\lambda(n)\xi_n) \mid\rep(\rho^{-1})(\psi_n+\lambda(n)\xi_n) \rangle
=\lambda(n)^2\langle \xi_n\mid \rep(\rho^{-1}) \xi_n\rangle+\\
&+\langle \psi_n\mid \rep(\rho^{-1}) \psi_n\rangle+\langle \lambda(n)\xi_n\mid \rep(\rho^{-1})\psi_n\rangle+\langle  \psi_n\mid \rep(\rho^{-1})\lambda(n)\xi_n\rangle.
\end{align*}
Thus one gets
$$
(1-\lambda(n)^2)\langle \xi_n\mid \rep(\rho^{-1}) \xi_n\rangle=
\langle \psi_n\mid \rep(\rho^{-1}) \psi_n\rangle+\lambda(n)\left(\langle  \xi_n \mid \rep(\rho^{-1})\psi_n\rangle+\langle  \psi_n\mid \rep(\rho^{-1})\xi_n\rangle \right)
$$
and since $\zeta_n=\frac{1}{\sqrt{1-\lambda(n)^2}}\psi_n$ one obtains, with $\tau(n):=\frac{\lambda(n)}{\sqrt{1-\lambda(n)^2}}$
$$
\langle \xi_n\mid \rep(\rho^{-1})\xi_n\rangle=
\langle \zeta_n\mid \rep(\rho^{-1})\zeta_n\rangle+\tau(n)(\langle \xi_n\mid \rep(\rho^{-1})\zeta_n\rangle+\langle  \zeta_n\mid \rep(\rho^{-1})\xi_n\rangle)
$$
which is \eqref{hattrick}. \end{proof} 

\begin{rem}\label{sym}
$(i)$~Equality \eqref{chirem1} implies in particular that $\delta(1)=\sum \lambda(n)^2$, \ie 
$$
2 \left(\frac{\text{Si}(4 \pi )}{4 \pi }+1\right)=\sum \lambda(n)^2
$$
and one checks numerically that both sides\footnote{We are only using the even prolate functions, the sum of squares of eigenvalues including the odd ones is $4$.} are $\sim 2.237484835$. One has, by \eqref{chirem0.5}, $\cP_1\fourier_{e_\R}\xi_n=\lambda(n)\xi_n$ which implies that $\psi_n(1)=\fourier_{e_\R}\xi_n(1)=\lambda(n)\xi_n(1)$. The function $\fourier_{e_\R}\xi_n$ is smooth, as the Fourier transform of a function with compact support. The derivative, at $\rho=1^+$, of the function $\sum \lambda(n)\langle  \xi_n\mid \rep(\rho^{-1})\psi_n\rangle$ which appears in \eqref{chirem1}, is  $\sum \lambda(n)\psi_n(1)\xi_n(1)$. This is thus equal to $\sum \lambda(n)^2\xi_n(1)^2$ which is numerically $\sim 2$. To prove that it is equal to $2$, note that the derivative of $\delta(\rho)$ at $\rho=1^+$ is equal to $1$ by \eqref{sch18.5}, and that the derivative of $\langle \xi_n\mid \rep(\rho^{-1})\xi_n\rangle$ is (using $\xi$ instead of $\xi_n$)
$$
\partial_\rho\left( \rho^{1/2} \int_0^{\rho^{-1}}\xi(x)\xi(\rho x)dx \right)=
\sqrt{\rho } \left(\int_0^{\frac{1}{\rho }} x \xi (x) \xi '(\rho  x) \, dx-\frac{\xi (1) \xi \left(\frac{1}{\rho }\right)}{\rho ^2}\right)+\frac{\int_0^{\frac{1}{\rho }} \xi (x) \xi (\rho  x) \, dx}{2 \sqrt{\rho }},
$$
which for $\rho=1$ gives, using integration by parts, 
$$
-\xi (1)^2+\int_0^1 x \xi (x) \xi '(x) \, dx+\frac{1}{2} \int_0^1 \xi (x)^2 \, dx
=-\frac{1}{2}\xi (1)^2.
$$
It follows that the contribution of the sum $\sum \lambda(n)^2\langle \xi_n\mid \rep(\rho^{-1}) \xi_n\rangle$ to the derivative at $\rho=1^+$ is $-\frac{1}{2}\sum \lambda(n)^2\xi_n(1)^2$ so that \eqref{chirem1} implies:  $\sum \lambda(n)^2\xi_n(1)^2=2$.
\newline 
$(ii)$~Note that for $\rho\geq 1$ one has
\begin{equation}\label{sym1}
\langle   \psi_n\mid \rep(\rho^{-1})\xi_n\rangle=0, \ \langle  \zeta_n\mid \rep(\rho^{-1})\xi_n\rangle=0,
\end{equation}	
since $\xi_n(\rho x)=0$ for $\vert x\vert >1$ so that  
$$
\langle   \psi_n\mid \rep(\rho^{-1})\xi_n\rangle=\rho^{1/2}\int \xi_n(\rho x)\psi_n(x)dx=0.
$$
	Thus one can rewrite \eqref{chirem1} in a more symmetric manner replacing the term 
	$\langle  \xi_n\mid \rep(\rho^{-1})\psi_n\rangle$ (multiplied by $\lambda(n)$) with the symmetric form (using the fact that the $\xi_n, \psi_n$  are real valued)
	$$
	\langle  \xi_n\mid \rep(\rho^{-1})\psi_n\rangle+\langle  \psi_n\mid \rep(\rho^{-1})\xi_n\rangle=\langle  \xi_n\mid \rep(\rho^{-1})\psi_n\rangle+\langle \xi_n\mid \rep(\rho)\psi_n\rangle	$$
	which is invariant under $\rho\mapsto \rho^{-1}$ so that, after this replacement, \eqref{chirem1} is valid for all $\rho \in \R_+^*$.
\end{rem}

Next theorem refines the local trace formula  \eqref{sch13}

\begin{thm}\label{devil} Let $\bf S$  be the orthogonal projection of $L^2(\R)_{\rm ev}$ on the closed subspace $S(1,1)$. The following functional is  positive  
\begin{equation}\label{sonine1}
\tr(\rep(f){\bf S})=W_\infty(f)+\int f(\rho^{-1})\epsilon(\rho) d^*\rho,\qquad \forall f \in C_c^\infty(\R_+^*),
\end{equation}	
where $W_\infty$ is as in \eqref{sch22}, $\epsilon(\rho)$ is the function of $\rho\in \R_+^*$, with $\epsilon(\rho^{-1})=\epsilon(\rho)$, which is given, for $\rho \geq 1$, by 
\begin{equation}\label{sonine0}
\epsilon(\rho)=\sum \frac{\lambda(n)}{\sqrt{1-\lambda(n)^2}}\langle \xi_n\mid  \rep(\rho^{-1})\zeta_n\rangle.
\end{equation}
\end{thm}
\begin{proof} By \eqref{sch13} one has, rewriting the formula in $L^2(\R)_{\rm ev}$,  
\begin{equation}\label{sch13bis}
W_\infty(f)+\int f(\rho^{-1})\delta(\rho)d^*\rho=\tr\left(\rep(f)P\widehat PP\right).\end{equation}
By \eqref{chirem1} one has, for $\rho\geq 1$, 
\begin{equation}\label{chirem1bis}
\delta(\rho)=\sum\left( \lambda(n)^2\langle   \xi_n\mid \rep(\rho^{-1})\xi_n\rangle +\lambda(n)\sqrt{1-\lambda(n)^2}\langle  \xi_n\mid \rep(\rho^{-1})\zeta_n\rangle\right).
\end{equation}
By \eqref{hattrick}, one has 
\begin{equation}\label{hattrickbis}
\langle \xi_n \mid \rep(\rho^{-1})\xi_n\rangle=
\langle \zeta_n \mid  \rep(\rho^{-1})\zeta_n\rangle+\tau(n)(\langle \xi_n \mid \rep(\rho^{-1})\zeta_n)\rangle+\langle \zeta_n\mid  \rep(\rho^{-1})\xi_n\rangle).
\end{equation}
This gives for $\rho\geq 1$, 
\begin{equation}\label{chirem1ter}
\delta(\rho)=\sum \lambda(n)^2\langle  \zeta_n\mid \rep(\rho^{-1})\zeta_n\rangle +\sum  T_n,
\end{equation}
where 
$$
T_n=\lambda(n)\sqrt{1-\lambda(n)^2}\langle  \xi_n\mid \rep(\rho^{-1})\zeta_n\rangle
+\lambda(n)^2\tau(n)(\langle  \xi_n\mid \rep(\rho^{-1})\zeta_n)\rangle+\langle  \zeta_n  \mid\rep(\rho^{-1})\xi_n\rangle).
$$
Since $\rho\geq 1$ one has, by \eqref{sym1}:
$
\langle  \zeta_n\mid \rep(\rho^{-1})\xi_n\rangle=0
$.
Thus one gets, using\footnote{Note that for $-1< \lambda <1$ one has $\lambda\sqrt{1-\lambda^2}+\lambda^2\frac{\lambda}{\sqrt{1-\lambda^2}}=\frac{\lambda}{\sqrt{1-\lambda^2}}$} $\tau(n)=\frac{\lambda(n)}{\sqrt{1-\lambda(n)^2}}$  
$$
T_n=\frac{\lambda(n)}{\sqrt{1-\lambda(n)^2}}\langle   \xi_n\mid \rep(\rho^{-1})\zeta_n\rangle.
$$
Thus we obtain,  for $\rho\geq 1$  
\begin{equation}\label{chirem1terbis}
\delta(\rho)=\sum \lambda(n)^2\langle  \zeta_n\mid \rep(\rho^{-1})\zeta_n\rangle +\epsilon(\rho).
\end{equation}
Since both $\delta$, $\epsilon$ and the terms $\langle  \zeta_n\mid \rep(\rho^{-1})\zeta_n\rangle$ are invariant under $\rho\mapsto \rho^{-1}$, we thus obtain, for any test function $f\in C_c^\infty(\R_+^*)$
$$
\int f(\rho^{-1})\delta(\rho)d^*\rho=\int f(\rho^{-1})\epsilon(\rho)d^*\rho+\sum \lambda(n)^2\langle  \zeta_n\mid \rep(f)\zeta_n\rangle.
$$ 
Using \eqref{sch13bis} we conclude that 
$$
W_\infty(f)+\int f(\rho^{-1})\epsilon(\rho)d^*\rho+\sum \lambda(n)^2\langle \zeta_n\mid \rep(f)\zeta_n\rangle=\tr\left(\rep(f)P\widehat PP\right).
$$
This gives the required formula \eqref{sonine0} provided we prove that 
\begin{equation}\label{hattrickter}
\tr\left(\rep(f)P\widehat PP\right)=\tr(\rep(f){\bf S})+\sum \lambda(n)^2\langle \zeta_n\mid \rep(f)\zeta_n\rangle.
\end{equation}
This follows from the spectral decomposition \eqref{spectral} of the operator  $P\widehat PP$ since   Sonin's space $S(1,1)$ is the eigenspace of $P\widehat PP$ for the eigenvalue $1$, so that $R=\bf S$. \end{proof} 

\begin{figure}[H]	\begin{center}
\includegraphics[scale=0.5]{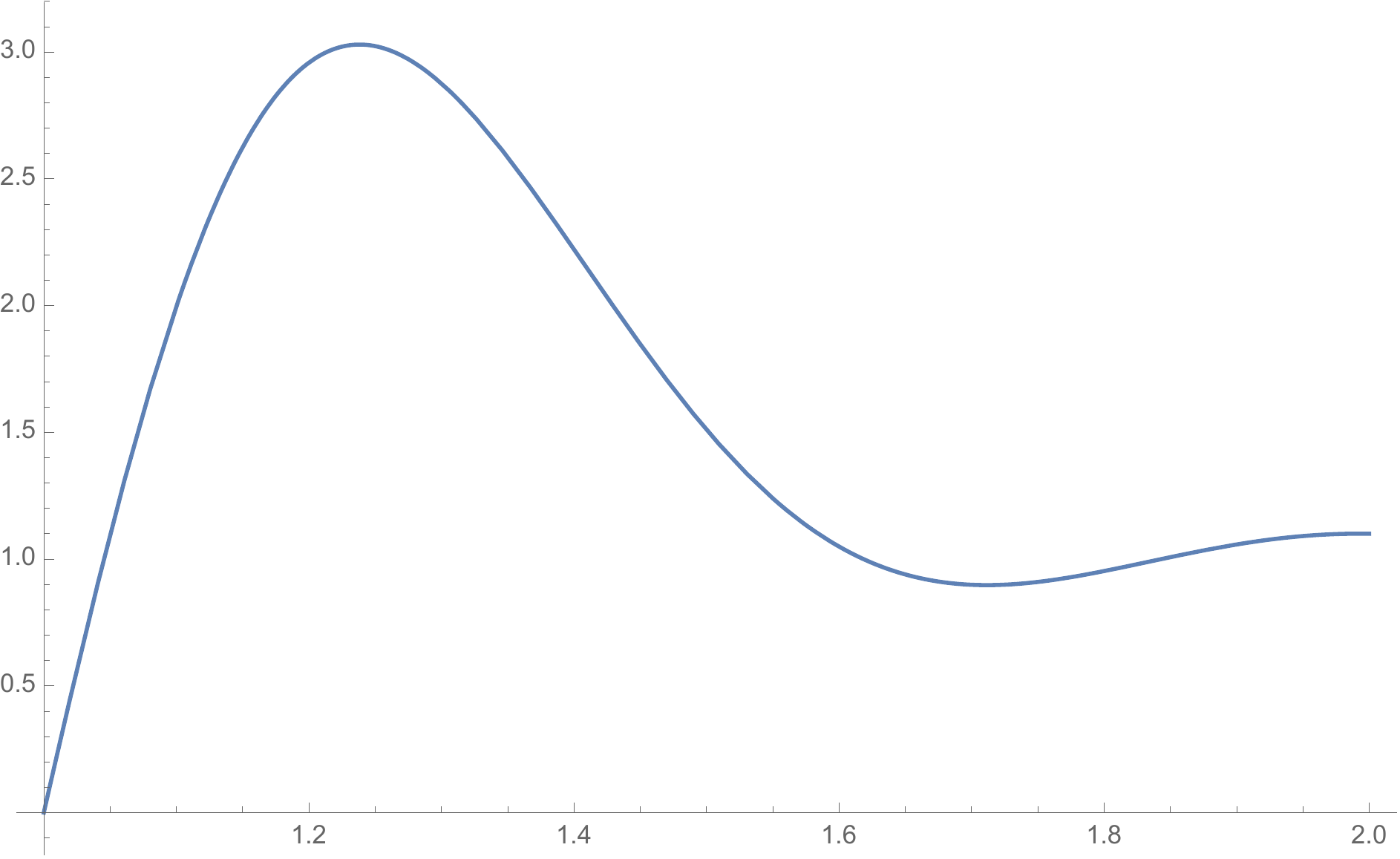}
\end{center}
\caption{Graph of $\epsilon(\rho)$ in $[1, 2]$. \label{epsilongraph} }
\end{figure}

\section{The functional $E\circ Q$ and the compact operator $\kf_I$}\label{sectcompactK}
Let $E$ be the functional defined on $C_c^\infty(\R_+^*)$ by 
\begin{equation}\label{Efunctional} 
	E(f):=\int f(\rho^{-1})\epsilon(\rho) d^*\rho,\qquad  \forall f \in C_c^\infty(\R_+^*).
\end{equation}
By Theorem \ref{devil} the functional $W_\infty+E$ is positive. Thus, to control $W_\infty$ after imposing the vanishing conditions we need, by Proposition \ref{vanishing2}, to analyse the functional $E\circ Q$,  with $Q$ as in \eqref{Qop}. By \eqref{sonine0} the function $\epsilon(\rho)$ is a sum of coefficients of the representation of the multiplicative group $\R_+^*$ by scaling. Thus we start by investigating how the operator $Q$ acts on such coefficients. We first assume (Lemma \ref{devil1}) that the vectors involved are smooth functions and then compute in Lemma \ref{devil3} the boundary terms in the case of the vectors involved in \eqref{sonine0}.

The operator $\cD:=D_u^2+D_u$, where $D_u(f)x:=xf'(x)$ is the scaling Hamiltonian, commutes with the Fourier transform $\fourier_{e_\R}$ since
		 $$
		 \fourier_{e_\R}^{-1}D_u \fourier_{e_\R}\xi(x)=-\partial_x(x\xi(x))=((-1-D_u)\xi)(x).
		 $$
		 One  has 
		\begin{equation}\label{kahane}
	\cD(f)(x)= x^2 f''(x)+2 xf'(x)
		\end{equation}
thus $\cD(f)(0)=0$. Since $\cD$ commutes with the Fourier transform  the range $\cD(\cS(\R))$ consists of functions  fulfilling the two conditions
\begin{equation}\label{twocond}
	f(0)=\hat f(0)=0.
		\end{equation}
 Note also that, for the inner product on $L^2(\R)_{\rm ev}$, the adjoint of $D_u$ is $D_u^*(f)x):=-\partial_x(xf(x))$ and thus $\cD^*=\cD$.
 
\begin{lem}\label{devil1} Let $\xi, \eta \in  L^2(\R)_{\rm ev}$ be smooth functions. Then with $Q$ as in \eqref{Qop} and $\cD:=D_u^2+D_u$, one has
\begin{equation}\label{sonineconst}
 \langle  \eta\mid \rep(Q f)\xi\rangle= -\langle   \eta\mid \rep(f)\cD\xi\rangle,\qquad \forall f \in C_c^\infty(\R_+^*).
\end{equation}	
\end{lem}
\begin{proof} Using the self-adjointness of  $Q$ for the inner product in $L^2(\R_+^*,d^*\rho)$ together with the commutation of $Q$ with the inversion $I$, one has
$$
\langle   \eta\mid \rep(Q f)\xi\rangle=\int Qf(\rho^{-1})\langle   \eta\mid \rep(\rho^{-1})\xi\rangle d^*\rho=\int f(\rho^{-1})(Qk)(\rho)d^*\rho,
$$
where the function $k(\rho)$ is defined as 
$$
k(\rho):=\langle  \eta\mid \rep(\rho^{-1})\xi\rangle=\rho^{1/2}\int \xi(\rho x)\overline{\eta(x)}dx.
$$
To obtain a formula for $Qk$, we apply $Q$ to the function $g(\rho):=\rho^{1/2}\xi(\rho x)$ and use the equality
$$
\rho\partial_\rho\left( \rho^{1/2}\xi(\rho x)\right)=\frac 12\left( \rho^{1/2}\xi(\rho x)\right)+ \rho^{1/2}(D_u\xi)(\rho x)
$$
which holds since 
$$
\rho\partial_\rho\left(\xi(\rho x)\right)=\rho x\,\xi'(\rho x)=(D_u\xi)(\rho x).
$$
We thus obtain
\begin{align*}
\rho\partial_\rho\left(\rho\partial_\rho\left(\rho^{1/2}\xi(\rho x)\right)\right)&=
\rho\partial_\rho\left(\frac 12\left( \rho^{1/2}\xi(\rho x)\right)\right)
+\rho\partial_\rho\left(\rho^{1/2}(D_u\xi)(\rho x) \right)=\\
&=\frac 14\left( \rho^{1/2}\xi(\rho x)\right)+\frac 12 \rho^{1/2}(D_u\xi)(\rho x)+\rho\partial_\rho\left(\rho^{1/2}(D_u\xi)(\rho x) \right)=\\
&=\frac 14\left( \rho^{1/2}\xi(\rho x)\right)+\frac 12 \rho^{1/2}(D_u\xi)(\rho x)+\frac 12 \rho^{1/2}(D_u\xi)(\rho x)+\rho^{1/2}(D_u^2\xi)(\rho x).
\end{align*}
Thus, using \eqref{Qop}, we have 
$$
Qg(\rho)=-(\rho\partial_\rho)^2g(\rho)+\frac 14 g(\rho)=-\rho^{1/2}(D_u^2\xi)(\rho x)-\rho^{1/2}(D_u\xi)(\rho x)=-\rho^{1/2}(\cD\xi)(\rho x).
$$
This gives in turn
$$
Qk(\rho)=-\rho^{1/2}\int (\cD\xi)(\rho x)\overline{\eta(x)}dx=
-\langle  \eta\mid \rep(\rho^{-1})(\cD\xi)\rangle
$$
and 
$$
\langle   \eta\mid \rep(Qf)\xi\rangle=\int f(\rho^{-1})(Qk)(\rho)d^*\rho
=-\int f(\rho^{-1})\langle  \eta\mid \rep(\rho^{-1})(\cD\xi)\rangle  d^*\rho=-\langle  \eta\mid \rep(f)\cD\xi\rangle.
$$
which proves \eqref{sonineconst}. \end{proof} 

One has $D_u^*=-1-D_u$,  thus $-\cD=D_uD_u^*=D_u^*D_u$. Moreover $D_u$ commutes with the scaling action, hence with $\rep(f)$, thus one obtains, at the formal level
\begin{equation}\label{sonineconst1}
 \langle  \eta\mid \rep(Qf)\xi\rangle= \langle  D_u \eta\mid \rep(f)D_u\xi\rangle= \langle  D_u^*\eta\mid \rep(f)D_u^*\xi\rangle.
 \end{equation}
 In principle, we would like to apply \eqref{sonineconst1} to  \eqref{sonine0} and get a formula for $E \circ Q$, but the functions $\xi_n$ and $\zeta_n$ do not belong to the domain of the operator $D_u$, since both have a discontinuity at $x=\pm 1 $. This means that some boundary terms appear and 
we are going to compute them. Since we handle even functions, we concentrate on the positive real axis and  we only consider the scaling parameter $\rho \in [1,2]$.  

\begin{lem}\label{devil3} Let $\xi\in C^\infty((0,1])$ and  $\zeta\in C^\infty([1,\infty))$ be smooth and real valued functions. Extend first $\xi,\zeta$ to $[0,\infty)$  as follows: $\xi(x):=0$ for $x>1$ and $\zeta(x):=0$ for $x<1$. Then, with $Q$ as in \eqref{Qop} and $k(\rho)=\rho^{1/2}\int_0^\infty\xi(x)\zeta(\rho x)dx$, for $\rho \in (1,2]$,  one has 
\begin{equation}\label{sonineconst3}
(Qk)(\rho)=\rho^{1/2}\int_{\rho^{-1}}^1(D_u\xi)(x)(D_u\zeta)(\rho x)dx+\rho^{-1/2}(D_u\xi)(\rho^{-1})\zeta(1)-\rho^{1/2}\xi(1)(D_u\zeta)(\rho).
\end{equation}	
\end{lem}
\begin{proof} One has $k(\rho)=\rho^{1/2}\int_{\rho^{-1}}^1\xi(x)\zeta(\rho x)dx$, since $\xi(x)\zeta(\rho x)=0$ for $x\notin [\rho^{-1},1]$. Furthermore,
\begin{align*}
(\rho\partial_\rho)\left(\rho^{1/2}\int_{\rho^{-1}}^1\xi(x)\zeta(\rho x)dx\right)&=
\frac 12\left(\rho^{1/2}\int_{\rho^{-1}}^1\xi(x)\zeta(\rho x)dx\right)+ \rho^{-1/2}\xi(\rho^{-1})\zeta(1)+\\
&+\left(\rho^{1/2}\int_{\rho^{-1}}^1\xi(x)(D_u\zeta)(\rho x)dx\right),
\end{align*}
using for the computation of $\rho^{1/2}\rho\partial_\rho\left(\int_{\rho^{-1}}^1\xi(x)\zeta(\rho x)dx\right)$ the equality
\begin{align*}
&\int_{(\rho+\epsilon)^{-1}}^{1}\xi(x)\zeta((\rho+\epsilon) x)dx-\int_{\rho^{-1}}^1\xi(x)\zeta(\rho x)dx=\\
&=\int_{(\rho+\epsilon)^{-1}}^{\rho^{-1}}\xi(x)\zeta((\rho+\epsilon) x)dx+
\int_{\rho^{-1}}^1\xi(x)(\zeta((\rho+\epsilon) x)-\zeta(\rho x))dx.
\end{align*}
By continuity of $\zeta$ in $[1,2]$ one has, 
$$
\frac 1 \epsilon \int_{(\rho+\epsilon)^{-1}}^{\rho^{-1}}\xi(x)\zeta((\rho+\epsilon) x)dx
\stackbin[\epsilon \to 0^+]{\longrightarrow}{} \rho^{-2}\xi(\rho^{-1})\zeta(1).
$$ 
Iterating this formula one obtains, for $k(\rho)=\rho^{1/2}\int_{\rho^{-1}}^1\xi(x)\zeta(\rho x)dx$ 
$$
(\rho\partial_\rho)^2k(\rho)=
\frac 12(\rho\partial_\rho)k(\rho)+ (\rho\partial_\rho)\left(\rho^{-1/2}\xi(\rho^{-1})\zeta(1)\right)
+(\rho\partial_\rho)\left(\rho^{1/2}\int_{\rho^{-1}}^1\xi(x)(D_u\zeta)(\rho x)dx\right)
$$
$$
=\frac 14\left(\rho^{1/2}\int_{\rho^{-1}}^1\xi(x)\zeta(\rho x)dx\right)+\frac 12 \rho^{-1/2}\xi(\rho^{-1})\zeta(1)+\frac 12\left(\rho^{1/2}\int_{\rho^{-1}}^1\xi(x)(D_u\zeta)(\rho x)dx\right)+
$$
$$
+(\rho\partial_\rho)\left(\rho^{-1/2}\xi(\rho^{-1})\zeta(1)\right)+\frac 12\left(\rho^{1/2}\int_{\rho^{-1}}^1\xi(x)(D_u\zeta)(\rho x)dx\right)+ \rho^{-1/2}\xi(\rho^{-1})(D_u\zeta)(1)+
$$
$$
+\left(\rho^{1/2}\int_{\rho^{-1}}^1\xi(x)(D_u^2\zeta)(\rho x)dx\right).
$$
This gives 
$$
(Qk)(\rho)=-\rho^{1/2}\int_{\rho^{-1}}^1\xi(x)((D_u^2+D_u)\zeta)(\rho x)dx-B,
$$
where the boundary term is 
$$
B=\frac 12 \rho^{-1/2}\xi(\rho^{-1})\zeta(1)+(\rho\partial_\rho)\left(\rho^{-1/2}\xi(\rho^{-1})\zeta(1)\right)+\rho^{-1/2}\xi(\rho^{-1})(D_u\zeta)(1).
$$
Using the equality
$$
(\rho\partial_\rho)\left(\rho^{-1/2}\xi(\rho^{-1})\zeta(1)\right)=-\frac 12 \rho^{-1/2}\xi(\rho^{-1})\zeta(1)-\rho^{-3/2}\xi'(\rho^{-1})\zeta(1)
$$
one obtains
$$
B=\rho^{-1/2}\xi(\rho^{-1})(D_u\zeta)(1)-\rho^{-1/2}(D_u\xi)(\rho^{-1})\zeta(1).
$$
The integration by parts, which gives the adjoint of $D_u$ as $-1-D_u$, is 
$$
\int_a^b D_u(f)(x)g(x)dx+\int_a^b f(x)((1+D_u)g)(x)dx=bf(b)g(b)-af(a)g(a)
$$
as can be seen by differentiating the product $(xf(x)g(x))'=xf'(x)g(x)+f(x)(g(x)+xg'(x))$.
We apply it with $f=\xi$ and $g(x)=D_u\zeta(\rho x)$ and use the commutation of $D_u$ with scaling to get
$$
\int_{\rho^{-1}}^1(D_u\xi)(x)(D_u\zeta)(\rho x)dx+\int_{\rho^{-1}}^1\xi(x)((D_u^2+D_u)\zeta)(\rho x)dx=\xi(1)(D_u\zeta)(\rho)-\rho^{-1}\xi(\rho^{-1})(D_u\zeta)(1).
$$
Hence we obtain 
\begin{align*}
(Qk)(\rho)&=\rho^{1/2}\int_{\rho^{-1}}^1(D_u\xi)(x)(D_u\zeta)(\rho x)dx-\rho^{1/2}\left(\xi(1)(D_u\zeta)(\rho)-\rho^{-1}\xi(\rho^{-1})(D_u\zeta)(1)\right) -B=\\
&=\rho^{1/2}\int_{\rho^{-1}}^1(D_u\xi)(x)(D_u\zeta)(\rho x)dx+\rho^{-1/2}(D_u\xi)(\rho^{-1})\zeta(1)-\rho^{1/2}\xi(1)(D_u\zeta)(\rho).
\end{align*}
which is the required formula. \end{proof} 

We thus derive the following formula for $Q\epsilon(\rho)$

\begin{prop}\label{propQe}
	For $\rho>1$ one has the equality \begin{equation}\label{qe}
Q\epsilon(\rho)=\sum \frac{\lambda(n)}{\sqrt{1-\lambda(n)^2}} T_n(\rho),\end{equation}
where 
\begin{equation}\label{sonineQ}
T_n(\rho)=\rho^{1/2}\int_{\rho^{-1}}^1(D_u\xi_n)(x)(D_u\zeta_n)(\rho x)dx+\rho^{-1/2}(D_u\xi_n)(\rho^{-1})\zeta_n(1)-\rho^{1/2}\xi_n(1)(D_u\zeta_n)(\rho) \end{equation}
\end{prop}
\begin{proof} This follows from \eqref{sonine0} combined with Lemma \ref{devil3} and by recalling the normalization \eqref{innerltwoeven} of the inner product in $L^2(\R)_{\rm ev}$.
We refer to Appendix \ref{appendix-cv} for the proof of the convergence of the infinite series and for an explicit control of the remainder.\end{proof} 

When implementing \eqref{sonineQ} in a computer program, it is convenient to use the function $\xi_n^{\rm an}$ which is the analytic continuation of $\xi_n$. By \eqref{chirem0.5} one has  $\eta_n=\lambda(n)\,\xi_n^{\rm an}$,  thus  for $x\in [1,\infty)$, one derives
$$
\zeta_n(x)=\frac{1}{\sqrt{1-\lambda(n)^2}}\,\eta_n(x)=\frac{\lambda(n)}{\sqrt{1-\lambda(n)^2}}\xi_n^{\rm an}(x).
$$
Combining this with Proposition \ref{propQe}, one obtains the equality $
Q\epsilon(\rho)=\sum \frac{\lambda(n)^2}{1-\lambda(n)^2} C_n$,
where 
$$
C_n=\rho^{1/2}\int_{\rho^{-1}}^1(D_u\xi_n^{\rm an})(x)(D_u\xi_n^{\rm an})(\rho x)dx+\rho^{-1/2}(D_u\xi_n^{\rm an})(\rho^{-1})\xi_n^{\rm an}(1)-\rho^{1/2}\xi_n^{\rm an}(1)(D_u\xi_n^{\rm an})(\rho).
$$
Using the equality $D_u(f(x)=xf'(x)$ one gets the following formula for $C_n$
\begin{equation}\label{sonineQbis}
C_n=\rho^{1/2}\int_{\rho^{-1}}^1 x(\xi_n^{\rm an})'(x)\rho x(\xi_n^{\rm an})'(\rho x)dx+\rho^{-3/2}(\xi_n^{\rm an})'(\rho^{-1})\xi_n^{\rm an}(1)-\rho^{3/2}\xi_n^{\rm an}(1)(\xi_n^{\rm an})'(\rho).
 \end{equation}
This formula shows, in particular, that the function $Q\epsilon(\rho)$ is $0$ for $\rho=1$.\newline 
Next, we show that the function $\epsilon(\rho)$ which fulfills $\eqref{sonine0}$ for $\rho\geq 1$, and the symmetry $\epsilon(\rho^{-1})=\epsilon(\rho)$, has a jump in its derivative at $\rho=1$, \ie a behavior (at $\rho=1$) similar to that of the function $\delta(\rho)$.

\begin{lem}\label{epsilon'} The derivative of $\epsilon(\rho)$ at  $\rho=1^+$ is 
$$
\epsilon'(1^+)=\sum \frac{\lambda(n)^2}{1-\lambda(n)^2}\xi_n(1)^2.
$$	
\end{lem}
\begin{proof} 
By \eqref{sonine0}, using the above functions $\xi_n^{\rm an}$,  one has   
$$
\epsilon(\rho)=\sum \frac{\lambda(n)^2}{1-\lambda(n)^2} \left( \rho^{1/2}\int_{\rho^{-1}}^1 \xi_n^{\rm an}(x)\xi_n^{\rm an}(\rho x) dx\right),
$$
and moreover $\xi_n^{\rm an}(1)=\xi_n(1)$. \end{proof} 
The convergence of the series is ensured by the inequality (see \cite{Rokhlin}, Theorem 12)
\begin{equation}\label{Rokh}
	\vert \xi_n(1)\vert \leq \sqrt{2n+\frac 12}.
\end{equation}
The  numerical values of the terms $t(n)=\frac{\lambda(n)^2}{1-\lambda(n)^2}\xi_n(1)^2$ are of the form 
$$
t(0)=11.9719,\ t(1)=8.77574,\ t(2)=2.20528,\ t(3)=0.0433983,\ t(4)=0.000125459\ldots
$$ 
and the total value is of the order of $22.9965$. As in \eqref{Qprime} we get, for the linear form 
\begin{equation}\label{Eprimedef}
E_+(f):=\int f(x)\epsilon(\exp(\vert x\vert))dx,\qquad \forall f\in C_c^\infty(\R)
\end{equation}
the expression 
\begin{equation}\label{Eprime}
	 E_+(Q_+f)=-2\epsilon'(1_+)f(0)+\int_0^\infty (f(x)+f(-x))(Q\epsilon)(\exp(x))dx.
\end{equation}

The equality $
Q\epsilon(\rho)=\sum \frac{\lambda(n)^2}{1-\lambda(n)^2} C_n$,  together with \eqref{sonineQbis} give a computable expression for $(Q\epsilon)(\rho)$. The graph, after division by $2\epsilon'(1_+)$ and passing to the additive scale, (\ie using $\rho=\exp(x)$) is  
\begin{figure}[H]	\begin{center}
\includegraphics[scale=0.7]{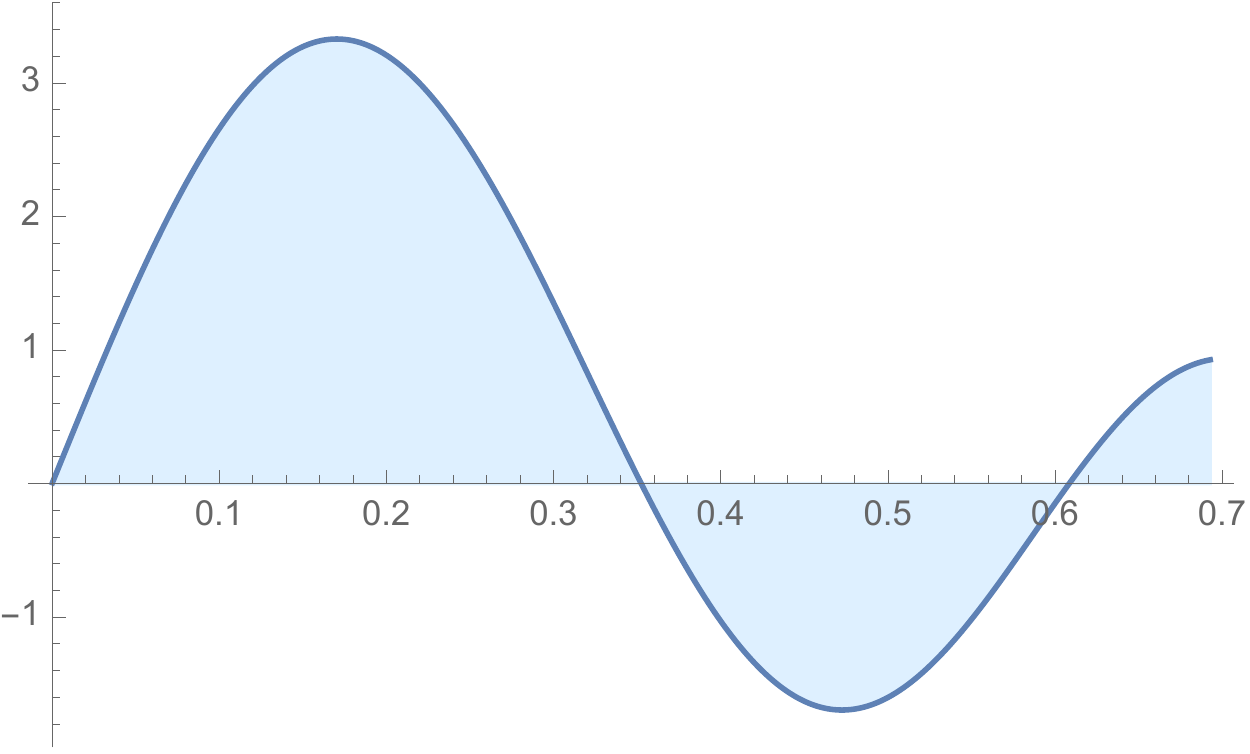}
\end{center}
\caption{Graph of $(Q\epsilon)(\exp(x))/(2\epsilon'(1_+))$ in $[0,\log 2]$. \label{qprime1} }
\end{figure}
In order to determine in which intervals $I$ the form $E_+(Q_+f)$ remains negative, we introduce the associated operator. 

\begin{prop}\label{propcompact} Let $I\subset [-\log 2,\log 2]$ be an interval of length $\leq \log 2$. \newline
$(i)$~The following equality defines a bounded operator $\nf_I$ in the Hilbert space $L^2(I,dx)$
\begin{equation}\label{quadform}
	\langle   \eta\mid \nf_I(\xi)\rangle= E_+(Q_+f), \ \ f=\eta^**\xi, \ \ f(v)=\int  \overline{\eta(x)}\xi(x+v) dx.
\end{equation}
$(ii)$~One has $\nf_I=-2\epsilon'(1_+)\left(\id - \kf_I\right)$, where $\kf_I$ is the compact operator defined by
\begin{equation}\label{opkf}
	\langle   \eta\mid \kf_I(\xi)\rangle=\frac {1}{2\epsilon'(1_+)}\int_{-\log 2}^{\log 2}\int  \overline{\eta(x)}\xi(x+v)  (Q\epsilon)(\exp(\vert v\vert))dxdv.
\end{equation}	
\end{prop}
\begin{proof} $(i)$~will follow from $(ii)$, since a compact operator is bounded.\newline 
$(ii)$~By applying \eqref{Eprime} to  $f(v)=\int  \overline{\eta(x)}\xi(x+v) dx $ one 
obtains 
$$
\langle   \eta\mid \nf_I(\xi)\rangle=-2\epsilon'(1_+)\langle   \eta\mid \xi\rangle+
\int_0^\infty (f(x)+f(-x))(Q\epsilon)(\exp(x))dx.
$$
Since the length of $I$ is $\leq \log 2$, the function $f(v)$ vanishes outside $[-\log 2,\log 2]$ and thus we obtain \eqref{opkf}. The proof of the compactness of the operator $\kf_I$ is the same as the argument developed at the end of the proof of Theorem \ref{thmqkey1}. This shows that $\kf_I$ is of Hilbert-Schmidt class. \end{proof} 
\begin{rem}\label{traceki} The function $Q\epsilon(\rho)$ is $0$ for $\rho=1$. This shows that the integral of diagonal values of the Schwartz kernel defining the compact operator $\kf_I$ is $0$ independently of the size of $I$. This result is a clear improvement on the compact operators involved in  Theorem \ref{thmqkey1} (see Figure~\ref{qprime}), whose trace is proportional to the length.	
\end{rem}

\section{Computation of the spectrum of the compact operator $\kf_I$}\label{sectspectrum}

This section is the most elaborate of the paper: it describes the computation  of the spectrum of the compact operator $\kf_I$ of \eqref{opkf}, for an interval $I\subset [-\log 2,\log 2]$ of length $\log 2$. In fact, we shall consider the interval $I=[-\frac{\log 2}{2},\frac{\log 2}{2}]$. The first difficulty that we encounter concerns  formula \eqref{qe} for the function $Q\epsilon(\rho)$  that involves the prolate spheroidal functions and their derivatives in a complicated way. Fortunately, it is accessible to computer calculations due to the fast decay \eqref{rapid-decay}; moreover we show in Appendix \ref{appendix-cv}  (Lemma \ref{lemesti}) that the sum of the first $11$ terms in the series gives a uniform approximation of the function $Q\epsilon(\rho)$ up to $10^{-11}$. The second difficulty analyzed in this section  has to do with the infinite dimensionality of the Hilbert space on which the operator $\kf_I$  acts. This seems, a priori, to preclude the use of  computational power  to understand its spectrum. Here, the strategy we follow  is to use our original idea of ``$q\to 1$" that has underlied, since the start,  our algebraic work \cite{Crh,CCscal1,CCsurvey}. Thus, we replace the multiplicative group $\R_+^*$ by the discrete subgroup  group $q^\Z$,  and we approximate the infinite dimensional space with  a finite dimensional one where the interval  $I$ (in additive notation) is replaced by the finite set $I_q$ of size $N\sim \log 2/\log q$, of integral multiples of $\epss=\log q$ which belong to $I$. It turns out that the  natural discretization  of the operator $\kf_I$ (\S \ref{sectdiscreteq}) is a Toeplitz matrix and one can investigate its spectrum numerically to see  how it varies, when $q\to 1$. What one  discovers is that, while for intervals of length a bit smaller  than $\log 2$  the largest eigenvalue $\lambda_{\rm max}$ of  $\kf_I$ is  less than $1$, so that $\nf_I$ is negative, this no longer holds true when the interval length gets closer to $\log 2$  (see Figure \ref{protrude}). By looking more closely at the spectrum of $\kf_I$, one  sees that the next to largest eigenvalue stays always smaller than $1$! The next step
then is  to use the powerful theory of self-adjoint Toeplitz matrices \cite{toeplitz,BW} which shows that the eigenvector $\xi$  associated to the largest eigenvalue fulfills a kind of ``baby version" of RH. If one views $\xi$ as a complex-valued function $\xi:I_q\to \C$ by forming the expression $$\tilde \xi (z)=\sum \xi(\log q^j)z^j\in \C[z,z^{-1}], $$
then all the zeros of this function are complex numbers of modulus  $1$. Their number is $N\sim \log 2/\log q$. By computing these $N$ roots, we noticed  that they get distributed as $(N+1)$-th roots of unity except for the trivial root $1$. Moreover, one has the symmetry: $z\mapsto z^{-1}$ due to the equality $\tilde \xi (z^{-1})=\tilde \xi (z)$. When one rescales the arguments taken in $[-\pi,\pi]$, by multiplication  by $(N+1)/2$, their pattern converges when  $q\to 1$ (see \eqref{tableangles} for the beginning of the patern). The general theory of Toeplitz matrices $T$ also shows \cite{toeplitz} that the above roots can be used as a discrete Fourier transform to provide a canonical formula for the lower rank positive Toeplitz matrix $\lambda_{\rm max}\id-T$, as a  linear combination with positive coefficients $d(j)$ of the form 
$$
\lambda_{\rm max}\id -T=\lambda_{\rm max}\sum d(j)e(j)
$$
 where the $e(j)$ are the one dimensional projections (also Toeplitz matrices)  associated to the above roots. It turns out that the pattern of the scalar coefficients $d(j)$ also converges when $q\to 1$ (see \eqref{tableds}).  The next step is taken in \S \ref{sectbasicapprox} and  consists of guessing,  from the discrete approximation and the above patterns for the roots and the coefficients $d(j)$, a function of the continuous parameter $\rho\in [\frac 12,2]$ which approximates $Q\epsilon(\rho)$. This guess is then verified numerically and gives, by a simple estimate, a good approximation of $\kf_I$ and  a good control of its spectrum using a finite rank operator  (see \S \ref{sectapproxfiniterank}). Finally, \S \ref{subsectmaxvect},  \ref{sectspec} contain the  computation of the spectrum of this finite rank operator as well as the eigenvector of maximal eigenvalue, and the proof that the  operator $\kf_I$ becomes $<1$ in the orthogonal complement of this vector. The proof of the main theorem (Theorem~\ref{mainthmfine}) is done in \S \ref{sectproof}.

\subsection{Discrete approximation for variable intervals}\label{sectdiscreteq}
We  discretize the framework by replacing $\R_+^*$ by $q^\Z$, where $q>1$ and then we let $q\to 1$. By setting $\epss=\log q$  we replace the interval $I=[0,a]$ by its finite intersection with the lattice $\epss \Z$,  whose elements $j \epss$  are  labeled by $j\in\{0,\ldots ,N\}$, where $N$ is the integer part of $a/\epss$. Then we  replace integrals by sums and consider, in the finite dimensional Hilbert space $\ell^2(\{0,\ldots ,N\})$,  the following quadratic form which is the discretized version of \eqref{opkf}
\begin{equation}\label{Eprimeqq}
	\cQ_q(\xi):=\epss \sum_{j=0}^{N}\sum_{k=-j}^{N-j} \overline{\xi(j)}\xi(j+k)(Q\epsilon)(q^{\vert k\vert}).
\end{equation}
Following Proposition \ref{propcompact}, one needs to compare $\cQ_q$ with the inner product 
$
2\epsilon'(1_+)\sum_{j=0}^N \overline{\xi(j)}\xi(j)
$. One expresses \eqref{Eprimeqq} as 
$$
\cQ_q(\xi)=\epss\langle \xi\mid \cT_q\xi\rangle,
$$
where the Toeplitz matrix  $\cT_q$ is of the form

\begin{equation}\label{toeplitzQ}
\cT_q= \begin{pmatrix}
Q\epsilon(1) & Q\epsilon(q) & Q\epsilon(q^2) & Q\epsilon(q^3) & \ldots& Q\epsilon(q^N) \\
Q\epsilon(q) & Q\epsilon(1)  & Q\epsilon(q) & Q\epsilon(q^2)&\ldots &Q\epsilon(q^{N-1})\\ 
Q\epsilon(q^2) & Q\epsilon(q) & Q\epsilon(1)  & Q\epsilon(q) & \ldots & Q\epsilon(q^{N-2})\\
Q\epsilon(q^3) & Q\epsilon(q^2) & Q\epsilon(q) & Q\epsilon(1)  & \ldots & Q\epsilon(q^{N-3})\\
\vdots & \vdots & \vdots & \vdots & \ddots & \vdots\\
Q\epsilon(q^N) & Q\epsilon(q^{N-1}) & Q\epsilon(q^{N-2}) & Q\epsilon(q^{N-3})  & \cdots & Q\epsilon(1) 
\end{pmatrix}
\end{equation}

Thus we shall compare the largest eigenvalue of $\frac {1}{2\epsilon'(1_+)}\epss\cT_q$ with $1$ since this tests  the positivity  of $\id -\kf_I$. 

\begin{figure}[H]	\begin{center}
\includegraphics[scale=0.6]{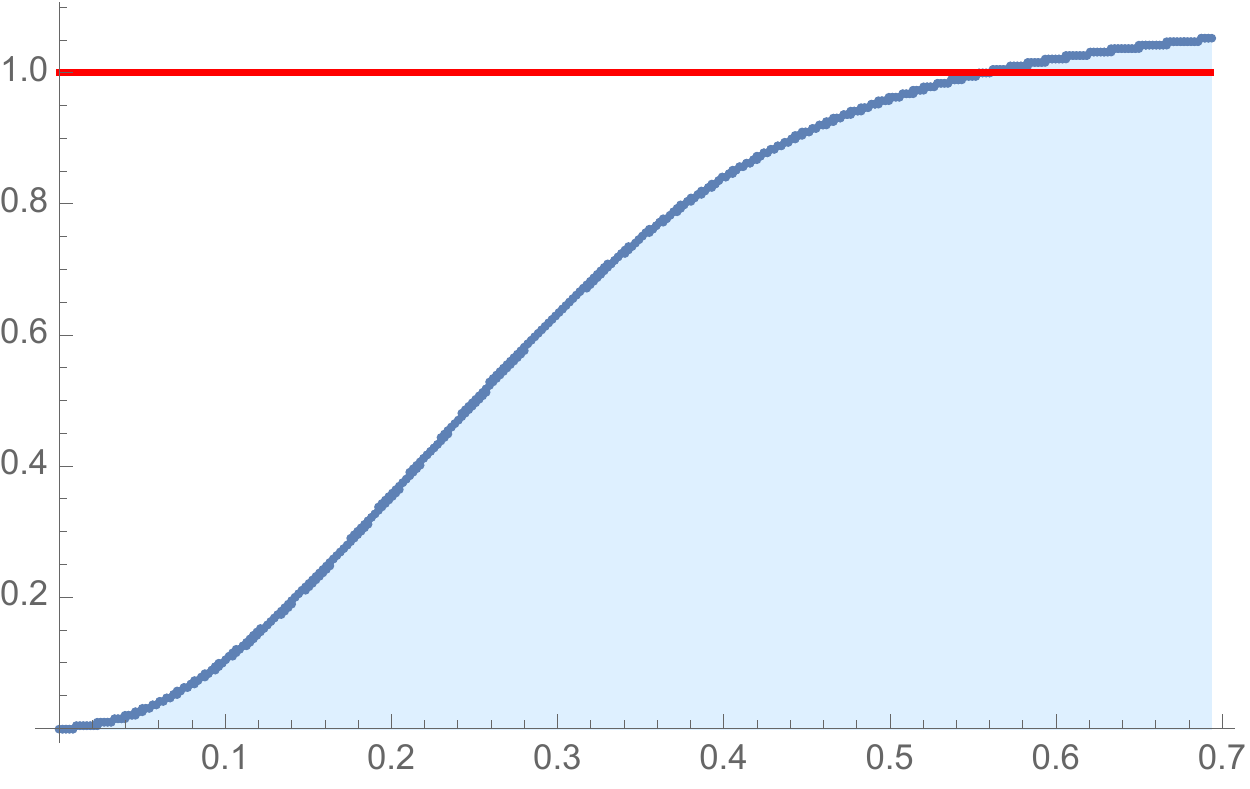}
\end{center}
\caption{Largest eigenvalue of $\frac {1}{2\epsilon'(1_+)}\epss\cT_q$,  for  $a\in [0,\log 2]$ and $q=\exp(10^{-3})$. \label{protrude} }
\end{figure}
Figure \ref{protrude} shows that, for $\epss=10^{-3}$ and $q=\exp(10^{-3})$, the largest eigenvalue  of $\frac {1}{2\epsilon'(1_+)}\epss\cT_q$, \ie $\lambda\sim 1.05177$,  slightly exceeds $1$ when one considers the full interval $[0,\log 2]$.

\subsection{Discrete approximation and Toeplitz matrices}\label{secttoeplitz}
We now  fix the interval $I=[-\frac 12 \log 2,\frac 12 \log 2]$: symmetric and of length $\log 2$. By computing the first eigenvalues of the Toeplitz matrix $\frac {1}{2\epsilon'(1_+)}\epss\cT_q$ one finds that the one next to the largest $\lambda=\lambda_1$ is $\lambda_2\sim 0.687925$: hence well below $1$. Thus   the lack of positivity of $\id -\kf_I$ is due to a single eigenvector $\zeta$. In the following part we  make use of the  general theory of Toeplitz matrices \cite{BW,toeplitz}. A first key classical result in the theory asserts that the eigenvector for the largest eigenvalue of a self-adjoint Toeplitz matrix is of a very special form since the associated polynomial equation has all its roots of modulus $1$. In our notations, this means that if we denote the components of the eigenvector $\zeta$ as $\zeta(\log q^j)=\zeta(j \omega)$ which are defined for $\vert j\vert \omega\leq \frac 12 \log 2$, and let $\tilde \zeta (z)=\sum \zeta(j\omega)z^j\in \C[z,z^{-1}]$, then one has the implication
\begin{equation}\label{rhsmall}
z\in \C\ \& \ \tilde \zeta (z)=0~ \Longrightarrow ~\vert z\vert =1.	
\end{equation}
When we first computed these zeros with the value of $q$ used  in \S \ref{sectdiscreteq} (\ie  $\epss=10^{-3}$, $q=\exp(10^{-3})$),  we found  that indeed these zeros are all of modulus $1$ and obey the symmetry $z\mapsto \bar z$, owing to the fact that the coefficients $\zeta(j\omega)$ are real (they also fulfill $\zeta(-j\omega)=\zeta(j\omega)$). With the symmetric choice of $I=[-\frac 12 \log 2,\frac 12 \log 2]$, the finite number $N$ of elements in $I\cap \omega \Z$ is odd $N=2m+1$ and the computation shows that the $N$ roots of $\tilde \zeta (z)=0$ resemble the non-trivial $N+1$-roots of unity, \ie all of them except $z=1$. Since they are  symmetric these  roots are best written in the form 
\begin{equation}\label{rootspol}
z_j^{\pm}=\exp(\pm \frac{2\pi i \alpha_j}{N+1} )	, \ \ j=1,\ldots m, 
\end{equation}
and it turns out that $z=-1$ is also a root and thus it should be added to this list of $2m$ elements.
In this way, one discovers  that when the roots are labeled as in \eqref{rootspol}, the obtained numbers $\alpha_j$, which depend on the choice of $q= \exp(\omega)$, stabilize when $q\to 1$ and  the difference $\alpha_j-j$ tends to zero when the index $j\to \infty$. The following is a table showing the first values  
 $\alpha_1,\alpha_2,\ldots$, indexed by the integral part \ie $j=$ IntegerPart$(\alpha_j)$: 
\begin{equation}\label{tableangles}
\begin{array}{lllll}
  \alpha _1 =1.33371  &  \alpha _2 =2.10964  &  \alpha _3 =3.07018  &  \alpha _4 =4.0524  &  \alpha _5 =5.04184  \\
  \alpha _6 =6.03484  &  \alpha _7 =7.02984  &  \alpha _8 =8.0261  &  \alpha _9 =9.0232  &  \alpha _{10} =10.0209  \\
  \alpha _{11} =11.019  &  \alpha _{12} =12.0174  &  \alpha _{13} =13.016  &  \alpha _{14} =14.0149  &  \alpha _{15} =15.0139  \\
  \alpha _{16} =16.013  &  \alpha _{17} =17.0123  &  \alpha _{18} =18.0116  &  \alpha _{19} =19.011  &  \alpha _{20} =20.0104  \\
  \alpha _{21} =21.0099  &  \alpha _{22} =22.0095  &  \alpha _{23} =23.0091  &  \alpha _{24} =24.0087  &  \alpha _{25} =25.0083  \\
  \alpha _{26} =26.008  &  \alpha _{27} =27.0077  &  \alpha _{28} =28.0074  &  \alpha _{29} =29.0072  &  \alpha _{30} =30.007  \\
  \alpha _{31} =31.0067  &  \alpha _{32} =32.0065  &  \alpha _{33} =33.0063  &  \alpha _{34} =34.0061  &  \alpha _{35} =35.006  \\
  \alpha _{36} =36.0058  &  \alpha _{37} =37.0056  &  \alpha _{38} =38.0055  &  \alpha _{39} =39.0053  &  \alpha _{40} =40.0052  \\
  \alpha _{41} =41.0051  &  \alpha _{42} =42.005  &  \alpha _{43} =43.0048  &  \alpha _{44} =44.0047  &  \alpha _{45} =45.0046  \\
  \alpha _{46} =46.0045  &  \alpha _{47} =47.0044  &  \alpha _{48} =48.0043  &  \alpha _{49} =49.0043  &  \alpha _{50} =50.0042  \\
  \alpha _{51} =51.0041  &  \alpha _{52} =52.004  &  \alpha _{53} =53.0039  &  \alpha _{54} =54.0039  &  \alpha _{55} =55.0038  \\
  \alpha _{56} =56.0037  &  \alpha _{57} =57.0037  &  \alpha _{58} =58.0036  &  \alpha _{59} =59.0035  &  \alpha _{60} =60.0035  \\
\end{array}
\end{equation}
The above list gives some idea of the numerical values of the $\alpha_j$. The precise numerical values (for $\omega=1/5000$) are not integers and involve more digits. They can be downloaded  at \href{https://www.dropbox.com/s/kjmra7t9ejet1pw/rangles5000.docx?dl=0}{\color{blue}link to download the angles}.   

\begin{figure}[H]	\begin{center}
\includegraphics[scale=0.8]{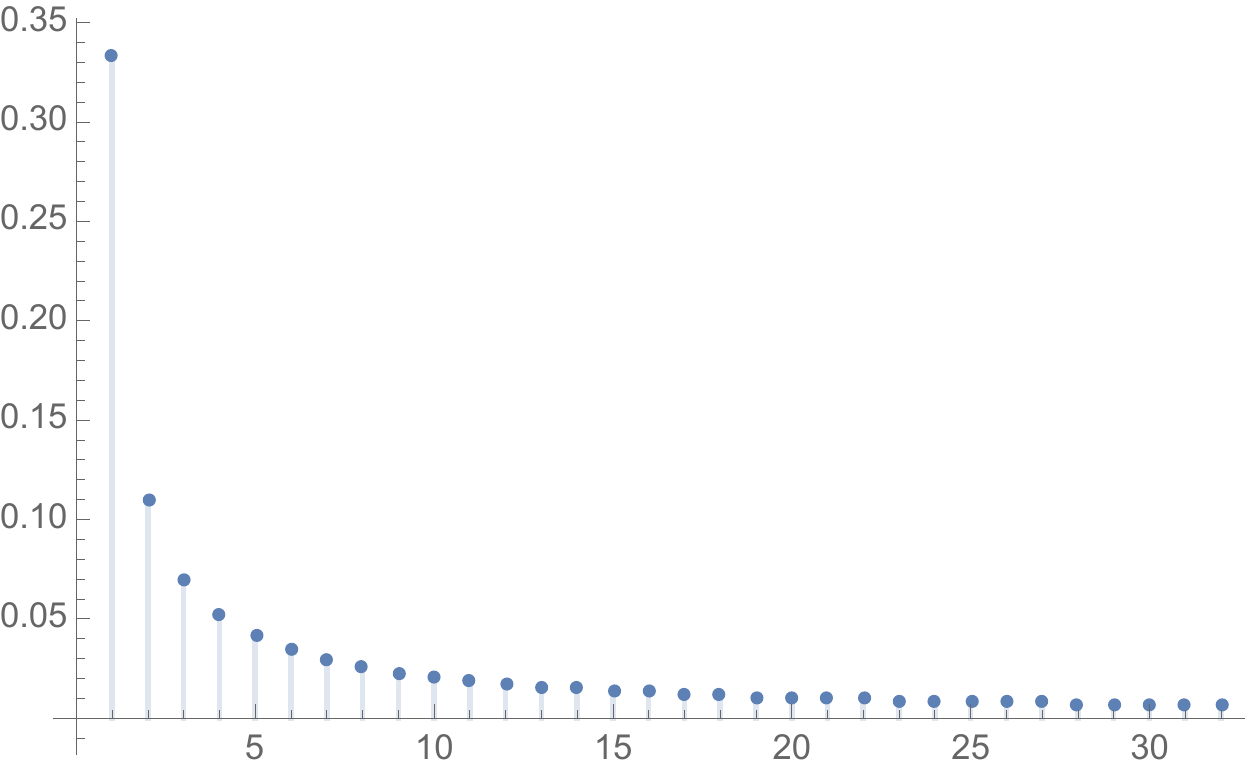}
\end{center}
\caption{Graph of $\alpha_j-j$. \label{angles2} }
\end{figure}
The second key classical  result \cite{BW,toeplitz} of the theory of Toeplitz matrices is that for a positive Toeplitz matrix $T$ of co-rank  $1$ (\ie rank equal to the dimension of the matrix minus one), the face of $T$ in the positive cone is a cone based on a simplex. It follows that  there is a unique decomposition of $T$ as a sum of elements of extreme rays of the positive cone. Moreover, the extreme rays of the  cone are the rank one (positive) Toeplitz matrices, and they are parametrized (up to a positive multiplicative scalar) by the unit circle $\{z\in \C\mid \vert z\vert=1\}$, formed by rank one projections $e(z)$. Moreover, the unique complex numbers of modulus one involved in the decomposition $T=\sum d(j) e(z_j)$ are precisely the zeros of the polynomial associated to the vector in the kernel of $T$. We use this result and apply it to the Toeplitz matrix 
\begin{equation}\label{toepapprox0}
S=\lambda\, \id -\frac {1}{2\epsilon'(1_+)}\epss\cT_q
\end{equation}
($\lambda $
 is the largest eigenvalue of $\frac {1}{2\epsilon'(1_+)}\epss\cT_q$). By construction and using the simplicity of the largest eigenvalue $\lambda $, the obtained Toeplitz matrix $S$ is of co-rank $1$ so that it admits a canonical decomposition of the form 
\begin{equation}\label{toepapprox}
S=\lambda \sum  \, d(j) e(z_j), \qquad \{z_j\}= \{z\in \C\mid \tilde \zeta (z)=0\}.
\end{equation}
We computed the list of positive scalars $d(j)$ corresponding to this unique decomposition and found that when $q\to 1$ they behave similarly to the angles, \ie when they are labeled by the corresponding $z_j$ they converge to a fixed value. 
We give a sample of these values in the next table \eqref{tableds}, where we use the same labeling as in \eqref{tableangles} so that terms correspond bijectively \begin{equation}\label{tableds}
{\small\begin{array}{ccccc}
  d(1)  =1.17111   &  d(2)  =1.12443   &  d(3)  =1.05904   &  d(4)  =1.03248   &  d(5)  =1.02052   \\
  d(6)  =1.01414   &  d(7)  =1.01033   &  d(8)  =1.00787   &  d(9)  =1.00619   &  d(10)  =1.005   \\
  d(11)  =1.00411   &  d(12)  =1.00344   &  d(13)  =1.00292   &  d(14)  =1.00251   &  d(15)  =1.00217   \\
  d(16)  =1.0019   &  d(17)  =1.00167   &  d(18)  =1.00148   &  d(19)  =1.00132   &  d(20)  =1.00119   \\
  d(21)  =1.00107   &  d(22)  =1.00097   &  d(23)  =1.00088   &  d(24)  =1.0008   &  d(25)  =1.00073   \\
  d(26)  =1.00067   &  d(27)  =1.00062   &  d(28)  =1.00057   &  d(29)  =1.00052   &  d(30)  =1.00048   \\
  d(31)  =1.00045   &  d(32)  =1.00042   &  d(33)  =1.00039   &  d(34)  =1.00036   &  d(35)  =1.00034   \\
  d(36)  =1.00031   &  d(37)  =1.00029   &  d(38)  =1.00027   &  d(39)  =1.00026   &  d(40)  =1.00024   \\
  d(41)  =1.00022   &  d(42)  =1.00021   &  d(43)  =1.0002   &  d(44)  =1.00018   &  d(45)  =1.00017   \\
  d(46)  =1.00016   &  d(47)  =1.00015   &  d(48)  =1.00014   &  d(49)  =1.00013   &  d(50)  =1.00013   \\
  d(51)  =1.00012   &  d(52)  =1.00011   &  d(53)  =1.0001   &  d(54)  =1.0001   &  d(55)  =1.00009   \\
  d(56)  =1.00008   &  d(57)  =1.00008   &  d(58)  =1.00007   &  d(59)  =1.00007   &  d(60)  =1.00006   \\
\end{array}}
\end{equation}

Again, this is just to give an idea of these values, and  (for $\omega=1/5000$) the precise numerical values involve more digits and can be downloaded at  \href{https://www.dropbox.com/s/e2bswzrh3tps00h/coefficients.docx?dl=0}{\color{blue}link to download the coefficients}.
To achieve a good control of the compact operator $\kf_I$  we then need to approximate the function $(Q\epsilon)(\exp(x))$ for all $x\in [0,\log 2]$ and not just on the finite set of multiples of $\omega$. In the next section we shall show how the above Toeplitz decomposition \eqref{toepapprox0} allows one to guess an efficient approximation of the function $(Q\epsilon)(\exp(x))$ by a finite trigonometric sum. This approximation is then shown, by a computer calculation, to give the required control.

\subsection{The basic approximation of $(Q\epsilon)(\exp(x))$}\label{sectbasicapprox}
By combining \eqref{toepapprox0} and \eqref{toepapprox},  the Toeplitz  matrix  $T_q=\frac {1}{2\epsilon'(1_+)}\epss\cT_q$ can be re-written in the form 
\begin{equation}\label{toepapprox1}
T_q=\lambda\, \left(\id -\sum d(j)e(z_j)\right)
\end{equation}
 where the $e(z_j)$ are the one dimensional Toeplitz projections  matrices  obtained by conjugating the one dimensional projection on the constant function by the unitary operators 
 $$
 (U(\alpha)\xi)(x):= \exp(i 2 \pi \alpha x/\log 2)\xi(x).
 $$ 
 This suggests that one can approximate the function $\chi(x):=(Q\epsilon)(\exp(x))/(2\epsilon'(1_+))$ in $[0,\log 2]$ by a trigonometric expression of the form 
\begin{equation}\label{approx0}
\tau(\lambda,\alpha,d,m)(x):=\frac{2\lambda}{\log 2} \left(\frac 12+\sum_{n=1}^m \left(\cos \frac{2 \pi  n x}{\log 2}-d(n)\cos \frac{2 \pi  \alpha_n x}{\log 2}\right) \right).
\end{equation}

The following fact holds

 \begin{fact}\label{approximation}
	The distance in $L^1([0,\log 2],dx)$ of the function $\chi(x):=(Q\epsilon)(\exp(\vert x\vert))/(2\epsilon'(1_+))$ to the function $\tau(\lambda,\alpha,d,m)(x)$ of \eqref{approx0} (for $m=1732$, and with the values of the angles $\alpha_j$ and of the coefficients $d(j)$ fixed above) fulfills
		\begin{equation}\label{computerverif}
2\int_0^{\log 2}\vert\tau(\lambda,\alpha,d,m)(x)-\chi(x)\vert dx \sim 0.00122.
\end{equation} 
\end{fact}
\begin{proof} The proof is a computer calculation of the $L^1([0,\log 2],dx)$ norm of the difference of the two functions.  The function $\tau(\lambda,\alpha,d,m)(x)$ oscillates (these oscillations are visible in the neighborhood of $\log 2$) but it otherwise approximates very well the function $\chi(x)$ as shown by the computer calculation of the $L^1$ norm  of the difference
  \begin{figure}[H]	\begin{center}
\includegraphics[scale=0.8]{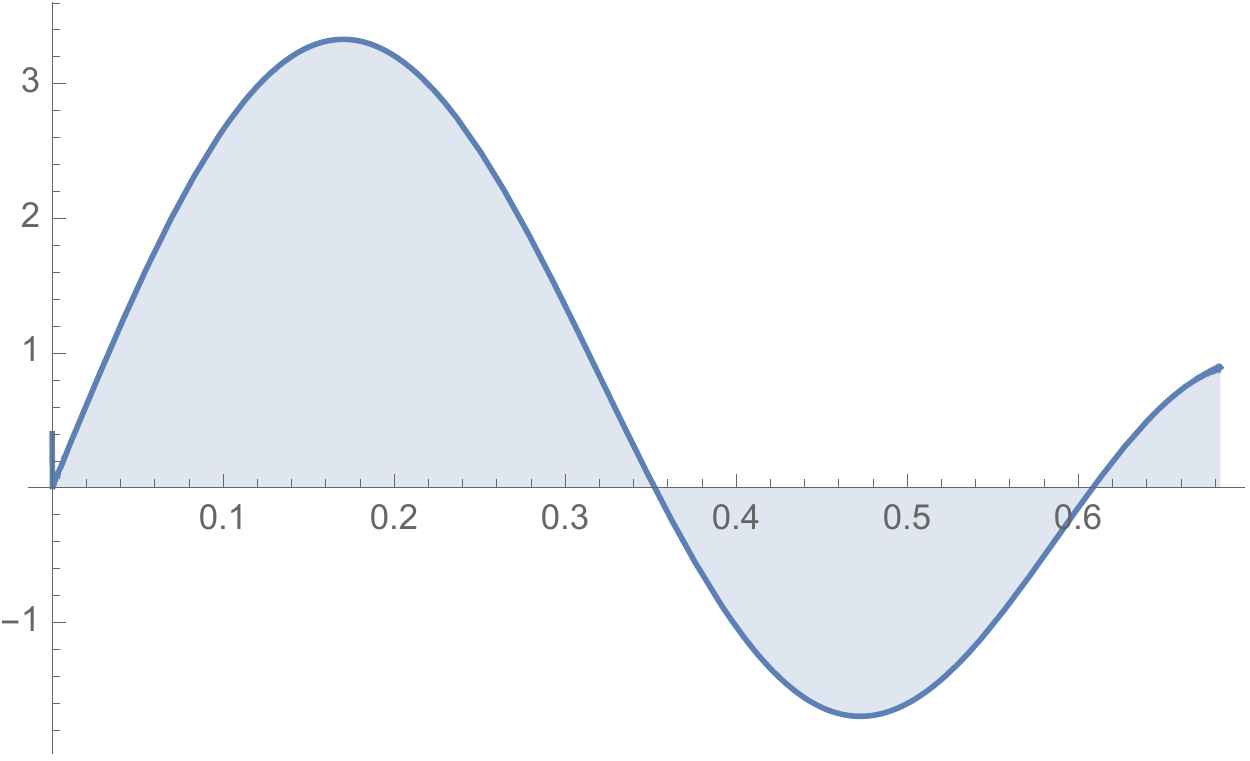}
\end{center}
\caption{Graph of $\tau(\lambda,\alpha,d,m)(x)$ in $[0,\log 2]$ for $m=1732$. \label{taufunctiongraph} }
\end{figure}
  
To justify  \eqref{computerverif}, one first uses \eqref{computersafe1} of Appendix \ref{appendix-cv} to replace, without any loss, the function $\chi(x):=(Q\epsilon)(\exp(x))/(2\epsilon'(1^+))$ using the contribution  of the first $11$ terms of the series \eqref{computersafe} defining $Q\epsilon$. 
\end{proof}

 \subsection{The approximation of $\kf_I$ by a finite rank operator $T$}\label{sectapproxfiniterank}
 Here, the goal is to estimate the quadratic form obtained when one replaces the function $\chi(x)$ by its approximation $\tau(\lambda,\alpha,d,m)(x)$.
 
 \begin{lem}\label{poisson} Let $f\in C_c^{\infty}(\R)_{\rm ev}$ be an even smooth function with support contained in the closed interval $[-\log 2,\log 2]$. Then,  after rearranging the order of summation, one obtains
 \begin{equation}\label{deltafunction}
\frac{2}{\log 2}\int_{-\log 2}^{\log 2} \left(\frac 12+\sum_1^\infty \cos \frac{ 2\pi  n x}{\log 2}\right)f(x)dx=f(0)
\end{equation} 	
 \end{lem}
\begin{proof} The equality follows by applying  Poisson's formula. Let $L= \Z\, \log 2$ be the lattice of integral multiples of $\log 2$ and $L^\perp=\Z/\log 2$ be the dual lattice. 
 The Poisson summation formula gives 
 $$
 \sum_L f(x)=\frac{1}{\log 2}\,\sum_{L^\perp}\widehat{ f}(y),\qquad 
 \widehat{ f}(y)= \int f(u)\exp(-2 \pi i uy)du.$$
 Since $f$ is even and its support is contained in the closed interval $[-\log 2,\log 2]$ one has,   
 $$
\sum_L f(x)=f(0), \qquad \widehat{f}(y)= \int_{-\log 2}^{\log 2}\exp(2 \pi i y x)f(x)dx=  \int_{-\log 2}^{\log 2}\cos(2 \pi  y x)f(x)dx
 $$
 and the Poisson formula thus gives 
 $$
 f(0)=\frac{1}{\log 2}\,\sum_{L^\perp}\widehat{ f}(y)=\frac{1}{\log 2}\sum_\Z\int_{-\log 2}^{\log 2}\cos\left(\frac{ 2\pi  n x}{\log 2}\right)f(x)dx
 $$
  which gives \eqref{deltafunction}.\end{proof} 
  
  Next, we consider the Hilbert space $\cH:=L^2([-\frac 12 \log 2,\frac 12 \log 2],dx)$. For $\alpha\in \R$ we let 
  \begin{equation}\label{xialpha}
  	\xi_\alpha(x):=(\log 2)^{-\frac 12}\exp(\frac{2 \pi i \alpha x}{\log 2}),\qquad \forall x\in [-\frac 12 \log 2,\frac 12 \log 2]
  \end{equation}
  and $\bp_\alpha=\vert \xi_\alpha\rangle\langle \xi_\alpha\vert$ be the associated orthogonal projection,
  $$
  \bp_\alpha(\xi)=\xi_\alpha\langle \xi_\alpha\mid \xi\rangle,\qquad \forall \xi \in \cH.
 $$
 One then has for any $\xi, \eta \in \cH$,  using the special form \eqref{xialpha} of the vector  $\xi_\alpha$ 
 $$
 \langle    \eta\mid \bp_\alpha(\xi)\rangle=
 \langle   \eta\mid \xi_\alpha\rangle\langle \xi_\alpha\mid \xi\rangle=\frac{1}{\log 2}\int_{I\times I}\overline{\eta(x)}\exp(\frac{2 \pi i \alpha x}{\log 2})\xi(y)\exp(\frac{-2 \pi i \alpha y}{\log 2})dx dy
 $$
 so that one obtains 
\begin{equation}\label{opkf1}
	 \langle    \eta\mid \bp_\alpha(\xi)\rangle=\frac{1}{\log 2}\int_{-\log 2}^{\log 2}\int  \overline{\eta(x)}\xi(x+v) \exp(\frac{-2 \pi i \alpha v}{\log 2})dxdv.
\end{equation}

The following lemma plays a key role in the approximation process

 \begin{lem}\label{approachk} Let $\tau(\lambda,\alpha,d,m)(x)$ be an approximation of the function $\chi(x)$ so that the $L^1$ norm of the difference $\tau(\lambda,\alpha,d,m)-\chi$ is $\leq \epsilon$. Then the compact operator $\kf_I$ of \eqref{opkf}, for $I=[-\frac 12 \log 2,\frac 12 \log 2]$, is at a norm distance less than $\epsilon$ from the finite rank operator
\begin{equation}\label{opT}
T=\lambda \ \sum_{n\in \Z} \left(\bp_n-d(\vert n\vert)\bp_{\alpha_n}\right). 
 \end{equation}
 Here, we set $\alpha_{-n}=-\alpha_n~\forall n$ and  $d(0)=0$; while for $n>m$, we set $\alpha_n=n$ and $d(n)=1$  so that all the terms in the above sum for $\vert n\vert>m$ vanish.
   \end{lem}
   \begin{proof} By \eqref{approx0}, one has
   $$
\tau(\lambda,\alpha,d,m)(v)=\frac{\lambda}{\log 2} \sum_{-m}^m \left(\exp(\frac{-2 \pi i n v}{\log 2})-d(\vert n\vert)\exp(\frac{-2 \pi i \alpha_n v}{\log 2})\right)
$$
   so that, by \eqref{opkf1}, the operator $T$ of \eqref{opT} on $\cH:=L^2([-\frac 12 \log 2,\frac 12 \log 2],dx)$ fulfills the equality
   \begin{equation}\label{opT1}
	\langle   \eta\mid  T(\xi)\rangle=\int_{-\log 2}^{\log 2}\int  \overline{\eta(x)}\xi(x+v)  \tau(\lambda,\alpha,d,m)(v)dxdv.
\end{equation}
The compact operator $\kf_I$ of \eqref{opkf} fulfills the same equality, with $\chi(x)$ in place of $\tau(\lambda,\alpha,d,m)(x)$ in the integral. Thus the norm of  $\kf_I-T$ is bounded by  the inequality 
\begin{equation}\label{opTbound}
\bigg\lvert\int_{-\log 2}^{\log 2}\int  \overline{\eta(x)}\xi(x+v)  a(v)dxdv\bigg\rvert \leq \Vert \xi \Vert \Vert \eta \Vert \int_{-\log 2}^{\log 2}\vert a(v)\vert dv 
\end{equation}
  which follows from the Schwarz inequality 
  $
  \bigg\lvert \int  \overline{\eta(x)}\xi(x+v)  dx\bigg\rvert \leq \Vert \xi \Vert \Vert \eta \Vert
  $.\end{proof}  
  
  \subsection{The eigenvector of maximal eigenvalue}\label{subsectmaxvect}
   In order to understand the finite rank operator $T$ of \eqref{opT} we first construct a vector $\zeta\in\cH$  orthogonal to all  vectors $\xi_{\alpha_n}$ for $n\neq 0$, using the conventions of Lemma \ref{approachk}: \ie for $n>m$ we set $\alpha_n=n$.  We first consider the infinite product 
   $$
   h(z):=\prod_{n>0}\left(1-\frac{z^2}{\alpha_n^2} \right) 
   $$
   which is convergent likewise the product defining  $\frac{\sin (\pi  z)}{\pi  z}$ and is, by construction, the product of $\frac{\sin (\pi  z)}{\pi  z}$ by a rational fraction whose role is to replace the zeros $\pm n$ for $n\in \{1,\ldots,m\}$, by the  $\pm \alpha_n$. We then  consider the Fourier transform of $h(z \log 2)$. We use the notations of Lemma \ref{approachk}.
   
    \begin{lem}\label{supportf} The Fourier transform $\psi(x)=\frac{1}{\log 2}\widehat{h}(\frac{x}{\log 2})$ of $h(z \log 2)$ has support in the interval $I=[-\frac 12 \log 2,\frac 12 \log 2]$. One has $T\psi= \lambda \psi$ and (using the conventions of Lemma \ref{approachk}) \begin{equation}\label{psiortho}
   \langle   \xi_{0}\mid \psi\rangle= (\log 2)^{-\frac 12}, \qquad \langle   \xi_{\alpha_n}\mid \psi\rangle=0,\qquad \forall n\neq 0.
   \end{equation}
   \end{lem}
   \begin{proof} By construction one has 
 \begin{equation}\label{fourierloc}
   \prod_{0<n\leq m}\left(1-\frac{z^2}{\alpha_n^2} \right)\frac{\sin (\pi  z)}{\pi  z}=\prod_{0<n\leq m}\left(1-\frac{z^2}{n^2} \right)h(z).
   \end{equation}
   The Fourier transform of $\frac{\sin (\pi  z)}{\pi  z}$ is the characteristic function of the interval $[-\frac 12,\frac 12]$, while  the Fourier transform of the left hand side of \eqref{fourierloc} is a distribution with support in the interval $[-\frac 12,\frac 12]$. Thus, by \eqref{fourierloc},  the Fourier transform $\widehat{h}$ of the function $h(z)$ fulfills the differential equation of degree $2m$
   $$
   \prod_{0<n\leq m}\left(1+\frac{\partial^2}{(2 \pi n)^2 }\right)\widehat{h}(x)=0,\qquad \forall x \notin [-\frac 12,\frac 12].
   $$
   Since the space of solutions of this differential equation is made by   functions which are linear combinations of the $2m$ trigonometric functions $\exp(\pm 2 \pi i n x)$ for $\vert n\vert\leq m, n\neq 0$, one sees that all these functions are periodic of period $1$; thus  since $\widehat{h}$ is square integrable it must vanish identically outside $[-\frac 12,\frac 12]$. Rescaling by $\log 2$, \ie using $\psi(x)=\frac{1}{\log 2}\widehat{h}(\frac{x}{\log 2})$, one obtains that $\psi$ has support in the interval $I=[-\frac 12 \log 2,\frac 12 \log 2]$. By Fourier inversion, one has for any $n\in \Z$, $n\neq 0$ 
   \begin{align*}
   \langle   \xi_{\alpha_n}\mid \psi\rangle&=(\log 2)^{-\frac 12}\int_I\psi(x)\exp(-\frac{2 \pi i \alpha_n x}{\log 2})dx=   \\
  &= (\log 2)^{-\frac 12}\int \widehat{h}(y)\exp(-2 \pi i \alpha_n y)dy=
(\log 2)^{-\frac 12}h(\alpha_n)=0
   \end{align*}
   which gives \eqref{psiortho} (since $h(1)=1\Rightarrow \langle   \xi_0\mid \psi\rangle=(\log 2)^{-\frac 12}$).  The  orthogonality of $\psi$ to all the vectors $\xi_{\alpha_n}$  shows using \eqref{opT} that $T\psi=\lambda\sum_\Z \bp_n \psi=\lambda \psi $, since the vectors $\xi_n$ form an orthonormal basis of $\cH$.
   \end{proof} 
   
   Note that the function $\psi$ is not normalized. The computation of the $L^2$ norms gives
   \begin{equation}\label{normpsi}
   	\Vert \psi \Vert_2 = (\log 2)^{-\frac 12}	\Vert h \Vert_2, \ \ 	\Vert h \Vert_2 \sim 1.05143, \ \ \text{for}\ m=1732
   \end{equation}

   \begin{figure}[H]	\begin{center}
\includegraphics[scale=0.8]{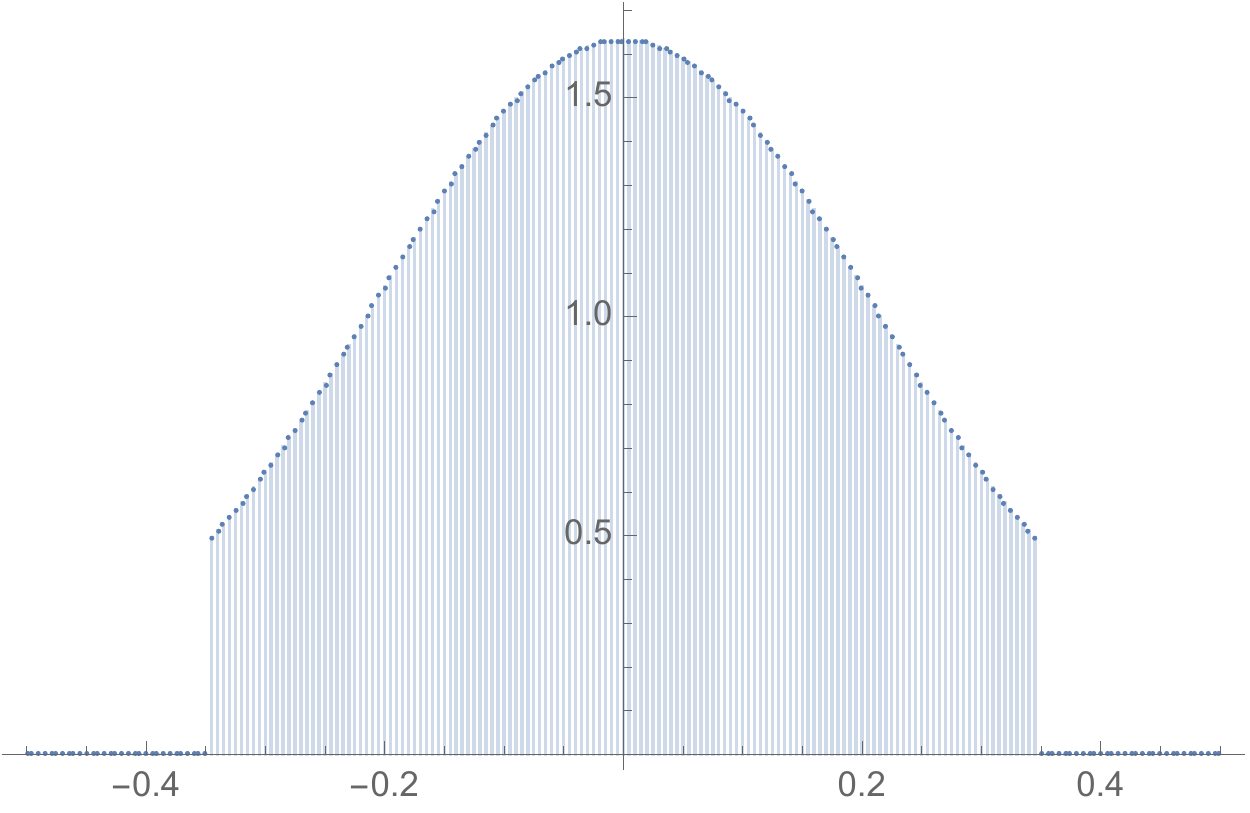}
\end{center}
\caption{Graph of $\zeta(x)=\psi(x)/\Vert \psi\Vert_2$ in $[-\frac 12 ,\frac 12 ]$. \label{psifunct} }
\end{figure}

The important numerical fact is 
\begin{fact}\label{factnum}
For $m=1732$, one has:  $\langle   \xi_{0}\mid \zeta\rangle\sim 0.94865$, where $\zeta(x)=\psi(x)/\Vert \psi\Vert_2$.
\end{fact}

\subsection{Computation of the spectrum of $T$}\label{sectspec}
   The first method to compute the spectrum of the operator  $T$ of \eqref{opT} is to approximate this finite rank operator  using the orthogonal projection $P(n)$ on the linear span of the vectors $\xi_j$, for $\vert j\vert < n$. We use the following expression of the norm square $\Vert\xi_\alpha-P(n)\xi_\alpha\Vert^2$.
   
   \begin{lem}\label{projPn}  Let $(\log \Gamma)^{(2)}$ be the derivative of the logarithmic derivative of the $\Gamma$-function, then one has
   \begin{equation}\label{majorred}
  \Vert\xi_\alpha-P(n)\xi_\alpha\Vert^2= \pi ^{-2} \sin ^2(\pi \alpha)\left((\log \Gamma)^{(2)}(n-\alpha)+(\log \Gamma)^{(2)}(\alpha+n) \right),\qquad \forall \alpha\in [-n,n] \end{equation}
       \end{lem}
   \begin{proof} The components of the vector $\xi_\alpha$ in the basis $\xi_m$ are given as follows 
   $$
   (\xi_\alpha)_k=\frac{1}{\log 2}\int _{-\frac{\log (2)}{2}}^{\frac{\log (2)}{2}} \exp \left(\frac{2 \pi  i (\alpha-k) x}{\log 2}\right)dx=\frac{\sin (\pi  (\alpha-k))}{\pi  (\alpha-k)}.
   $$
   One then uses  the identity, for $a<n$~ ($\sin ^2(\pi  (a-n))=\sin ^2(\pi a)$)
   $$
   \sum _{k=n}^{\infty } \left(\frac{\sin (\pi  (a-k))}{\pi  (a-k)}\right)^2=\frac{\sin ^2(\pi  (a-n)) }{\pi ^2}\sum _{k=0}^{\infty }(n-a+k)^{-2} =\pi ^{-2} \sin ^2(\pi a)(\log \Gamma)^{(2)}(n-a)
   $$
Similarly for $-n<a$ one has 
   $$
   \sum _{k=-\infty}^{-n } \left(\frac{\sin (\pi  (a-k))}{\pi  (a-k)}\right)^2=\pi ^{-2} \sin ^2(\pi a) \sum _{k=-\infty}^{-n } (a-k)^{-2}=\pi ^{-2} \sin ^2(\pi a)(\log \Gamma)^{(2)}(n+a)
   $$
   which gives \eqref{majorred}.
   \end{proof} 

  The equality
  $$
  P(n)\bp_\alpha(P(n)\xi) =\langle   \xi_\alpha\mid P(n)\xi\rangle P(n)\xi_\alpha
  =\langle   P(n) \xi_\alpha\mid \xi\rangle P(n)\xi_\alpha
  $$
  gives the simple estimate in operator norm 
  $$
  \Vert P(n)\bp_\alpha P(n)-\bp_\alpha\Vert \leq 2 \Vert\xi_\alpha-P(n)\xi_\alpha\Vert.
  $$
  This allows one to control the norm of the difference  $T-P(m)TP(m)$ as follows
\begin{equation}\label{opTcompressed}
\Vert T-P(m)TP(m)\Vert \leq 2\lambda \ \sum_{\vert n\vert <m} d(\vert n\vert)\Vert\xi_{\alpha_n}-P(m)\xi_{\alpha_n}\Vert.
 \end{equation}

Using \eqref{majorred} and the asymptotic behavior 
   $
   (\log \Gamma)^{(2)}(x)=\frac{1}{x}+\frac{1}{2 x^2}+O\left(\frac{1}{x}\right)^3
   $, one obtains a first control of $\Vert T-P(m)TP(m)\Vert$. Then, one can  compute the eigenvalues of the finite dimensional matrix $P(m)TP(m)$. We did it for $m=1733$, after dividing by $\lambda$, to check the highest eigenvalue to be $1$. One can then obtain the list of its eigenvalues;  the first few, when they are arranged in decreasing order, are the following
   $$
   \{1.,0.652824,0.027475,0.000290146,0.0000877245,0.0000756436\}.
   $$
Only the first three stand out as stable positive eigenvalues for $T/\lambda$. After multiplication by $\lambda$ these become
  \begin{equation}\label{opTeigen}
 \lambda=1.05158,\ \ \lambda_2=0.686494, \ \ \lambda_3=0.0288921.
    \end{equation}
One can also get the components $c_n$, on the basis of the $\xi_n$, of the eigenvector associated to the eigenvalue $\lambda$. These components are smaller than $10^{-4}$ for $n>30$ and their graph near $n=0$ (and for $n$ a bit further)  is reproduced here below  
 \begin{figure}[H]
\begin{minipage}[b]{0.45\linewidth}
\centering
\includegraphics[width=\textwidth]{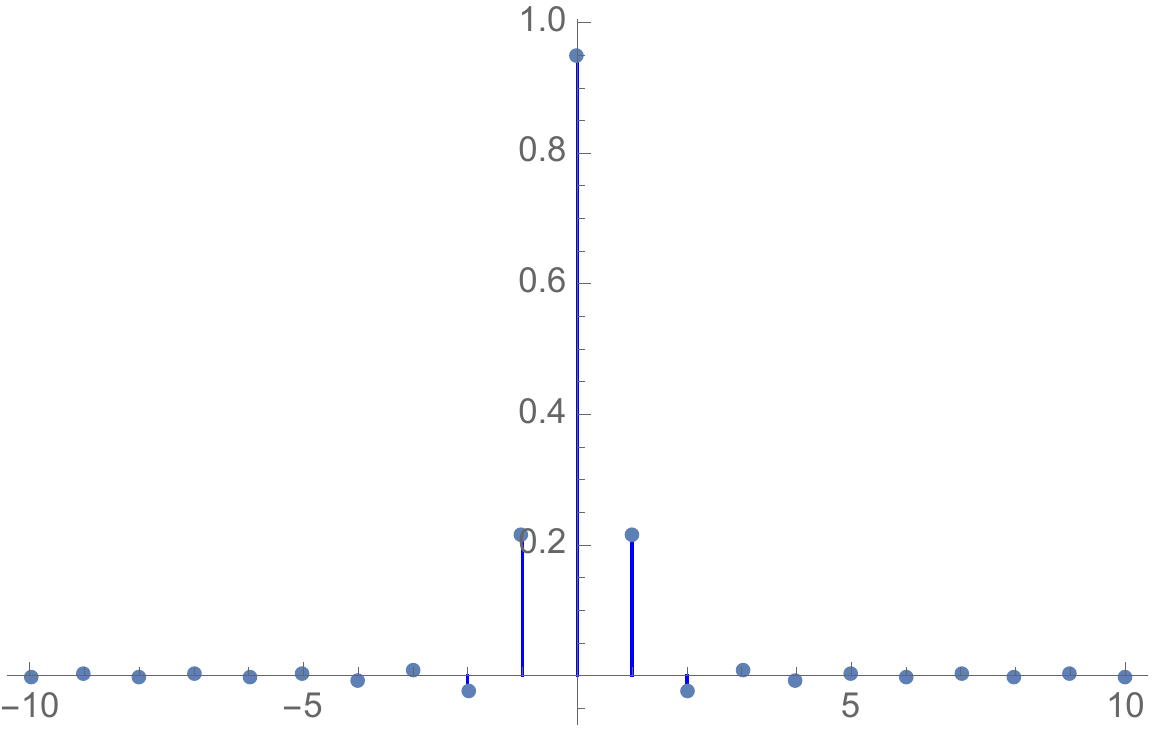}
\caption{\small{Graph of the components $c_n$ for $\vert n\vert <10$.}}
\label{eigencomps1}
\end{minipage}
\hspace{0.5cm}
\begin{minipage}[b]{0.45\linewidth}
\centering
\includegraphics[width=\textwidth]{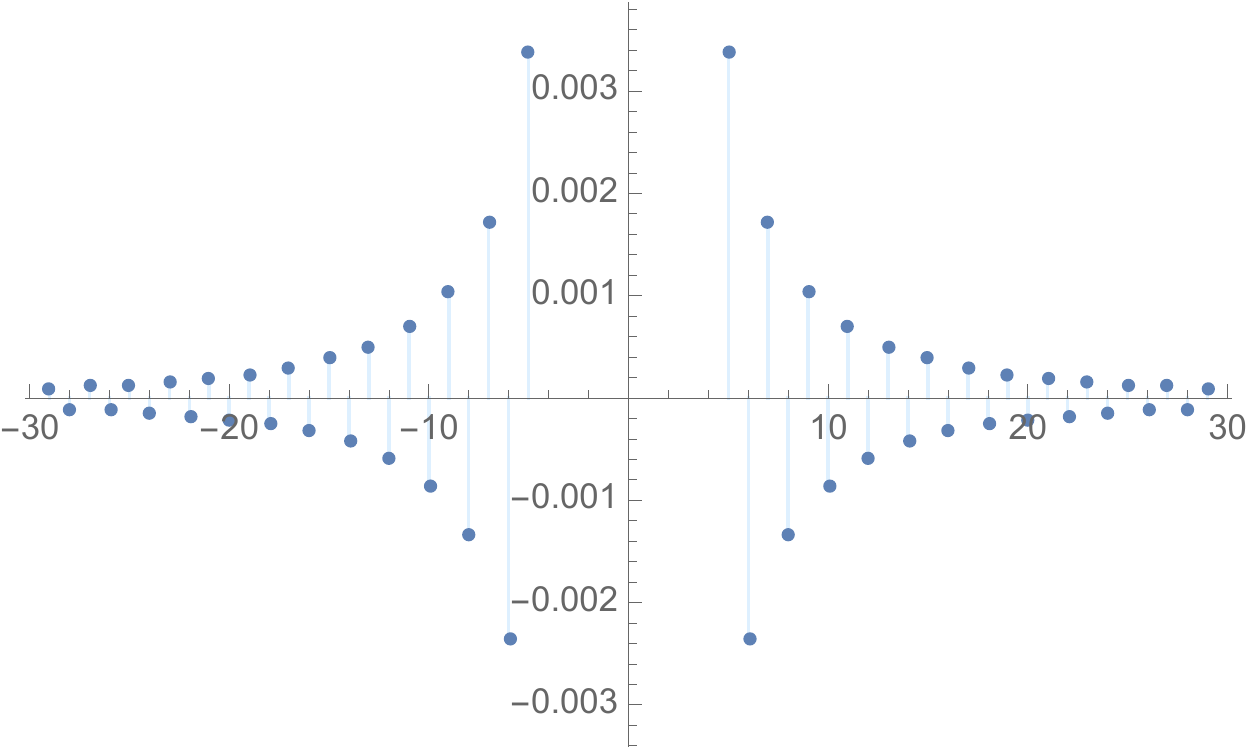}
\caption{\small{Graph of the components $c_n$ for $7<\vert n\vert <30$. } }
\label{eigencomps2}
\end{minipage}
\end{figure}

One also checks that the graph (Figure \ref{reconstructedeigen}) of the reconstructed function $\sum c_n\xi_n$ coincides with the graph (Figure \ref{psifunct}) of the theoretical eigenvector of \S \ref{subsectmaxvect}.
 \begin{figure}[H]	\begin{center}
\includegraphics[scale=0.7]{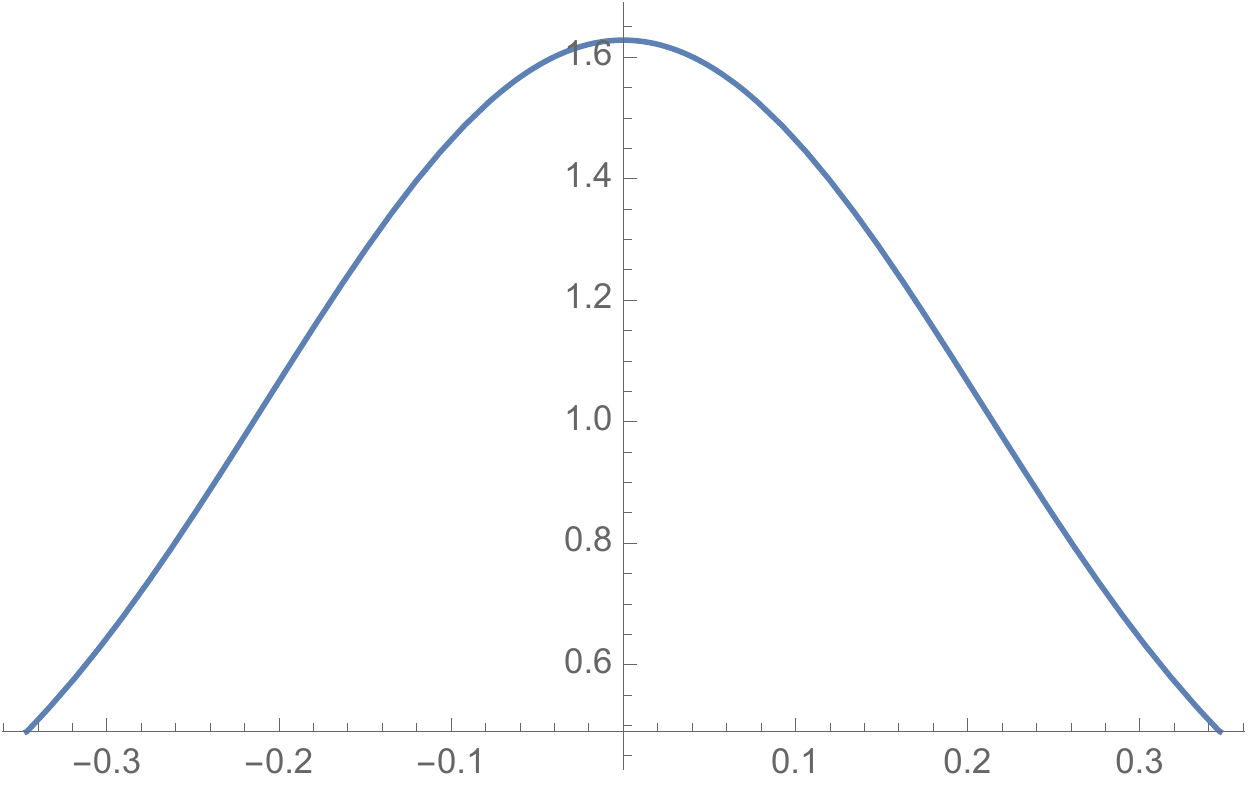}
\end{center}
\caption{Graph of the reconstructed function $\sum c_n\xi_n$. \label{reconstructedeigen} }
\end{figure}

The important fact is that the first component $c_0\sim 0.951067$ is very close to $1$.

\begin{rem}\label{remevenodd} The components $c_n$ fulfill the symmetry: $c_{-n}=c_n$ for all $n$. In fact, the finite dimensional real symmetric matrix $M=P(m)TP(m)$ fulfills the further symmetry $M_{-i,-j}=M_{i,j}$ \ie it commutes with the parity involution. It follows that the eigenvectors associated to simple eigenvalues are even or odd (with respect to the parity involution). One finds, for example, that the eigenvector associated to the second eigenvalue $\lambda_2$ is odd, \ie its components $c'_n$ fulfill $c'_{-n}=-c'_n$ for all $n$.	
\end{rem}

For our purposes, the estimate \eqref{opTcompressed} does not insure enough precision in the spectrum of $T$ and moreover the need to input the sum \eqref{opT} gives lengthy computations. 
We now describe a second  method to compute the spectrum of $T$ which improves the precision. 

We consider a new basis $(\zeta_n)$ of $\cH=L^2([-\frac 12 \log 2,\frac 12 \log 2],dx)$ which is no longer orthonormal. More precisely let, with the notation \eqref{xialpha},
$$
\zeta_0=\zeta, \qquad \zeta_k=\begin{cases}\xi_k, & \text{for~ $\vert k\vert>m$,}\\
\xi_{-\alpha_{\vert k\vert}} &\text{for ~$-m\leq k\leq - 1,$}\\
\xi_{\alpha_k} &\text{for~~$1\leq k\leq m$.}\end{cases}  
$$
One needs to check that the $\zeta_k$, for $\vert k\vert\leq m$, are linearly independent and this suffices to show that the $(\zeta_n)$ form a basis. In fact, by Lemma \ref{supportf}, $\zeta_0=\zeta$ is orthogonal to all $\zeta_j$ for $j\neq 0$. In this basis the inner product in $\cH$ is given by the matrix $J$: $J_{i,j}:=\langle \zeta_i \mid \zeta_j\rangle$. The operator \begin{equation}\label{G} G:= \sum_{n\in \Z} d(\vert n\vert)\bp_{\alpha_n}\end{equation} is such that 
$$
\langle \zeta_i \mid G\zeta_j\rangle=\sum_{n\in \Z} d(\vert n\vert)\langle \zeta_i \mid \bp_{\alpha_n}\zeta_j\rangle=\sum_{n\in \Z} d(\vert n\vert)\langle \zeta_i \mid \xi_{\alpha_n}\rangle \langle\xi_{\alpha_n}\mid \zeta_j\rangle=\sum_{n\in \Z} d(\vert n\vert)\langle \zeta_i \mid \zeta_n\rangle \langle\zeta_n\mid \zeta_j\rangle.
$$

\begin{lem}\label{lemspec} $(i)$~The spectrum of the operator $T$ of  \eqref{opT} is $\{\lambda_{\rm max} (1-\beta_j)\}$, where  $\beta_j$ are the eigenvalues of the matrix $A$:   $A_{n,k}:=d(\vert k\vert)\langle \zeta_n \mid \zeta_k\rangle$.\newline
$(ii)$~Let $N>m$. The eigenvalues of the matrix $A$ are approximated by the eigenvalues of the  matrix $A^{(N)}$ defined by: $A^{(N)}_{i,j}=A_{i,j}$ if
 $\vert i\vert \leq N$,  $\vert j\vert \leq N$ and  	$A^{(N)}_{i,j}=\delta_{i,j}$ otherwise, up to an error of $11\,\epsilon(N)$ where $$\epsilon(N)=\max (e(N),e'(N)), \qquad e(N)^2=\sum_{\vert j\vert\leq N}\epsilon(j,N), \ \ e'(N)^2=\sum_{\vert j\vert\leq N}d(\vert j\vert)^2\epsilon(j,N)$$
 with  $\epsilon(j,N):= \pi ^{-2} \sin ^2(\pi \alpha_j)\left((\log \Gamma)^{(2)}(N-\alpha_j)+(\log \Gamma)^{(2)}(\alpha_j+N) \right)$.\newline
 $(iii)$~The spectrum of $T$ is contained in $\{\lambda_{\rm max}\}\cup [-2,\lambda_2]$, where $\lambda_2\leq 0.772216$.
 \end{lem}
\begin{proof} By  \eqref{opT}: $T=\lambda_{\rm max}(\id-\sum_{n\in \Z} d(\vert n\vert)\bp_{\alpha_n})=\lambda_{\rm max}(\id-G)$, where    $ G$ is given in \eqref{G}.

Let $V:\ell^2(\Z)\to \cH$ be the linear map:  $V(\delta_n)=\zeta_n$ $\forall n\in \Z$, where $\delta_n$ is the canonical basis. One has by definition 
of the matrix $J$ ($J_{i,j}:=\langle \zeta_i \mid \zeta_j\rangle$), that 
$$
\langle V\eta \mid V\eta'\rangle=\langle \eta \mid J\eta'\rangle=\langle J^{\frac 12}\eta \mid J^{\frac 12}\eta'\rangle.
$$
This shows that $U:=VJ^{-\frac 12}:\ell^2(\Z)\to \cH$ is a unitary operator. 
The spectrum of $G$ (\eqref{G})  as an operator in $\cH$ is the same as the spectrum of the matrix $U^*GU=J^{-\frac 12}L J^{-\frac 12}$, where $L=V^*GV$ thus 
$$
L_{i,j}=\langle \delta_i\mid V^*GV \delta_j\rangle=\langle \zeta_i \mid G\zeta_j\rangle
=\sum_{n\in \Z} d(\vert n\vert)\langle \zeta_i \mid \zeta_n\rangle \langle\zeta_n\mid \zeta_j\rangle.
$$ 
The spectrum of $G$ is thus the same as that of the conjugate matrix $LJ^{-1}$. One has 
$$
(LJ^{-1})_{i,k}=\sum_{n,j\in \Z} d(\vert n\vert)\langle \zeta_i \mid \zeta_n\rangle \langle\zeta_n\mid \zeta_j\rangle (J^{-1})_{j,k}=\sum_{n\in \Z} d(\vert n\vert)\langle \zeta_i \mid \zeta_n\rangle \delta_{n,k},
$$
since for any $n\in \Z$ one has: $\sum \langle\zeta_n\mid \zeta_j\rangle (J^{-1})_{j,k}=(J \,J^{-1})_{n,k} =\delta_{n,k}$. Thus $(LJ^{-1})_{i,k}=d(\vert k\vert)\langle \zeta_i \mid \zeta_k\rangle$ and this proves $(i)$.\newline 
$(ii)$~For $\vert n\vert>m$ and $\vert k\vert>m$, one has: $d(\vert k\vert)=1$, $\zeta_n=\xi_n$, $\zeta_k=\xi_k$, so  $A_{n,k}=d(\vert k\vert)\langle \zeta_n \mid \zeta_k\rangle=\delta_{n,k}$. The entries $A_{i,j}-A^{(N)}_{i,j}$ of the matrix $A-A^{(N)}$ are non-zero only in the range\footnote{$[-N,N]^c$ denotes the complement of $[-N,N]$}
$$(i,j)\in [-N,N]\times [-N,N]^c \cup [-N,N]^c\times  [-N,N].
$$ 
Thus the operator norm of $A-A^{(N)}$ is less than the sup of the norms of the two blocks corresponding to $[-N,N]\times [-N,N]^c$ and $[-N,N]^c\times  [-N,N]$. In turns, the operator norm of these blocks is majored by their Hilbert-Schmidt norm, whose square is 
$$
\sum_{\vert i\vert\leq N,\vert j\vert> N} \vert A_{i,j}\vert^2=\sum_{\vert i\vert\leq N}  \Vert \zeta_i-P(N)\zeta_i\Vert^2, \ \sum_{\vert i\vert> N,\vert j\vert \leq  N} \vert A_{i,j}\vert^2=\sum_{\vert j\vert\leq N} d(\vert j\vert)^2 \Vert \zeta_j-P(N)\zeta_j\Vert^2.
$$
Since we assume $N>m$, one has $\zeta_0=P(N)\zeta_0$, and by \eqref{majorred}, for $ j\neq 0, \vert j\vert\leq N$,
$$
\Vert \zeta_j-P(N)\zeta_j\Vert^2=\epsilon(j,N):= \pi ^{-2} \sin ^2(\pi \alpha_j)\left((\log \Gamma)^{(2)}(N-\alpha_j)+(\log \Gamma)^{(2)}(\alpha_j+N) \right). 
$$
We thus obtain the following control on the operator norm 
\begin{equation}\label{aminusan}
	\Vert A-A^{(N)}\Vert \leq \max (e(N),e'(N)), \ e(N)^2=\sum_{\vert j\vert\leq N}\epsilon(j,N), \ \ e'(N)^2=\sum_{\vert j\vert\leq N}d(\vert j\vert)^2\epsilon(j,N).
\end{equation}
Let $J^{(N)}$ be the  matrix  defined by: $J^{(N)}_{i,j}=J_{i,j}$ if
 $\vert i\vert \leq N$,  $\vert j\vert \leq N$ and  	$J^{(N)}_{i,j}=\delta_{i,j}$ otherwise. By the same argument as above one obtains 
 \begin{equation}\label{jminusjn}
	\Vert J-J^{(N)}\Vert \leq  e(N).
	\end{equation}
	The matrices $J$ and  $J^{(N)}$ are strictly positive. Let $0<r<1<s$ be such that 
	\begin{equation}\label{apriori}
	\Spec J\subset [r,s], \qquad \Spec J^{(N)}\subset [r,s], \qquad  \Vert A\Vert \leq s,\quad  \Vert A^{(N)}\Vert \leq s.
	\end{equation}
 We shall provide the numerical values of $r,s$ later. Now we estimate the norm difference of the  matrices $\Pi=J^{-\frac 12}AJ^{\frac 12}$ and $\Pi_N=(J^{(N)})^{-\frac 12}A^{(N)}(J^{(N)})^{\frac 12}$ which, as shown below, are both positive. We use the equality for strictly positive operators $X$, $X'$
	$$
	X^{-\frac 12}=\frac{1}{\pi}\int_0^\infty (\lambda +X)^{-1}\lambda^{-\frac 12}d\lambda, \ \ X^{-\frac 12}-X'^{-\frac 12}=
	\frac{1}{\pi}\int_0^\infty (\lambda +X)^{-1}(X'-X) (\lambda +X')^{-1}\lambda^{-\frac 12}d\lambda
	$$
	This gives for $X, X'\geq r>0$ the estimate: $\Vert X^{-\frac 12}-X'^{-\frac 12}\Vert\leq \frac 12 \, r^{-\frac 32}\Vert X-X'\Vert $. Similarly
	$$
	X^{\frac 12}-X'^{\frac 12}=
	\frac{1}{\pi}\int_0^\infty (\lambda +X')^{-1}(X-X') (\lambda +X)^{-1}\lambda^{\frac 12}d\lambda
	$$
	which gives the estimate: $\Vert X^{\frac 12}-X'^{\frac 12}\Vert\leq \frac 12 \, r^{-\frac 12}\Vert X-X'\Vert $. Then we obtain using \eqref{apriori}, \eqref{aminusan}, \eqref{jminusjn} 
	$$
	\Vert \Pi-\Pi_N \Vert \leq  \Vert (J^{-\frac 12}-(J^{(N)})^{-\frac 12})AJ^{\frac 12}\Vert +  \Vert (J^{(N)})^{-\frac 12}(A-A^{(N)})J^{\frac 12}\Vert+
	\Vert (J^{(N)})^{-\frac 12}A^{(N)}(J^{\frac 12}-(J^{(N)})^{\frac 12})\Vert
	$$
	so that
	\begin{equation}\label{apriori1}
	\Vert \Pi-\Pi_N \Vert \leq  \frac 12 \, (s/r)^{\frac 32}e(N)+(s/r)^{\frac 12} \max (e(N),e'(N))+ \frac 12 \, (s/r)e(N):=\epsilon_1(N).
	\end{equation}
	The operator  $\id-\Pi$ is compact and self-adjoint, since it corresponds to the matrix $U^*GU=J^{-\frac 12}L J^{-\frac 12}$. The operators $A^{(N)}$, $J^{(N)}$ decompose as direct sums in the decomposition $\ell^2(\Z)=\ell^2([-N,N])\oplus \ell^2([-N,N]^c)$ and both act as identity in $\ell^2([-N,N]^c)$. Their actions in $\ell^2([-N,N])$ are given respectively by the matrices $d(\vert j\vert)\langle \zeta_i \mid \zeta_j\rangle$ and $\langle \zeta_i \mid \zeta_j\rangle$. Thus one derives  that the operator $J^{(N)}$ is positive and that, as above, 
	$$
	\sum_{n,j\in [-N,N]} d(\vert n\vert)\langle \zeta_i \mid \zeta_n\rangle \langle\zeta_n\mid \zeta_j\rangle ((J^{(N)})^{-1})_{j,k}=\sum_{n\in [-N,N]} d(\vert n\vert)\langle \zeta_i \mid \zeta_n\rangle \delta_{n,k},
	$$
	which shows that both $A^{(N)}J^{(N)}$ and $\Pi_N=(J^{(N)})^{-\frac 12}A^{(N)}(J^{(N)})^{\frac 12}$ are positive operators. Thus the operator  $\id-\Pi^{(N)}$ is compact and self-adjoint. Moreover, 
	 the norm of the difference $\Pi-\Pi_N$ is majored by $\epsilon_1(N)$. It follows (see \cite{Simon} Theorem 1.7)
	that the eigenvalues $\lambda_j, \lambda_j^{(N)}$ of $\id-\Pi$ and $\id-\Pi^{(N)}$  arranged in decreasing order, fulfill the inequality:  $\vert\lambda_j- \lambda_j^{(N)}\vert \leq \epsilon_1(N)$, $\forall j$. But these eigenvalues are the same as those of the conjugate operators $\id-A$ and $\id-A^{(N)}$. It remains to determine $0<r<1<s$ such that \eqref{apriori} holds. Note first that if $0<r<1<s$ are chosen so  that 
	\begin{equation}\label{apriori2}
	\Spec J\subset [r,s], \qquad  \Vert A\Vert \leq s, 
	\end{equation}
	then \eqref{apriori} holds since $J^{(N)}=P(N)JP(N)+(1-P(N))$; so  $r\leq J\leq s$ implies $r\leq J^{(N)}\leq s$. Also one has: $A^{(N)}=P(N)AP(N)+(1-P(N))$, so that $\Vert A^{(N)}\Vert \leq \Vert A\Vert$. One then takes $N=2000$ and uses \eqref{aminusan} and \eqref{jminusjn}. Note that the matrix $A^{(N)}$ is not self-adjoint and to bound its norm one uses its decomposition as a sum of a symmetric and an antisymmetric matrix, together with the computation of the eigenvalues of both, which gives the upper bounds $1.533$ and $0.0285$ for their norms. The value of $\max (e(N),e'(N))$ for $N=2000$ is $\sim 0.017$ which provides the bound $s=1.578$. One finds that the eigenvalues of $J^{(N)}$, for $N=2000$, are inside the interval $[0.313,1.346]$ and using $e(N)\sim 0.0145$ one gets $r=0.299$. This gives $s/r\sim 5.27$ and,  by \eqref{apriori1}, $\epsilon_1(N)\leq 11\, \epsilon(N)$: the required bound.\newline
	$(iii)$~We use $(ii)$ and take $N=10 000$. One gets $\epsilon(N) \sim 0.00740487$, while the first non-zero eigenvalue of  the matrix $A^{(N)}$ is $\beta_2^{(N)}=0.347112$. Thus by $(ii)$ the first non-zero eigenvalue of  the matrix $A$ is   $\beta_2\geq \beta_2^{(N)}-11 \epsilon(N)\sim 0.265658$. This shows that the second eigenvalue $\lambda_2=\lambda_{\rm max}(1-\beta_2)$ of $T$ fulfills  $\lambda_2\leq 0.772216$. 
  \end{proof}

\subsection{Proof of Theorem \ref{mainthmintro}}\label{sectproof}
The above computation of the spectrum of the compact operator $T$ together with Lemma \ref{approachk} and the estimate \eqref{computerverif} provide the needed information on the spectrum of the compact operator   $\kf_I$, since both $T$ and $\kf_I$ are self-adjoint. Then,  by \cite{Simon} Theorem 1.7, with their eigenvalues arranged in decreasing order and for $\epsilon_1\simeq 0.00122$, one has: 
\begin{equation}\label{spectral0}
\Vert \kf_I-T\Vert \leq \epsilon_1,  \qquad \vert \lambda_n(\kf_I)-\lambda_n(T)\vert \leq \epsilon_1.
\end{equation}
These bounds  allow one to transfer the results of \S \ref{sectspec} from $T$, to $\kf_I$. Up to some computational imprecision which we evaluate later, the results on $T$ are as follows
\begin{enumerate}
\item The three largest eigenvalues of $T$ are given by \eqref{opTeigen} \ie
 $$\lambda_{\rm max}=1.05158,\ \ \lambda_2=0.686494, \ \ \lambda_3=0.0288921.$$
\item The inner product of $\zeta$ with the constant function\footnote{normalized to be of  norm $1$} $\xi_0$ is $\sim 0.94865$.	
\end{enumerate}
Let  $P_\zeta$ be the orthogonal projection onto $\zeta^\perp:=\{ \eta \in \cH\mid \langle \zeta \mid \eta\rangle =0\}$. The spectral decomposition
 \begin{equation}\label{spectral1}
 T=	\lambda_{\rm max}\vert \zeta\rangle \langle \zeta \vert +R, \qquad R\leq \lambda_2\, P_\zeta
 \end{equation}
shows that the quadratic form associated to $\id - T$ is given, using $\id=\vert \zeta\rangle \langle \zeta\vert +P_\zeta$, by
\begin{equation}\label{spectral2}
 \langle \xi\mid  (\id - T)\xi\rangle =(1-	\lambda_{\rm max})\vert  \langle \zeta \mid \xi \rangle \vert^2 +\langle P_\zeta\xi\mid (P_\zeta - R)\xi\rangle  \end{equation}
and since $R\leq \lambda_
2\, P_\zeta
$, the last term fulfills 
\begin{equation}\label{spectral3}
 \langle P_\zeta\xi\mid (P_\zeta - R)\xi\rangle  \geq (1-\lambda_2) \Vert P_\zeta\xi\Vert^2.
 \end{equation} 

Next lemma shows how to restore positivity in a quadratic form which is positive on a codimension one subspace, by  adding a rank one quadratic form.

\begin{lem}\label{first} Let $\cH$ be a Hilbert space, $\phi,\psi\in \cH$ be unit vectors and $P_\phi$ the orthogonal projection on $\phi^\perp:=\{ \eta \in \cH\mid \langle \phi \mid \eta\rangle =0\}$. Let  $a,b,c\in \R_+$. Then the following quadratic form on $\cH$
$$
B(\xi):=-b \vert\langle \phi \mid \xi\rangle\vert^2+ a\vert\langle \psi \mid \xi\rangle\vert^2+c \Vert P_\phi(\xi)\Vert^2
$$
 is positive if and only if 
 \begin{equation}\label{posconds}
 a +c \geq b, \qquad 	b (a +c)\leq a  (b+c)\vert\langle \phi \mid \psi\rangle\vert^2.
 \end{equation}	
When \eqref{posconds} holds one has: $B(\xi)\geq \epsilon \Vert \xi \Vert^2$ $\forall \xi\in \cH$, where 
\begin{equation}\label{posconds2}
	2\epsilon=a-b+c-\left((a+b+c)^2-4a (b+c)\vert\langle \phi \mid \psi\rangle\vert^2 \right)^\frac 12.
\end{equation}
\end{lem}
\begin{proof} For any $\xi \in \cH$ one has
the orthogonal decomposition 
$$\xi=P_\phi(\xi)+\phi\, \langle \phi \mid \xi\rangle=P_\phi(\xi)+e_\phi(\xi). $$
 Let $\psi_1= e_\phi(\psi)$ and $\psi_2= P_\phi(\psi)$. We can assume $\psi_2\neq 0$ since otherwise $B$ is positive iff $a\geq b$. We can also assume that the scalar product $\alpha=\langle \phi \mid \psi\rangle$ is real. Let then $E\subset \cH$ be the two dimensional space generated by $\phi$ and $\psi$. Since by hypothesis $c\geq 0$, it follows that $B$ is positive iff its restriction to $E$ is positive. In the orthonormal basis of $E$ given by $(\phi, \psi_2/\Vert\psi_2\Vert)$ the matrix which represents $B$ is 
 $$
 \left(
\begin{array}{cc}
 a \alpha ^2-b & a \alpha  \beta  \\
 a \alpha  \beta  & a \beta ^2+c \\
\end{array}
\right), \qquad \alpha=\langle \phi \mid \psi\rangle, \quad \beta= \Vert\psi_2\Vert.
$$
It is real symmetric, hence its eigenvalues are real: they are both positive iff the trace and the determinant are positive. One also has: $\alpha ^2+\beta^2=\Vert \psi\Vert^2=1$. Thus the trace is $a+c-b$. The determinant is: $ca\alpha^2-ba \beta^2-bc=a\alpha^2(b+c)-b(a+c)$. The inequality $B(\xi)\geq \epsilon \Vert \xi \Vert^2$ $\forall \xi\in \cH$ then follows from the formula for the eigenvalues of the above matrix, with $\epsilon$ being the small one, and  the inequality $\epsilon \leq c$ which follows from 
$$
(a+b+c)^2-4a (b+c)\vert\langle \phi \mid \psi\rangle\vert^2-(-a+b+c)^2=4 a  (b+c)\left(1-\vert\langle \phi \mid \psi\rangle\vert^2\right).
$$
Finally, for $\xi=\xi_1+\xi_2$, $\xi_1\in E$, $\xi_2\in E^\perp$ one has $B(\xi)=B(\xi_1)+B(\xi_2)\geq \epsilon \Vert \xi_1 \Vert^2+c \Vert \xi_2 \Vert^2$.
 \end{proof}
 
 \begin{lem}\label{second} Let $\nf_I=-2\epsilon'(1_+)\left(\id - \kf_I\right)$  be the operator in $\cH=L^2(I)$, $I=[-\frac 12 \log 2,\frac 12 \log 2]$  which represents the quadratic form associated to $E_+\circ Q_+$ as in Proposition \ref{propcompact}. Then, with $\gamma\sim 2.94355$, 
  \begin{equation}\label{negativeNI}
  \langle   \xi\mid \nf_I(\xi)\rangle	 \leq \gamma\vert\langle \xi_0 \mid \xi\rangle\vert^2,
  \quad \forall\xi \in \cH.
  \end{equation} 	
 \end{lem}
\begin{proof} We first work with $T$ and apply Lemma \ref{first}, with $\phi=\zeta$, $\psi=\xi_0$.  
We determine the scalar $a>0$ to fulfill
\eqref{posconds} for $b=\lambda_{\rm max}-1$, $c=1-\lambda_2$ and the inner product of the two vectors given by $\langle \zeta\mid \xi_0\rangle$. Since for the above numerical values one has $c>b$, the condition $ a +c \geq b$ is automatic. The second condition of \eqref{posconds} is 
$$
a \left( (c+b)\vert\langle \zeta\mid \xi_0\rangle\vert^2-b\right)\geq bc.
$$
For the above numerical values one has  $b\sim 0.05158$,  $\langle \zeta\mid \xi_0\rangle\sim 0.94865$. By Lemma \ref{lemspec} one has $c>0.227784$, then   using \eqref{posconds2}, the following inequality becomes valid, for  
$a \sim 0.064$, $\epsilon_2 \sim 0.00441$
$$
\langle   \xi\mid (\id -T)\xi\rangle	 +a \vert\langle \xi_0 \mid \xi\rangle\vert^2\geq \epsilon_2 \Vert \xi \Vert^2,\qquad\forall \xi \in \cH.
$$
By \eqref{spectral0}, one has $\Vert \kf_I-T\Vert \leq \epsilon_1$, where $\epsilon_1\simeq 0.00122<\epsilon_2$,  thus
$$
\langle   \xi\mid (\id -\kf_I)\xi\rangle	 +a \vert\langle \xi_0 \mid \xi\rangle\vert^2\geq (\epsilon_2-\epsilon_1) \Vert \xi \Vert^2,\qquad\forall \xi \in \cH
$$ 
which gives \eqref{negativeNI} after multiplication by $-2\epsilon'(1_+)$. 
\end{proof} 

We can finally state our main result 

\begin{thm}\label{mainthmfine} Let $g\in C_c^\infty(\R_+^*)$ be a smooth function with  support in the interval $[2^{-1/2},2^{1/2}]$ and whose Fourier transform  vanishes at $-\frac i2$. Let $\bf S$ be the orthogonal projection of $L^2(\R)_{\rm ev}$ onto the subspace of even functions which vanish as well as their Fourier transform in the interval $[-1,1]$. Then 
\begin{equation}\label{maininequ}
	W_\infty(g*g^*)\geq \Tr(\rep(g)\, {\bf S}\,\rep(g)^*)  -c\,  \vert \widehat g(0)\vert^2, \qquad c=\frac{4 \gamma}{\log 2}\,. 
\end{equation}	
\end{thm}
\begin{proof} By Theorem \ref{devil} one has, for $f=g*g^*$, 
\begin{equation}\label{sonine1thm}
\tr(\rep(f){\bf S})=W_\infty(f)+\int f(\rho^{-1})\epsilon(\rho) d^*\rho=W_\infty(f)+E(f).
\end{equation}	Let $k(u):=u^{\frac 12}\int_{0 }^u v^{- \frac 12} g(v) \, d^*v$. One has: $0=\widehat g(-\frac i2)=\int_{0 }^\infty v^{- \frac 12} g(v) \, d^*v$, thus the  support of $k$ is contained in $[2^{-1/2},2^{1/2}]$. Moreover: $k=Y*g$, $k^*=Y^**g^*$, $k*k^*=Y^**Y*f$ where, as in Lemma \ref{vanishing1}, $Y(\rho)=0$ for $\rho<1$, $Y(\rho)=\rho^{\frac 12}$ for $\rho\geq 1$ and $ Y^*(\rho)=Y(\rho^{-1})$. Thus one gets
\begin{equation}\label{sonine1thm1}
Q(k*k^*)=g*g^*=f, \qquad \widehat k(0)= -2  \,\widehat g(0)
\end{equation}
where the second equality follows using integration by parts from
$$
\widehat k(0)=\int_0^\infty k(u)d^*u=\int_0^\infty \left( \int_{0 }^u v^{- \frac 32} g(v) \, dv\right) u^{- \frac 12}du=-\int_0^\infty u^{- \frac 32} g(u) 2 u^{\frac 12} du = -2  \,\widehat g(0).
$$ 
One thus obtains
$$
\int f(\rho^{-1})\epsilon(\rho) d^*\rho=E\circ Q( k*k^*).
$$
Let $\xi(x):=k(\exp(x))$. One has: $\xi \in \cH=L^2([-\frac 12 \log 2,\frac 12 \log 2])$
and using \eqref{quadform} of Proposition \ref{propcompact}
$$
E\circ Q( k*k^*)=E_+(Q_+(\xi*\xi^*))= \langle   \xi\mid \nf_I(\xi)\rangle.
$$
Then, by Lemma \ref{second} one gets, using   $\langle \xi_0 \mid \xi\rangle=(\log 2)^{-1/2} \widehat \xi(0)=(\log 2)^{-1/2} \widehat k(0)$
 $$
E(f)= E\circ Q( k*k^*)= \langle   \xi\mid \nf_I(\xi)\rangle\leq  \gamma\vert\langle \xi_0 \mid \xi\rangle\vert^2= \frac{\gamma}{\log 2} \vert \widehat k(0)\vert^2=\frac{4 \gamma}{\log 2} \vert\widehat g(0)\vert^2, $$
 which gives the required inequality. \end{proof}
 
 \begin{rem}\label{remlowc} By \eqref{spectral0}, the first eigenvalue $\lambda_1(\kf_I)$ of  $\kf_I$ fulfills $ \vert \lambda_1(\kf_I)-\lambda_{\rm max}\vert \leq \epsilon_1 
$ while $\lambda_{\rm max}=1.05158$  and $\epsilon_1\simeq 0.00122$. Thus, since $C_c^\infty((-\frac 12 \log 2,\frac 12 \log 2))$ is dense in the Hilbert space $\cH=L^2([-\frac 12 \log 2,\frac 12 \log 2])$, there exists a unit vector  $\xi\in C_c^\infty((-\frac 12 \log 2,\frac 12 \log 2))$ such that $\kf_I(\xi)\sim \lambda_1(\kf_I) \xi$.  It follows, using $\nf_I=-2\epsilon'(1_+)\left(\id - \kf_I\right)$, that
$$
\langle   \xi\mid \nf_I(\xi)\rangle	\geq 2\epsilon'(1_+) (1.05-1)\Vert \xi\Vert^2
 \geq 0.1\,\epsilon'(1_+) \vert\langle \xi_0 \mid \xi\rangle\vert^2
$$
 Let then $h(\rho):=\xi(\log \rho)$	and $g(\rho):=(\frac 12-\rho\partial_\rho)h(\rho)$, so that  $g(\rho)=\eta(\log \rho)$ for $\eta =\frac 12 \xi -\xi'$. One has $\widehat g(-\frac i2)=0$ and as in the proof of Theorem \ref{mainthmfine} one obtains 
 $$
 E(g*g^*)=E\circ Q( h*h^*)=\langle   \xi\mid \nf_I(\xi)\rangle\geq 0.1\,\epsilon'(1_+)\vert\langle \xi_0 \mid \xi\rangle\vert^2> 13 \vert\widehat g(0)\vert^2.
 $$
 Thus by \eqref{sonine1thm} one gets 
 $$
 W_\infty(g*g^*)= \Tr(\rep(g)\, {\bf S}\,\rep(g)^*)-E(g*g^*)<\Tr(\rep(g)\, {\bf S}\,\rep(g)^*)-13 \vert\widehat g(0)\vert^2.
 $$
 This shows that the best constant $c$ fulfilling \eqref{maininequ} is such that  $13<c<17$.
 \end{rem}

\appendix

\section{Fourier versus Mellin transforms}\label{appenmellinapp}
We use the convolution algebra $C_c^\infty(\R_+^*)$ of smooth functions with compact support on the multiplicative group $\R_+^*$. Its convolution product and involution are given by 
\begin{equation}\label{fouriermellinapp}
(f* g)(u):=\int f(v)g(u/v)d^*v, \qquad f^*(u):=\overline{f(u^{-1})}.
\end{equation}
The (multiplicative) Fourier transform of $f$ (see \eqref{PhiFourier})
$$
\hat f(s)=\int_0^\infty f(u)u^{-is}d^*u
$$
transforms convolution into pointwise product and the involution into the pointwise complex conjugation,  $s\in \R$. For complex values of $s$, the evaluation $f\mapsto \hat f(s)$ is  still multiplicative but no longer compatible with the involution. \newline
The translation to formulas using the Mellin transform is done in \eqref{fouriermellin} and uses the isomorphism 
\begin{equation}\label{MF1}
k\stackrel{\sim}{\mapsto} \Delta^{1/2}(k)=f, \qquad f(x):= x^{1/2}k(x)
\end{equation}
which respects the convolution product, and transforms $x^{-1}k(x^{-1})$ into $f(x^{-1})$. Hence, after taking complex conjugates,  the natural involution $k\mapsto \bar k^\sharp$ becomes $f\mapsto f^*$ . 
The Mellin transform $\tilde k(z):=\int_0^\infty k(u)u^zd^*u $  is related to the (multiplicative) Fourier transform of $f$ by \eqref{fouriermellin} \ie
\begin{equation}\label{fouriermellinapp1}
\tilde k(\frac 12+is)=\int_0^\infty k(u)u^{\frac 12+is}d^*u =\int_0^\infty f(u)u^{is}d^*u=\hat f(-s),
\end{equation}
where the sign  in $-s$ is due to  the convention for the multiplicative Fourier transform \eqref{PhiFourier}.

\section{Explicit formula}\label{appendix2}
 In this  appendix, we gather different sources on the normalization of the archimedean contribution to the explicit formula. Following \cite{EB},  one defines the Mellin transform of a function $f\in C^\infty_c(\R_+^*)$ as
\begin{equation}\label{mellin}
 \tilde f(s):=\int_0^\infty f(x)x^{s-1}dx.
 \end{equation}
 Then, with $f^\sharp(x):=x^{-1}f(x^{-1})$  the explicit formula takes the form 
 \begin{equation}\label{bombieriexplicit}
 \sum_{\rho}\tilde f(\rho)=\int_0^\infty f(x)dx+\int_0^\infty f^\sharp(x)dx-\sum_v {\mathcal W}_v(f),
 \end{equation}
 where $v$ runs over all places  $\{\R,2,3,5,\ldots \}$ of $\Q$,  the  sum on the left hand side is over all complex zeros $\rho$ of the Riemann zeta function, and for $v=p$ 
 \begin{equation}\label{bombieriexplicit1}
 {\mathcal W}_p(f)=(\log p)\sum_{m=1}^\infty\left(f(p^m)+f^\sharp(p^m)\right).
 \end{equation}
 The archimedean distribution is defined as
 \begin{equation}\label{bombieriexplicit2}
 {\mathcal W}_\R(f):=(\log 4\pi +\gamma)f(1)+\int_{1}^\infty\left(f(x)+f^\sharp(x)-\frac 2x f(1)\right)\frac{dx}{x-x^{-1}}.
 \end{equation}
 One then has
  \begin{equation}\label{bombieriexplicit3}
 {\mathcal W}_\R(f)=(\log \pi)f(1)-\frac{1}{2\pi i}\int_{1/2+iw}\Re\left(\frac{\Gamma'}{\Gamma}\left(\frac w2\right)\right)\tilde f(w)dw.
\end{equation}
  In \cite{EB0} and \cite{Burnol}  a positivity result for the distribution ${\mathcal W}_\infty=-{\mathcal W}_\R$ is proven, (for test functions with support in a small enough interval around $1$),  by writing  the distribution ${\mathcal W}_\infty$  in terms of the Mellin transform of the test function as follows
   \begin{equation}\label{burnolexplicit}
 {\mathcal W}_\infty(f)=\int_{w=1/2+i\tau}h_+(\tau)\tilde f(w)\frac{d\tau}{2\pi}.
\end{equation}  
  The function $h_+(\tau)$  is 
  \begin{equation}\label{burnolexplicit1}
h_+(\tau)=-\log \pi + \Re(\lambda(1/4 +i\tau/2)), \qquad \lambda(z)=\Gamma'(z)/\Gamma(z).
\end{equation}
It is the derivative of $2\,\theta(\tau)$, where $\theta$ is  the Riemann-Siegel angular function
 defined as 
\begin{equation}\label{riesie}
\theta(E) = - \frac{E}{2} \log \pi + \Im \log \Gamma \left(
\frac{1}{4} + i \frac{E}{2} \right),
\end{equation}
with $\log \Gamma(s)$, for $\Re (s)>0$, the branch of the $\log$
which is real for $s$ real.

\section{Positivity criterion}\label{apppositivity}
We follow \cite{yoshida} and state the following equivalence, using the Mellin transform

\begin{prop}\label{mainprop} Let $Z\subset \C$ be the set of non-trivial zeros of the Riemann zeta function and $F \subset \C$ a finite set disjoint from $Z$ and containing $\{0,1\}$, then
\begin{equation}\label{appendweilnegav}
RH \iff \sum_v {\mathcal W}_v(g*\bar g^\sharp)\leq 0, \quad \forall g\in C_c^\infty(\R_+^*)\mid \tilde g(z)=0,\ \forall z\in F.
\end{equation} 	
\end{prop}
\begin{proof} The implication ``$\Rightarrow$" follows from the explicit formula \eqref{bombieriexplicit} and the hypothesis $\{0,1\}\subset F$. Conversely, the proof of Proposition 1 of \cite{yoshida} applies verbatim, provided one first refines  the proof of Lemma 1 of \opcit by showing that, given $\epsilon >0$ and   $\rho_0\in Z$, there exists $g_0 \in C_c^\infty(\R_+^*)$ such that 
$$
\tilde g_0(z)=0,\quad \forall z\in F; \quad \tilde g_0(\rho_0)=1; \quad \vert \tilde g_0(\rho)\vert \leq 
\epsilon /\vert \rho -\rho_0\vert^2,\quad \forall\rho \in Z, \ \rho \neq \rho_0.
$$
In order to fulfill the additional vanishing condition: $\tilde g_0(z)=0, \forall z\in F$, one adjoins $F$ to the finite set of zeros fulfilling $\vert \rho-\rho_0\vert <R$ (same notation as in Proposition 1 of \cite{yoshida}), and one then proceeds exactly as in \opcit \end{proof} 

\section{Quantized calculus redux}\label{appquantized}
Let $C$ be a  locally compact abelian group endowed with the proper homomorphism \[\Mod:C\to \R_+^*,\qquad \Mod(u)=\vert u\vert\quad  u \in C.\]
We let  $\widehat{C}$ be the
Pontrjagin dual of $C$ endowed with its Haar measure. The elements  $f\in L^\infty(\widehat{C})$ act
as multiplication operators on the Hilbert space $\cH :=
L^2(\widehat{C})$. We define the ``quantized" differential of
$f$ to be the operator
\begin{equation}\label{qdd}
\qd\,f:= [H,\,f]= Hf-fH,
\end{equation}
where the operator $H$ on $\cH$ is 
\begin{equation}\label{HilbTrsf}
 H:= 2 \fourier_C\,{\bf 1}_P\,\fourier_C^{-1} -1,
\end{equation}
where $\fourier_C: L^2(C) \to \cH$ is the Fourier transform,  and
${\bf 1}_P$ is the multiplication by the characteristic function of
the set $P=\{u \in C\,| \, \vert u\vert\geq 1\}$. 

We take the case $C=\R$ with module $\exp: \R \to \R_+^*$ considered in this paper and identify the dual $\widehat C \sim \R$ using the bi-character $\nu(s,t):=\exp(-ist)$ which corresponds to \eqref{FwIPhi0} under the isomorphism given by the module. We give a "geometric" proof of the following (see  \cite{Co-book} Chapter IV for the general theory, a compact operator has infinite order when its characteristic values form a sequence of rapid decay; this implies that it is of trace class).
\begin{lem} \label{quantsmooth} For $f\in \cS(\widehat C)$ the quantized differential $[H,f]$ is an infinitesimal of infinite order and in particular a trace class operator.	
\end{lem}
\proof Let us work in the Hilbert space $L^2(C)$ so that the action of $K=\fourier_C^{-1} f\fourier_C$ is a convolution operator with Schwartz kernel $k(x,y)=\widehat f(x-y)$. The projection ${\bf 1}_P$ is the multiplication by the characteristic function of the halfline $[0,\infty]$. It is enough to show that the operator $PK(1-P)$ is of infinite order. After precomposition with the symmetry $\sigma(\xi)(x):=\xi(-x)$ the Schwartz kernel of $PK(1-P)\sigma$ is $h(x,y)=P(x)\widehat f(x+y)P(y)$. Let $\phi\in C^\infty(\R)$ be a smooth function which is identically $0$ for $x\leq -1$ and identically $1$ for $x\geq 0$. Let $g=\widehat f$ and $T$ be the operator in $L^2(\R)$ given by 
$$
T\xi(x):=\int \phi(x)g(x+y)\phi(y)\xi(y)dy
$$ 
One has $PTP=PK(1-P)\sigma $ and thus it is enough to show that $T$ is of infinite order. Let $A:=-\partial_x^2 +x^2$ be the harmonic oscillator, it is enough to show that the operator $A^nT$ is bounded for any $n>0$. Indeed the  eigenvalues of $A$ are the positive integers with multiplicity $1$ and the above boundedness ensures that the characteristic values of $T$ are of rapid decay. Now the  Schwartz kernel of $A^nT$ is a finite linear combination of products of the form
$$
\phi(y)\phi(x)^{(\ell')} x^k g^{(\ell)}(x+y)
$$
Since $g=\widehat f\in \cS(\R)$ the derivatives $g^{(\ell)}$ are of rapid decay and for any given $m>0$ one has an inequality of the form $\vert g^{(\ell)}(a)\vert\leq C_m (3+a)^{-m} $ for all $a\geq -2$. it follows that one controls the Hilbert Schmidt norm by the square root of the integral 
$$
C_m^2\int_{-1}^\infty \int_{-1}^\infty \vert \phi(y)\phi(x)^{(\ell')}\vert^2\vert x\vert^{2k}(3+x+y)^{-m}dx dy
$$
which is finite for $m$ large enough. This shows as required that $A^nT$ is bounded for any $n>0$.\endproof 
\begin{rem}\label{2proofs} One can give two alternate proofs of Lemma \ref{quantsmooth}. The first uses the conformal invariance of the quantized calculus to get a unitary operator $U:L^2(S^1)\to L^2(\R)$ of the form $ (U\xi)(t):=\frac{\pi^{-1/2}}{t+i}\ \xi(\frac{t-i}{t+i})$ which conjugates, up to sign,  the Hilbert transform $H$ (which acts in $L^2(\widehat C)$) with the operator $2P_{H^2}-1$ where $P_{H^2}$ is the orthogonal projection on boundary values of holomorphic functions. Moreover the conjugate of the multiplication operator by $f\in \cS(\R)$ is the multiplication by the smooth function $g(z)=f(i(1+z)/(1-z))$. Then the result follows since the quantized differential of a smooth function $g\in C^\infty(S^1)$ is of the form $2\sum \widehat g(n)[P_{H^2},z^n]$ which is of infinite order because the $\widehat g(n)$ are of rapid decay. The fact that $g\in C^\infty(S^1)$ comes from the smoothness os the extension of a Schwartz function to $P^1(\R)$ by the value $0$ at $\infty$. \newline
Another instructive alternate proof is to estimate directly, for $f\in \cS(\R)$, the  Schwartz kernel given by 
$$
k(x,y)=\frac{f(x)-f(y)}{x-y}.
$$	
\end{rem}

\section{Signs and normalizations} \label{appendixsigns}
We follow \cite{tate}  and use the
classical formula expressing the Fourier transform as a composition
of the inversion
\begin{equation}\label{inversionI}
I(f)(s):= f(s^{-1})
\end{equation}
and a multiplicative convolution operator.
In our framework the unitary $u$ is given by the ratio of archimedean local  factors on the critical line 
\begin{equation}\label{uFourier2}
u(s)=\,\frac{\pi^{-z/2}\Gamma(z/2)}{\pi^{-(1-z)/2}\Gamma((1-z)/2)},\qquad
z=1/2+is.
\end{equation}
In terms of the Riemann-Siegel angular function \eqref{riesie}, one has 
\begin{equation}\label{RieSiegel1}
u(s)= e^{2\, i\, \theta(s)},
\end{equation}
so that the function $Z(t):=e^{i\theta(t)}\zeta(\frac 12+it)$ is real valued. Indeed, this  follows from the functional equation since the complete zeta function $\zeta_\Q(z):=\pi^{-z/2}\Gamma(z/2)\zeta(z)$ is real valued on the critical line. 
The quantized differential $\ \qd f$ of $f$ is given by the kernel
\begin{equation}\label{qd1}
k(s,t)= \frac{i}{ \pi} \, \frac{f(s)-f(t)}{s-t}.
\end{equation}
Thus when one takes the logarithmic derivative of $u$ one obtains on the diagonal 
\begin{equation}\label{qdofu}
u^*\,\qd u(s)=e^{-2\, i\, \theta(s)}\frac{i}{ \pi}\partial_s e^{2\, i\, \theta(s)}=-\frac{2}{ \pi}\partial_s\theta(s)~\Longrightarrow ~\frac 12 u^*\,\qd u(s)=\frac{-2\partial_s\theta(s)}{2 \pi}.
\end{equation}
Thus one can write
\begin{equation}\label{qdofu1}
\Tr (\hat h_1\left(\frac{1}{2} u^{-1}\,  \qd \,u\right)\hat h_2)=-\int \hat h_1(t)\hat h_2(t)\frac{2\partial_t\theta(t)}{2 \pi}dt.
\end{equation} 
This corresponds to \eqref{burnolexplicit} since ${\mathcal W}_\R=-{\mathcal W}_\infty$ and to  the semi-local trace formula 
\begin{equation}\label{semilocTrace0}
\Tr (\hat h_1\left(\frac{1}{2} u^{-1}\,  \qd \,u\right)\hat h_2) =  \sum_{v \in S}
\int'_{\Q^*_v} \frac{\vert w \vert^{1/2} }{ \vert 1-w \vert} \,h(w) d^* w, \quad h=h_1*h_2,
\end{equation}
for the single archimedean place.  

   \section{Issues of convergence}\label{appendix-cv} 
   We gather several inequalities which ensure the convergence of the series \eqref{sonineQ} of Proposition \ref{propQe}. We first consider the terms 
   $$
 A_n(\rho):= \frac{\lambda(n)}{\sqrt{1-\lambda(n)^2}}\  \rho^{1/2}\int_{\rho^{-1}}^1(D_u\xi_n)(x)(D_u\zeta_n)(\rho x)dx.
   $$
   We estimate the integral using Schwarz's inequality
   $$
    \bigg\lvert\rho^{1/2}\int_{\rho^{-1}}^1(D_u\xi_n)(x)(D_u\zeta_n)(\rho x)dx\bigg\rvert \leq \left(\int_{\rho^{-1}}^1(D_u\xi_n)(x)^2 dx\right)^{\frac 12}
    \left(\int_{\rho^{-1}}^1(D_u\zeta_n)(\rho x)^2 \rho dx\right)^{\frac 12}.
   $$
   One has, using $D_u(f)(x)=x\partial_xf(x)$
   $$
   \int_{\rho^{-1}}^1(D_u\zeta_n)(\rho x)^2 \rho dx=\int_1^\rho (D_u\zeta_n)(y)^2dy\leq \rho^2 \int_1^\rho (\partial_y\zeta_n)(y)^2dy.
   $$
   With $\zeta_n(x)=\frac{1}{\sqrt{1-\lambda(n)^2}}\,\eta_n(x)$ and $\eta_n=\fourier_{e_\R}\xi_n$ one thus obtains
    $$
    \int_1^\rho (\partial_y\zeta_n)^2(y)dy=\frac{1}{1-\lambda(n)^2}\int_1^\rho (\partial_y\eta_n)^2(y)dy\leq \frac{1}{1-\lambda(n)^2} (2 \pi)^2,
    $$
since $\partial_y\eta_n$ is the Fourier transform of $2 \pi i x \xi_n(x)$ whose $L^2$-norm is bounded by $2\pi$. We thus get 
$$
 \left(\int_{\rho^{-1}}^1(D_u\zeta_n)^2(\rho x) \rho dx\right)^{\frac 12}\leq \rho \frac{2\pi}{\sqrt{1-\lambda(n)^2}}.
 $$
 To estimate  $\int_{\rho^{-1}}^1(D_u\xi_n)(x)^2 dx$, we rewrite the equality \eqref{WLambdaq} as follows \begin{equation}
({\bf W}f)(x) = -\left(1-x^2\right) f''(x)+2 x f'(x)+4 \pi ^2 x^2 f(x),
\end{equation}
so that since $\xi_n$ is an eigenvector of ${\bf W}$ (\ie  ${\bf W}\xi_n=\chi_{2n}^{2\pi}\xi_n$), using the notations of \cite{Wang}, we get 
$$
D_u(\xi_n)(x)=\frac 12 \left(1-x^2\right) \xi_n''(x)+(\chi_{2n}^{2\pi}-2 \pi ^2 x^2) \xi_n(x).
$$ 
 Assuming $n\geq 3$ to ensure  $\chi_{2n}^{2\pi}\geq 2 \pi ^2 $ one then derives 
$$
\Vert D_u(\xi_n)\Vert \leq \chi_{2n}^{2\pi}+ \frac 12\left(\int_{0}^1 (\xi_n''(x))^2(1-x^2)^2dx\right)^{\frac 12}.
$$
 By \cite{Wang} (Theorem 3.6) one has (note the different normalization of inner product due to \eqref{innerltwoeven})
\begin{equation}\label{boundWang}
\left(\int_{0}^1 (\xi_n''(x))^2(1-x^2)^2dx\right)^{\frac 12}\leq (2n)^2 + (6 \pi + 1)2n + 3(2\pi + 1)^2,
\end{equation}
while the eigenvalues  $\chi_{2n}^{2\pi}$ fulfill (see \opcit): $\chi_{2n}^{2\pi}\leq 2n(2n+1)+(2\pi)^2$. Thus, one obtains  the inequality 
$$
\Vert D_u(\xi_n)\Vert \leq 8 n^2 +(6 \pi  + 2)2n+ 16 \pi^2+12 \pi +1
$$
and the following uniform bound (take $\rho\leq 2$)
\begin{equation}\label{boundAn} 
	\vert A_n(\rho)\vert \leq \frac{\lambda(n)}{1-\lambda(n)^2}4 \pi (8 n^2 +(6 \pi  + 2)2n+ 16 \pi^2+12 \pi +1),\qquad \forall \rho, \  1\leq \rho\leq 2.
\end{equation}
We then consider the terms 
$$
B_n(\rho):=\frac{\lambda(n)}{\sqrt{1-\lambda(n)^2}}\ \left(\rho^{-1/2}(D_u\xi_n)(\rho^{-1})\zeta_n(1)-\rho^{1/2}\xi_n(1)(D_u\zeta_n)(\rho)\right).
$$
By \eqref{Rokh}, one has $\vert \xi_n(1)\vert \leq \sqrt{2n+\frac 12}$. One also has $\zeta_n=\frac{1}{\sqrt{1-\lambda(n)^2}}\,\eta_n$ and $\eta_n=\fourier_{e_\R}\xi_n$ thus 
$$
(D_u\zeta_n)(\rho)=\frac{1}{\sqrt{1-\lambda(n)^2}} \ \rho \eta_n'(\rho)~\Longrightarrow~ \vert (D_u\zeta_n)(\rho)\vert \leq \frac{8 \pi}{\sqrt{1-\lambda(n)^2}},\qquad \forall\rho, \  1\leq \rho\leq 2
$$
using the equality $\eta_n'(\rho)=-4\pi \int_0^1 \sin(2\pi \rho x) \xi_n(x)xdx$ and Schwarz's inequality. Hence
$$
\frac{\lambda(n)}{\sqrt{1-\lambda(n)^2}}\bigg\lvert\rho^{1/2}\xi_n(1)(D_u\zeta_n)(\rho)\bigg\rvert \leq 
\frac{\lambda(n)}{ 1-\lambda(n)^2}8\pi\sqrt{2(2n+\frac 12)}, \qquad \forall\rho, \  1\leq \rho\leq 2.
$$ 
 By \eqref{chirem0.5} one gets:  $\eta_n(x)=\lambda(n)\xi_n(x)$, for $x\in [0,1]$ thus one obtains by proportionality
$$
(D_u\xi_n)(\rho^{-1})\zeta_n(1)=\frac{1}{\sqrt{1-\lambda(n)^2}}(D_u\xi_n)(\rho^{-1})\eta_n(1)=\frac{1}{\sqrt{1-\lambda(n)^2}}(D_u\eta_n)(\rho^{-1})\xi_n(1)
$$
and the above bound for $\eta_n'(y)=-4\pi \int_0^1 \sin(2\pi y x) \xi_n(x)xdx$, applied for $y=\rho^{-1}$ thus gives 
$$
\vert(D_u\xi_n)(\rho^{-1})\zeta_n(1)\vert \leq \frac{4 \pi}{\sqrt{1-\lambda(n)^2}}\sqrt{2n+\frac 12}
$$
so that 
$$
\frac{\lambda(n)}{\sqrt{1-\lambda(n)^2}}\vert \rho^{-1/2}(D_u\xi_n)(\rho^{-1})\zeta_n(1) \vert \leq \frac{\lambda(n)}{ 1-\lambda(n)^2}4\pi\sqrt{2n+\frac 12},\qquad \forall\rho, \  1\leq \rho\leq 2.
$$
The above inequalities then give 
\begin{equation}\label{boundBn}
	\vert B_n(\rho)\vert \leq \frac{\lambda(n)}{1-\lambda(n)^2}\left(8\pi \sqrt 2+4\pi\right)\sqrt{2n+\frac 12}.
\end{equation}
We thus obtain 
\begin{lem}\label{lemesti} $(i)$~The series \eqref{sonineQ} of Proposition \ref{propQe} is convergent and the remainder (after replacing the infinite sum by the sum of the first $N$ terms) is majored as follows 
\begin{equation}\label{computersafe}
	\vert Q\epsilon(\rho)-\sum_0^N \frac{\lambda(k)}{\sqrt{1-\lambda(k)^2}} T_k(\rho)\vert \leq \sum_{N+1}^\infty\frac{2^{2 n+2} \pi ^{2 n+\frac{3}{2}} p(n)((2 n)!)^2}{(4 n)! \Gamma \left(2 n+\frac{3}{2}\right)} \end{equation}
	where $p(n)=16 n^2+8 (1+3 \pi ) n+(4+\sqrt{2}) \sqrt{4 n+1}+32 \pi ^2+24 \pi +2$.\newline
$(ii)$~For $N=10$,  the remainder is less than $2.366 \times  10^{-12}$ for any $\rho\in [1,2]$:
\begin{equation}\label{computersafe1}
\vert Q\epsilon(\rho)-\sum_0^{10} \frac{\lambda(k)}{\sqrt{1-\lambda(k)^2}} T_k(\rho)\vert \leq	2.366 \times  10^{-12},\qquad \forall \rho\in [1,2].
\end{equation}	
\end{lem}
\begin{proof} $(i)$~ follows from \eqref{boundAn} and \eqref{boundBn} which, together with \eqref{rapid-decay},  combine to yield for $n\geq 3$, 
\begin{align*}
&\bigg\lvert \frac{\lambda(n)}{\sqrt{1-\lambda(n)^2}} T_n(\rho)\bigg\rvert \leq 2\lambda(n)\left(\vert A_n(\rho)\vert+\vert B_n(\rho)\vert\right)\leq\\
&\leq  \frac{2^{2 n+2} \pi ^{2 n+\frac{3}{2}} \left(16 n^2+8 (1+3 \pi ) n+(4+\sqrt{2}) \sqrt{4 n+1}+32 \pi ^2+24 \pi +2\right) ((2 n)!)^2}{(4 n)! \Gamma \left(2 n+\frac{3}{2}\right)}
\end{align*}
which gives \eqref{computersafe}.\newline
$(ii)$To compute the upper bound of the right hand side of \eqref{computersafe} for $N=10$, one splits the sum in two, using the simple estimate $p(n)\leq 120 n^2$ for $n\geq 35$:
$$
\frac{2^{2 n+2} \pi ^{2 n+\frac{3}{2}} p(n)((2 n)!)^2}{(4 n)! \Gamma \left(2 n+\frac{3}{2}\right)}\leq \frac{15\ 2^{2 n+4} n^2 \pi ^{2 n+\frac{1}{2}} ((2 n)!)^2}{(4 n)! \Gamma \left(2 n+\frac{3}{2}\right)},\qquad \forall n\geq 35.
$$
With  $\nu_n$ the right hand side of this inequality, one obtains the relation 
$$
\nu_{n+1}/\nu_n=\frac{8 \pi ^2 (n+1)^3 (2 n+1)}{n^2 (4 n+1) (4 n+3)^2 (4 n+5)}=\frac{\pi ^2}{16 n^2}+\frac{\pi ^2}{32 n^3}+O\left(n^{-4}\right)
$$
and  $n^2\nu_{n+1}/\nu_n<1$ for all $n\geq 35$. One has $\nu_{35}\leq 5 \times  10^{-81}$, and thus using the trivial bound by the geometric series one gets
$$
\sum_{35}^\infty \nu_n\leq \frac{1225}{1224} \nu_{35}\leq 10^{-80}.
$$
One then simply computes the missing terms and they give 
$$
\sum_{11}^{34}\frac{2^{2 n+2} \pi ^{2 n+\frac{3}{2}} p(n)((2 n)!)^2}{(4 n)! \Gamma \left(2 n+\frac{3}{2}\right)}\sim 2.365 \times  10^{-12}
$$
Thus combining the above inequalities we obtain \eqref{computersafe1}.\end{proof}

\begin{rem}\label{boundwang1} For completeness, we give a short proof of an improved form of \eqref{boundWang}. As in \cite{Wang} (Equation (3.26)) one has, using integration by parts, the identity, for $f\in C^\infty([-1,1],\R)$, and $c=2 \pi$, using \eqref{WLambdaq}
\begin{align*}
\int_{-1}^1({\bf W}f)^2(x)dx&=\int_{-1}^1(1-x^2)^2\vert f''(x)\vert^2dx+2\int_{-1}^1(1-x^2)(1+c^2x^2)\vert f'(x)\vert^2dx +\\
&+c^2\int_{-1}^1(c^2 x^4+6x^2-2)\vert f(x)\vert^2dx.
\end{align*}
Applying this to $f=\xi_n$ and using ${\bf W}\xi_n=\chi_{2n}^{2\pi}\xi_n$ one gets	
$$
\int_{-1}^1 (\xi_n''(x))^2(1-x^2)^2dx\leq \int_{-1}^1({\bf W}\xi_n)^2(x)dx+2 c^2\int_{-1}^1\xi_n(x)^2dx
$$
providing the following improvement of \eqref{boundWang}
$$
\left(\int_{0}^1 (\xi_n''(x))^2(1-x^2)^2dx\right)^{\frac 12}\leq\sqrt{(\chi_{2n}^{2\pi})^2+2c^2}\leq 2n(2n+1)+(2\pi)^2(1+\sqrt 2).
$$
\end{rem}

\begin{acknowledgements}
The second author is partially supported by the Simons Foundation collaboration grant n. 353677.
\end{acknowledgements}


\begin{thebibliography}{99} 

\bibitem{BW}  M.~Bakonyi, H.~Woerdeman, \emph{Matrix completions, moments, and sums of Hermitian squares}. Princeton University Press, Princeton, NJ, 2011.

\bibitem{BKe0} M.~Berry and J.~Keating, \emph{$H=qp$ and the Riemann zeros}, ``Supersymmetry and Trace Formulae: Chaos and Disorder", edited by
J.P. Keating, D.E. Khmelnitskii and I.V. Lerner, Plenum Press.


\bibitem{BKe} M.~Berry and J.~Keating, \emph{
 The Riemann zeros and eigenvalue asymptotics}. SIAM Rev. \textbf{41} (1999), no. 2, 236--266.
 
 \bibitem{Boas}  R. P.~Boas, \emph{More inequalities for Fourier transforms}. Duke Math. J. \textbf{15} (1948), 105--109.
 
 
 


\bibitem{BK}  R. P.~Boas, M.~Kac,  \emph{Inequalities for Fourier transforms of positive functions}. Duke Math. J. \textbf{12} (1945), 189--206.







\bibitem{EB0} E.~Bombieri, \emph{ Remarks on Weil's quadratic functional in the theory of prime numbers}. I. (English, Italian summaries)
Atti Accad. Naz. Lincei Cl. Sci. Fis. Mat. Natur. Rend. Lincei (9) Mat. Appl. 11 (2000), no. 3, 183--233.
 
 
 \bibitem{EB} E.~Bombieri, \emph{The Riemann hypothesis}. The millennium prize problems, 107--124, Clay Math. Inst., Cambridge, MA, 2006.
 
 \bibitem{Burnol} J.~F.~Burnol, \emph{Sur les formules explicites. I. Analyse invariante} [On explicit formulae. I. Invariant analysis] C. R. Acad. Sci. Paris S\'er. I Math. \textbf{331} (2000), no. 6, 423--428.

  
 \bibitem{burnol2} J.~F.~Burnol, {\it Sur les espaces de Sonine associ\'es par de Branges \`a la transformation de Fourier}. C. R. Acad. Sci. Paris, Ser. I 335 (2002) 689--692.
  
 \bibitem{Co-book}  A.~Connes, {\it Noncommutative geometry}, Academic Press (1994).
 
\bibitem{Co-zeta} A.~Connes, \emph{Trace formula in noncommutative
geometry and the zeros of the Riemann zeta function}.  Selecta Math.
(N.S.)  \textbf{5}  (1999),  no. 1, 29--106.

\bibitem{Crh} A.~Connes,  \emph{An essay on the Riemann Hypothesis}. In ``Open problems in mathematics", Springer (2016), volume edited by Michael Rassias and John Nash.



 
 \bibitem{CC1} A. Connes, C. Consani, \emph{Schemes over $\mathbb F_1$ and Zeta Functions}, Compos. Math. \textbf{146} (2010), no. 6, 1383--1415.
 
 \bibitem{CC2} A. Connes, C. Consani, \emph{From Monoids to Hyperstructures: in Search of an Absolute Arithmetic}, Casimir force, Casimir operators and the Riemann hypothesis, 147--198, Walter de Gruyter, Berlin, 2010.





\bibitem{CCscal1} A.~Connes, C.~Consani, \emph{Geometry of the Scaling Site}.  Selecta Math. (N.S.) \textbf{23} (2017), no. 3, 1803--1850.

\bibitem{CCsurvey} A.~Connes, C.~Consani, \emph{The Riemann-Roch strategy, complex lift of the Scaling Site},  ``Advances in Noncommutative Geometry, On the Occasion of Alain Connes' 70th Birthday", Chamseddine, A., Consani, C., Higson, N., Khalkhali, M., Moscovici, H., Yu, G. (Eds.), Springer (2020). Available at \url{http://arxiv.org/abs/1805.10501}. ISBN 978-3-030-29596-7. 

\bibitem{scalingH} A.~Connes, C.~Consani, \emph{The Scaling Hamiltonian}, arXiv:1910.14368 

\bibitem{CMbook} A.~Connes, M.~Marcolli,  \emph{  Noncommutative Geometry, Quantum Fields, and Motives},
Colloquium Publications, Vol.55, American Mathematical Society, 2008.

\bibitem{toeplitz} A.~Connes, W.~van Suijlekom, \emph{Spectral truncations in noncommutative geometry and operator systems}, arXiv:2004.14115.



\bibitem{EGR} W.~Ehm, T.~Gneiting, D.~Richards, \emph{ Convolution roots of radial positive definite functions with compact support}
Trans. Amer. Math. Soc. \textbf{356}  (2004), no. 11, 4655--4685.

\bibitem{toulouse} E.~Hallouin, M.~Perret, \emph{A unified viewpoint for upper bounds for the number of points of curves over finite fields via Euclidean geometry and semi-definite symmetric Toeplitz matrices}. Trans. Amer. Math. Soc. \textbf{372} (2019), no. 8, 5409--5451.

\bibitem{XJL} Li, Xian-Jin, \emph{A generalization of A. Connes' trace formula}. J. Number Theory \textbf{130} (2010), no. 2, 386--430.

\bibitem{XJL1} Li, Xian-Jin, \emph{ Prolate spheroidal wave functions, Sonine spaces, and the Riemann zeta function},  J. Math. Anal. Appl.. 389(1), 2012, pp. 379--393.


\bibitem{Rokhlin} V.~Rokhlin, H~Xiao, \emph{Approximate formulae for certain prolate spheroidal wave functions valid for large values of both order and band-limit}. Appl. Comput. Harmon. Anal. \textbf{22} (2007), no. 1, 105--123.

\bibitem{Simon} B.~Simon, \emph{Trace ideals and their applications}. Second edition. Mathematical Surveys and Monographs, 120. American Mathematical Society, Providence, RI, 2005.

\bibitem{Slepian0} D.~Slepian, H.~Pollack, \emph{Prolate Spheroidal Wave Functions, Fourier Analysis and Uncertainty}, The Bell System technical Journal (1961), 43--63.  

\bibitem{Sl} D.~Slepian, \emph{Some asymptotic expansions forprolate spheroidal wave functions},  J. Math. Phys.  Vol. \textbf{44} (1965), 99--140.

\bibitem{Slepian} D.~Slepian, \emph{Some comments on Fourier analysis, uncertainty and modeling},  Siam Review.  Vol. 23 (1983), 379--393.

\bibitem{Sonin} N.~Sonin, \emph{Recherches sur les fonctions cylindriques et le d\'eveloppement des fonctions continues en s\'eries}. (French) Math. Ann. \textbf{16} (1880), no. 1, 1–80. 


\bibitem{tate} J.~Tate, \emph{Fourier analysis in number fields and Hecke's zeta-function}, Ph.D. Thesis, Princeton, 1950. Reprinted in J.W.S.~Cassels and A.~Fr\"olich (Eds.) ``Algebraic Number Theory", Academic Press, 1967.

\bibitem{Wang} L.~L.~Wang, \emph{Analysis of spectral approximations using prolate spheroidal wave functions}, Math. of Comp. Volume \textbf{79}, Number 270, April 2010, Pages 807--827

\bibitem{Wang1} L.~L.~Wang, \emph{A Review of Prolate Spheroidal Wave Functions from the Perspective of Spectral Methods}, J. Math.Study Vol. \textbf{50}, No. 2, pp. 101--143.

\bibitem{Weil} A.~Weil, \emph{Sur les ``formules explicites" de la th\'eorie des nombres premiers}, Meddelanden Fran  Lunds Univ. Mat. (d\'edi\'e \`a M. Riesz), (1952), 252--265; Oeuvres Scientifiques–Collected Papers, corrected 2nd printing, Springer- Verlag, New York-Berlin 1980, Vol. II, 48--61.


\bibitem{yoshida} H.~Yoshida,  \emph{On Hermitian forms attached to zeta functions}. Zeta functions in geometry (Tokyo, 1990), 281--325, Adv. Stud. Pure Math., \textbf{21}, Kinokuniya, Tokyo, 1992.






\end{thebibliography}
\end{document}